\documentclass{amsart}
\usepackage[all]{xy}
 \newtheorem{thm}{Theorem}[section]
 \newtheorem{cor}[thm]{Corollary}
 \newtheorem{lem}[thm]{Lemma}
 \newtheorem{prop}[thm]{Proposition}
 \theoremstyle{definition}
 \newtheorem{defn}[thm]{Definition}
 \theoremstyle{remark}
 \newtheorem{rem}[thm]{Remark}
 \newtheorem*{ex}{Example}
 \numberwithin{equation}{section}

\newcommand{\e}{\mathbf e}

\def\m{\mathfrak{m}}

\let\b\bullet

\def\H{\mathbb H}
\def\C{\mathbb C}

\def\Q{\mathbb Q}
\def\P{\mathbb P}
\def\R{{\mathbb R}}
\def\Z{\mathbb Z}
\def\N{\mathbb N}

\def\DD{\mathcal D}
\def\FF{\mathcal F}

\def\HH{\mathcal H}
\def\LL{\mathcal L}
\def\KK{\mathcal K}
\def\II{\mathcal I}
\def\MM{\mathcal M}
\def\NN{\mathcal N}
\def\OO{\mathcal O}

\def\PP{\mathcal P}
\def\RR{\mathcal R}
\def\SS{\mathcal S}

\def\WW{{\mathcal W}}

\def\n{\noindent}

\newcommand{\pt}{ \,^{p}\!\tau  }

\newcommand{\ilm}{  j_{!*}{\mathcal L}}
\newcommand{\ph}[2]{ \,^p\!{\mathcal H}^{#1}(#2)}
\newcommand{\lorw}{\longrightarrow}

\newcommand{\im}{ \hbox{\rm Im} }
\newcommand{\coke}{ \hbox{\rm Coker} }
\newcommand{\ke}{ \hbox{\rm Ker} }
\newcommand{\can}{ \hbox{\rm can} }
\title{Decomposition,  purity and fibrations
 by normal crossing divisors}
\begin{document}
\author{Fouad El Zein}
\address{Institut de Math\'ematiques de Jussieu.\\
Paris, France}
 \email{fouad.elzein@imj-prg.fr, elzein@free.fr}
\author{D\~ung Tr\'ang L\^e}
\address{Universit\'e d'Aix-Marseille 
LATP, UMR-CNRS 7353
39, rue Joliot-Curie
F-13453 Marseille Cedex 13, France}
\email{ledt@ictp.it}
\author{Xuanming Ye}
 \address{School of Mathematics and Information Science in Guangzhou University, Guangzhou, P.R.China }
\email{yexm3@gzhu.edu.cn}
 

\begin{abstract}
 We   give a simple geometric proof of  the decomposition  theorem  in terms of Thom-Whitney stratifications by reduction to fibrations by  normal crossings divisors over the strata and explain the relation with the local purity theorem an unpublished result of Deligne and Gabber. 
 \end{abstract}
\maketitle

\tableofcontents
  

 \section{Introduction}

  Let $f: X \to V$ be a projective morphism of complex algebraic varieties. The decomposition theorem, as proved in  \cite{BBD}, describes the derived direct image complex $Rf_* \ilm$ on $V$, of an  intermediate extension $ \ilm$
 of a  geometric local system $\LL$ on a Zariski open subset of $X$, as  a direct sum of intermediate extensions of irreducible local systems on  Zariski open subsets of $V$.
  
 The proof 
 in \cite{BBD} is given first for  varieties over an algebraically closed field with positive characteristic, with coefficients in  a pure perverse sheaf on $X$,  then   deduced for geometric coefficients on complex algebraic varieties. 
 
 It has been preceded by an unpublished  note  by Deligne and Gabber  \cite{DG} on the local purity theorem. The draft of the proof is dense and it is established  for an algebraic variety over a field with positive characteristic. 
 
  According to the general theory of weight this result can be translated into a local purity theorem  in Hodge theory in  the transcendental case. 
  
  There exist a deep connection between these two results which doesn't appear in the existing proofs in the theory.
  The basic innovation  in this article is a combined direct proof of both theorems with coefficients in the intermediate extension
 of a  polarized variation of  Hodge structures (VHS), which  establish    the interaction between  local purity   and the decomposition   and leads to an overall clarification of the theory  in the complex case. 
 
  Precise statements are given below in the subsection \ref{state}(Theorem \ref{P} and corollary \ref{CP}).
  
 The proof is by induction on the dimension of the strata of  a Thom-Whitney stratification. 
 At the  center  of the inductive step we use a geometric interpretation of the local purity to deduce the decomposition (see subsections  \ref{p}, 2.3 and 2.4). 
    That is a  simplification of the  original proof  in \cite{DG}, when adapted to characteristic $0$.   
  
 The second innovation is to introduce  a class of   projective morphisms which are  fibrations  by normal crossing divisors (NCD)  over the strata.    
     The proof for such fibrations use Hodge theory  only in the basic case of a complement of a NCD. 
 
  We show in section 3 that any algebraic morphism $f: X \to V$ admits a desingularization $\pi$ of $X$ such that $\pi$  and $f \circ \pi$ are in such class of fibrations, from which we deduce without further effort the result for all projective morphisms. 
  
We rely in this paper on  computations in terms of perverse sheaves and perverse cohomology as explained in  \cite{BBD}.
 The perverse $\tau$-filtration on a complex $K$ in  the  derived category $D^b_c(X, \Q)$ of  constructible $\Q$-sheaves is defined in (\cite{BBD} proposition 1.3.3 p.29). A recent description of this filtration  in \cite{DM} 
  shows easily that for $K = Rf_* \ilm$ such filtration  consists of sub-MHS on the cohomology of the complement of a  NCD (see subsection 5.3).  This  is a basic result used  in our proof and extended to the local case.
  
In the case of constant coefficients, the proof of decomposition by de Cataldo and Migliorini \cite{DM0} shows the potential of applications  in understanding the topology of algebraic maps.   The proof here  for  coefficients in any admissible polarized VMHS covers all analytic cases and has similar  applications. 

       Complex computations  by Kashiwara intended  to check the  admissibility of (graded) polarized variation of mixed Hodge structures (VMHS) on the complement of a NCD
in codimension one (\cite{K}, find their natural applications here to develop Hodge theory with  coefficients  in the intermediate extension of
an admissible polarized VMHS.
 
 Using the decomposition theorem, the  full Hodge theory can be developed from the special case of fibrations over the strata, which gives an alternative approach to  the development in terms of differential modules \cite{MI}. Our proof is outlined  in the notes \cite{ED1}  and  \cite{ED2}.

It should be noted that 
 holonomic differential Hodge modules on a smooth complex algebraic variety \cite{MI}, are associated with  polarized VHS on smooth open subsets of  irreducible algebraic subvarieties for which   the   decomposition theorem is proved here. 
\subsubsection{ Origin of the theory}
  The  first statement of a  decomposition theorem in the sense of this article 
 appears  in the work of Deligne \cite{DL} as a criteria for a complex $K$ in the   derived category of an abelian category  \cite{VDC} to be isomorphic to the direct sum of its shifted cohomology:
$ K \overset{\sim}{\longrightarrow}   \oplus_{i \in\Z}H^i (K)[-i]$, which means essentially the degeneration of the spectral sequence with respect to the truncation functor $\tau$. 

The criteria in Deligne's  paper  is applied to the complex direct image $K := Rf_* \C_X$ by a smooth projective morphism $f: X \to V$. The  spectral sequence with respect to the  perverse truncation filtration  $\tau$ coincides up to indices with Leray's spectral sequence and the degeneration  is deduced from Hodge and Lefschetz theorems. The decomposition of $K$ follows. 
 
 This case apply on a big strata on $V$ of a Thom - Whitney stratification of any projective morphism $f$ if $X$ is smooth. This is the starting point of our induction.  
  
 Deligne's effort to present the proofs in full generality is a prelude to the vast generalization of the theorem to any projective morphism  in \cite{BBD} on schemes $X$ of finite  type 
on an  algebraically closed field  $k$ of characteristic $p > 0$, for pure  perverse sheaves in  the  derived category $D^b_c(X, \Q_l)$ of  constructible $\Q_l$-sheaves for $l \not= p$ \cite{BBD}.

 The statements may be  transposed in characteristic $ 0$  in terms of  Hodge theory
according to a   dictionary established by  Deligne   \cite{HI}.

   A general procedure to deduce  results concerning  geometric statements in characteristic $0$ from corresponding statements in characteristic $p > 0$, is described in (\cite{BBD}, 6.2). This covers geometric variations of Hodge structures.

However, the next development  waited until the  introduction of the Intersection complex
   in \cite{G-M}.
 The intermediate extension
 of a   local system on a Zariski open subset is a version of such complex. The link with Kashiwara's work on differential modules and the need to extend the theory to algebraic varieties over fields of positive characteristic lead to the introduction of  perverse cohomology.

\subsubsection{ Hodge theory}  According to the referee, it was  necessary  to add a note \cite{EY} to explain how to deduce Hodge theory  in terms of  logarithmic complexes in the NCD case with coefficients in admissible variation of Hodge structures (subsection \ref{LC}, more in  sections $4$ and $5$). 

This allows us to  state the local purity theorem \ref{pure} and to deduce the decomposition in section $2$. 
  By  using logarithmic complexes we  avoid the elaborate  development of the theory of Hodge modules.

\subsubsection{   The paper is organized as follows} In the rest of this introduction we give the main statements of the article and describe the ideas based on a simultaneous proof of decomposition and local purity by induction along the strata. 
 
 Full proofs using  logarithmic complexes  in  section $2$ are based on the reduction to the case of a special fibration  in  section $3$, while Hodge theory on logarithmic complexes is reviewed in  sections $4-5$.

\subsection{Statements}\label{state}
{\it We give below the statement of the theorem} \ref{P}   in terms of Thom-Whitney stratifications and the perverse cohomology sheaves ${^p\HH}^{i}(K)$ of a complex $K$ with constructible  cohomology  sheaves on an algebraic variety $V$. The perverse cohomology refers to a cohomology theory with value in the abelian category of perverse sheaves \cite{BBD}.

Let  $f: X \to V$ be a projective morphism  of  
 complex algebraic varieties, $\tilde \LL$ 
an admissible  polarized  variation of Hodge structures (VHS) on a smooth open subset $ \Omega $ in $X$ of dimension $m$ and   $j: \Omega \to X$  the embedding.
We denote by $\LL := \tilde \LL [m]$  the complex of sheaves reduced to   $\tilde \LL$ in degree  $ - m $,  by $j_{!*} \LL$ the  intermediate extension of  $\LL$ 
 (\cite{BBD}, prop. 2.1.11 p. 60, 2.2.7 p. 69) and by $ K := R f_* \ilm $ the direct image
  in the derived  category of cohomologically constructible sheaves of $\Q$-vector spaces $D^b_c (V, \Q)$.    
   
 The decomposition consists of two results. First, each perverse cohomology sheaf  $ \ph{i}{Rf_*(j_{!*}\LL)}$ decomposes naturally as a direct sum of 
 intermediate extensions (formula \ref{De} below).  
 This decomposition  is naturally expressed here in terms of {\it intersection morphisms} (\ref{E} below) which appear naturally in the theory (\cite{BBD} formula 1.4.6.1), and which importance appears with   a Thom-Whitney stratification of $f$.
 
The second result is the relative  Hard  Lefschetz theorem. Let $ \eta \in R^2 f_* \Q_X $  denote the section   of cohomology classes defined by a relative  hyperplane section of $f$. The corresponding   morphism of sheaves $K \xrightarrow{\eta} K[2]$ induce  morphisms on the perverse cohomology of degree $2$.
  Then, the result states that the $i$-th iterated  morphism 
\begin
{equation}\label{HLe}
 {^p\HH}^{-i}(Rf_*j_{!*}\LL)
\xrightarrow{\eta^i}{^p\HH}^i(Rf_*j_{!*}\LL).
\end{equation} 
is an isomorphism of perverse cohomology sheaves,  from which 
 the degeneration of the perverse Leray spectral sequence  is classically deduced.
 
 One consequence is  the existence of a non canonical
 isomorphism in $D^b_c (V, \Q)$ of the direct image complex $ Rf_*j_{!*}\LL$ with its shifted perverse cohomology sheaves (formula \ref{De1} below). It is deduced from Hard Lefschetz following the same  pattern as in \cite{DL}. 
 
 \subsubsection
{Thom-Whitney stratification $\SS$ of  $f$} Let $f: X \to V$ be a projective morphism on an algebraic variety $V$ of dimension $n$. The  proof presented here is by induction on the dimension of the strata  of  the  Whitney stratification   of $V$ underlying a Thom-Whitney stratification $\SS$ of  $f$ (section \ref{TW} below).

 We say that the stratification $\SS$ is adapted to $\ilm$ when the cohomology of $\ilm$ is locally constant when restricted to the strata of $X$, and the cohomology of $Rf_*(j_{!*}\LL)$ is locally constant when restricted on the strata of $V$.  We consider   in the text  only stratifications adapted  to $\ilm$. 
 
 Let   $ V_j := \cup_{dim.S \leq j}S$ denote  the union of all strata of dimension $\leq j$. 
 We start with a smooth open subset $U$, union of all strata of dimension $n$ over which $f$ is a topological fibration with fibers stratified by the induced stratification of $X$, and we set $V_{n-1}:= V - U$  the complement variety. Successively, let $U_1$ be the union of all algebraic smooth connected strata  $S$ of dimension $n-1$ over which $f$ is a topological fibration with induced stratification adapted  to $\ilm$ and set  $U_1 := V_{n-1} - V_{n-2}$.
 We continue the construction  by induction on the dimension of the strata $S$ of $V_j-V_{j-1}$ over which $f$ is a topological fibration with induced stratification adapted  to $\ilm$.

  \subsubsection
{Intersection morphisms }\label{E}
 Let  $X_{S}:= f^{-1}(S)$ be the inverse image of a strata $S$ of dimension $l \leq n$,  $i_{X_{S}}: X_{S} \to X$, $f_S: X_{S} \to S$. The intersection morphism $I_{S}:  R i_{X_{S}}^{!}\ilm 
\rightarrow i_{X_{S}}^* \ilm $, is defined on  $ S$ by the  composition of the  morphism $   i_{X_{S}*}  R i_{X_{S}}^{!}\LL \rightarrow   \LL$ with the  restriction morphism $
 \LL \rightarrow  i_{X_{S}*} i_{X_{S}}^* \ilm $. It defines   local systems ${\LL}^i_S $ as image of cohomology sheaves:
 \begin{equation}\label{E1}
 {\LL}^i_S := \im \biggl\{R^{-l+i} {f_S}_*
 ( R i_{X_{S}}^{!}   \ilm)  \stackrel{I_{S}}{\rightarrow}
 R^{-l+i} {f_S}_* ( i_{X_{S}}^* \ilm)\biggr\} ,
\end{equation}
 For a Thom-Whitney stratification adapted to $\ilm$,  the  images ${\LL}^i_S$ of $ I_{S}$  are local  systems  on the various   strata. 
  These local systems are necessarily components of any eventual  decomposition as it follows from  the proof of proposition \ref{prop}.

 When  $\LL$ is a variation of HS  of weight $a$,  the  image ${\LL}^i_{S_l}$ of $ I_{S_l}$   on a strata $S_l$ of dimension $l$ is a VHS  of weight  $a+i-l$,  as image of a  variation of MHS
of  weight $w \geq a+i-l$  into a  variation of MHS
of  weight $w \leq a+i-l$ (corollary \ref{wi}).  
   \subsubsection
{The  inductive step}\label{decomp }
 The  proof  is by induction on the dimension of the strata  of  the  Whitney stratification  $\SS$ of $V$ underlying a Thom-Whitney stratification of  $f: X \to V$.   
 Let  $V$ be  a complex  algebraic variety of dimension $n$, and  $\SS$ the stratification of  $f$, $ V_j := \cup_{dim.S \leq j}S$ the union of all strata of dimension $\leq j$, $k_j: (V-V_j) \to V, \,  i_j: V_j \to V$.
 The immersion of a  strata of dimension $l$ in $V_l^*:= V_l-V_{l-1}$ is denoted by $ i_{S_l}: S_l \to V$.  
 
  \begin{thm}
 [Inductive step] \label{P}
  Let $K = R f_ *  \ilm$ where $\LL$ is a shifted  polarized variation of Hodge structures (VHS) of weight $a$,  and $\SS$  a Thom-Whitney stratification of $f$ adapted to $\ilm$.
 
 We suppose by induction, there is a   decomposition  of the  perverse cohomology  on  $(V - V_j) \xrightarrow {k_j} V $ for   all degrees  $i$
 \begin{equation}
 \ph{i}{ (R f_ *  \ilm)_{\vert V-V_j}} \overset{\sim}{\longrightarrow} 
 \oplus^{S_l \subset V_l^*}_{j < l \leq n } \,\, k_j^* {i_{S_l}}_{! *}{\LL}^i_{S_l}[l]
\end{equation}
  into a direct sum    of intermediate extensions     of polarized VHS ${\LL}^i_{S_l}$ (formula \ref{E1})  of weight $a+i-l$ over the  strata $S_l$  for all  $l > j$.
 Let
\begin{equation*}
\cdots \to  \ph{i}{(i_j)_* R i_j^!  K}
\stackrel{^p\!\alpha_i}{\rightarrow}\ph{i}{K}
     \stackrel{^p\!\rho_i}{\rightarrow} \ph{i}{R k_{j *} K_{\vert V- V_j}}
    \stackrel{^p\!\delta_i}{\rightarrow}\cdots
\end{equation*}
be the  long exact sequence of  perverse  cohomology  associated to the triangle

\n  $
 (i_j)_* R i_j^{!} K
   \stackrel{\alpha}{\rightarrow}   K
   \stackrel{\rho}{\rightarrow} R k_{j *} K_{\vert V- V_j}
    \stackrel{[1]}{\rightarrow }$.
  
 i) The restriction to $V -V_{j-1}$ of the short exact sequences  of perverse sheaves 
\begin{equation*}
0 \to  Im\,  {^p\!\alpha_i}  \to {^p\!\mathcal { H}}^i ( R f_ *  \ilm)  \stackrel{^p\!\rho_i}{\rightarrow}  Im \,{^p\!\rho_i} \to 0
\end{equation*}
 is split over $V -V_{j-1}$:
 \begin{equation}
\begin{split}
{^p\!\mathcal { H}}^i &( R f_ *  \ilm)_{\vert V- V_{j-1}} =  k_{j-1}^*\im\,  {^p\!\alpha_i} \oplus  k_{j-1}^*Im \,{^p\!\rho_i}\\
  k_{j-1}^*&\im\,  {^p\!\alpha_i} =  \ke \, {^p\!\rho_i} \overset{\sim}{\longrightarrow}  \oplus_{S_j \subset V_j^*}\,\,
k_{j-1}^*{i_{S_j}}_{! *} {\LL}^i_{S_j}  \\
 k_{j-1}^*Im \,{^p\!\rho_i} &=  k_{j-1}^*(k_j)_{! *} k_j^*\, \ph{i}{R f_ *  \ilm} \overset{\sim}{\longrightarrow} 
\oplus^{S_l \subset V_l^*}_{j < l \leq n } \,\, k_{j-1}^* {i_{S_l}}_{! *} {\LL}^i_{S_l}[l].
\end{split}
\end{equation}
ii) Hard Lefschetz: If we  suppose  by induction Lefschetz isomorphisms on the open subset  $V - V_j $, then $\eta^i$ extends into an isomorphism  over  $V-V_{j-1}$

\centerline{$
 \eta^i: {^p\!\HH}^{-i}(Rf_*j_{!*}\LL)_{\vert V - V_{j-1}}
\stackrel{\sim}{\rightarrow}{^p\!\HH}^i(Rf_*j_{!*}\LL)_{\vert V - V_{j-1}}$} 
 \end{thm} 
Thus we obtain the following statement of the decomposition theorem proved in \cite{BBD} in the geometric case  and in \cite{MI} for Hodge modules.
 \begin{cor}\label{CP} 
  Let $K = R f_ *  \ilm$ where $\LL$ is a shifted  polarized variation of Hodge structures (VHS) of weight $a$,  and $\SS$  a Thom-Whitney stratification of $f$ adapted to $\ilm$. 
  
  i) There exists   a   decomposition  of the  perverse cohomology  on  $V $ in  each  degree  
 \begin{equation}\label{De}
 \ph{i}{ (R f_ *  \ilm)} \overset{\sim}{\longrightarrow} 
 \oplus^{S_l \subset V_l^*}_{ l \leq n } \,\,  {i_{S_l}}_{! *}{\LL}^i_{S_l}[l]
\end{equation}
  into a direct sum    of intermediate extensions     of shifted polarized VHS ${\LL}^i_{S_l}$ (formula \ref{E1})  of weight $a+i-l$ over the  strata $S_l$  for all  $l  \leq n$. 
    Moreover, 
   \begin{equation}\label{compat}
\im\, \biggl(\ph {i}{R
{k_j}_! k_j^*K}  \rightarrow  \ph {i}{K }\biggr) = {k_j}_{!*}k_j^*\ph {i}{ K}
\end{equation}

On  a projective variety $V$, we have an orthogonal decomposition of a polarized HS of weight $a+i+j$ 

 \begin{equation}\label{De3}
Gr^{\pt}_i \H^{i+j}(X,   \ilm) \simeq \H^j(V,  \ph{i}{ (R f_ *  \ilm)}) \overset{\sim}{\longrightarrow} 
\oplus^{S_l \subset V_l^*}_{ l \leq n } \,\H^j(V,    {i_{S_l}}_{! *}{\LL}^i_{S_l}[l]).
\end{equation}
  
 ii) Hard Lefschetz: The iterated cup-product  $\eta^i$ is an isomorphism  over  $V$

\centerline{$ \eta^i: {^p\!\HH}^{-i}(Rf_*j_{!*}\LL)
\stackrel{\sim}{\rightarrow}{^p\!\HH}^i(Rf_*j_{!*}\LL)$} 

Moreover, the perverse cohomology $\ph{-i}{ K}$ is dual to $\ph{i}{ K}$, and the duality is compatible with the natural decomposition:
 the local systems $\LL^i_{S_l}$ are polarized variation of Hodge structure (VHS) on $S_l$, $\LL^{-i}_{S_l}$ is dual to $\LL^i_{S_l}$, and  each $
 \eta^i$ induces an isomorphism $ \LL^{-i}_{S_l}\stackrel{\sim}{\rightarrow} \LL^{i}_{S_l}$
  for each $i$ and $l$.
  
  iii) There exists  an isomorphism in the derived category of constructible sheaves \begin{equation} \label{De1}
 K  =   \oplus_{i \in \Z}{^p\mathcal { H}}^i
( K) [-i].
\end{equation}

\smallskip 
 \end{cor} 
 The duality follows from the auto-duality of $\ilm$  and Verdier duality for the projective morphism $f $. The polarization follows from the duality and Hard Lefschetz.
  
 We use  the local topological   triviality property of a morphism along a strata (subsection \ref{TW}), to reduce the proof to the case of a zero dimensional  isolated strata (proposition {\ref{prop}).
 
  Indeed, we cut 
 a strata $S$ of $V_j^*$ by a general  normal section $N_v$ to $S$ at a point $v$, so that the proof is carried on the restriction  to $N_v$ at $v$.

{\it Hard Lefschetz}.
 We check the proof of Hard  Lefschetz isomorphism (\ref{HLe}) in the subsection \ref{Hard}  on the terms $ {\LL}^i_S$ of the explicit formula  (\ref{E1}). 
The  decomposition   follows from  Hard Lefschetz   by Deligne's   criteria  in \cite{DL}; which completes the proof of the decomposition  (formula \ref{De} above).

The rest of  the article is devoted to the proof.
 \begin{rem}
i) The statement is local on $V$. However, we may suppose $V$ projective if necessary during the proof. In fact we may suppose $V$ affine and then embed $V$ into a projective variety $\overline {V}$, extend $f$ into $\overline {X} \xrightarrow{\overline {f}}\overline {V}$ and the intermediate extension onto $\overline {X}$.
  
  ii) In the special case of a fibration by NCD over the  strata (definition \ref{St}  below),  we  use logarithmic complexes to describe Hodge theory.  The  use of Hodge theory occurs in
  the proofs of the three lemma \ref{L3}, \ref{L1} and \ref{L2}, where the last two follow from  the local purity theorem stating conditions on the MHS on cohomology of  the link at $v$ with value in $Rf_* \ilm$ (theorem \ref {pure} and proposition \ref{Sp} below).

The proof for any projective  morphism follows from this special fibration case.
\end{rem}

  \subsection
{Fibration over the  strata by normal crossing divisors} 
 Hodge theory appears in the proof  via a local purity statement (theorem \ref{pure} below).  We reduce the proof to the case where $f$ is a special fibration, hence we only need  to develop Hodge theory on the complement of a NCD (sections $4-5$). 
 
   Let $\pi: \tilde X \to X$ be a desingularization such that $\pi $ and $f\circ \pi$ are special fibrations. The decomposition theorem for $f$  follows   from the case of  $\pi$ and $f\circ \pi$ (\cite{DL} proposition 2.16),  in particular we can suppose $X$ smooth.
 In fact, the reduction to the case where  $f: X \to V$ is a fibration by NCD over the strata of an adequate Thom-Whitney stratification of $f$ in the  sense below  is possible (\cite{ED1} D\'efinition 2.1):
 
 \begin{defn}
 [topological fibration  over the strata by NCD]\label{St}
 
 \
 
i)  A morphism $f: X \to V$ is a topological fibration  by NCD over the strata of a    stratification ${\SS} = (S_\alpha)$ of  $V$ underlying a Thom-Whitney stratification of $f$, if  $X$ is smooth and the spaces
$V_l=\cup_{\dim S_\alpha \leq l} S_\alpha$ satisfy the following  properties:

1)  The  sub-spaces  $ X_{V_i} := f^{-1} ( V_i)$ are successive  sub-NCD embedded in $X$.

2)  The restriction of $f$ to $X_S := f^{-1} (S)$ over each  strata $ S$  of ${\mathcal S}$ is a topological fibration: 
 $f_{\vert}: X_{S} \to S$.

 ii) The fibration  is adapted to a  NCD  $Y$ in $X$, or to a    local system $\LL$ defined on an open algebraic set in the  complement of
 $Y$ in $X$, if in addition,  for each $1\leq i\leq n$,
  the union  of the sub-spaces  
 $ X_{V_i} \cup Y$ are relative NCD over the strata of $V$, and the intermediate extension of $\LL$ is constructible with respect to  the strata.
  \end{defn}
  
\subsubsection{Fibration by NCD over the strata}

 We explain the terminology.  Let $X$ be smooth, and $v \in  V_i - V_{i-1}$ (hence $V_i$ is smooth at $v$); the inverse image of  a normal section  $N_v$ to $V_i$ at $v$ in general position, is a smooth subvariety $f^{-1}(N_v)$ of $X$ and intersects the NCD   $X_{V_i} $ transversally, then
  $X_{V_i} \cap f^{-1}(N_v) =  f^{-1} (v) $  is a NCD  in  $f^{-1} (N_v)$.
 
 We say for simplification, that $X_{{V_i} - V_{i-1}}:= f^{-1} ( {V_i}- V_{i-1})$ is a relative NCD (eventually empty), and that  the stratification  $ \SS $ is admissible for $f$. 

Due to the next proposition, proved in  section  $3$, it is enough  to prove the decomposition
for a fibration by NCD over the strata.  
\begin{prop}
 \label{RSt}
Let $f: X \to V $ be a  projective morphism and $Y$ a closed  algebraic strict sub-space  containing the singularities of $X$. There  exist a diagramm 
$X  \stackrel{\pi'}{\leftarrow} X' \stackrel{f'}{\rightarrow}  V$ where  $X'$  is a non-singular variety, and stratifications of $V, \, X$, and $X'$ such that   $\pi'$ and  $f':= f \circ \pi'$ are fibrations by relative NCD  over the strata, adapted to  $Y':= \pi'^{-1}(Y)$. 

 Moreover,  $\pi'$ is a modification  of $X$, and there exist an open subset  
 $U \subset f(X) $ dense in the image $f(X) $ of $f$ such that $f$ is smooth over $U$,  $\pi'$ 
 induce an  isomorphism of 
 $  f'^{-1}(U) - (f'^{-1}(U) \cap Y') \overset{\sim}{\longrightarrow} f^{-1}(U) - ( f^{-1}(U) \cap Y)  $, and $f'^{-1}(U) \cap Y'$ is a strict relative NCD in the smooth fibers of $f$ eventually empty (called  horizontal or strict in the fibers).
  \end{prop}
  \subsubsection{
 Reduction to the case of zero dimensional strata with NCD fiber}
 The notion of  relative NCD  is used to reduce the proof at  a general point $v$ of a strata  $S$, to the case of a point $v$ in a transversal section $N_v$ to $S$ at $v$, hence to a point $v$ in a zero dimensional strata. 
 In such situation we can use logarithmic complexes (\cite{C} section 8.3.3, theorem 8.3.14) since the inverse image of $v$ is a NCD.
\subsubsection
{Purity of Intersection cohomology on a singular variety (lemma \ref{polar})}
We illustrate  the results of theorem \ref{P} and corollary \ref{CP} in the case of a polarized variation of HS  $\LL$ on a smooth open subset  $U$ of a singlular projective variety $X$ with intermediate extension $j_{!*}\LL$.
  The HS  on  $H^k(X, \ilm)$ is deduced from the compatibility 
  of the perverse filtration with the HS on the intersection cohomology of a desingularization of $X$.
  
  We consider an adequate desingularization diagram $X'  \overset{\pi}{\longrightarrow} X $ of $X$ and the intermediate extension $j'_{!*}\LL$ on $X'$ of $\LL$ on $U \subset X'$. The proof is based on  the fact that the perverse filtration defined by $\pi$ is compatible with  HS defined on a projective $X'$.
  
   The polarized VHS  on the components $\LL^i_S$ of the decomposition  on $X$  are defined  on the various strata $S$   via   the Intersection  formula $\ref{E1}$ and Hodge theory on $X'$. 
   
   We deduce from the compatibility with  $\pt$  the HS on  
$H^k(X,  \ph{0}{R\pi_* j'_{!*}\LL}) \simeq Gr^{\pt}_0 H^k(X',   j'_{!*}\LL)$ for $X$ projective. 
The decomposition 
  $H^k(X,  \oplus_{S_l \in \SS} \LL^0_{S_l}[l]) $ is compatible with the HS on $Gr^{\pt}_0 H^k (X,   j'_{!*}\LL)$ which follows from the formula \ref{De3}
  
\subsubsection{Variation of Hodge structure on $\LL^i_{S}$
    over a strata $S\subset V$ (corollary \ref{Vpolar})}
    To illustrate  the construction of  the variation of HS on $\LL^i_{S}$, we consider an adequate diagram $X'  \overset{\pi}{\longrightarrow} X  \overset{f}{\longrightarrow}V$ with a desingularization $X'$ of $X$ and the intermediate extension $j'_{!*}\LL$ on $X'$,  then $\ilm$   is a component of  the decomposition of $\ph{0}{R\pi_* j'_{!*}\LL}$ on $X$. 
  
   At  a point $v$ in the zero dimensional strata of $f$, 
$\H_{X_v}^i(X, \ilm) $ (resp. $\H^i(X_v, \ilm) $) is a summand of $Gr^{\pt}_0 \H_{X'_v}^i(X', j'_{!*}\LL) $ (resp. $Gr^{\pt}_0 \H^i(X'_v, j'_{!*}\LL) $ ) with induced sub-mixed Hodge  structure   of weight $\geq a+i $ (resp.   $\leq a+i $), hence the image HS on $\LL^i_v$ is pure of weight $ a+i$ .  The polarization is the result of Poincar\'e duality and Hard Lefschetz.
At  a point $v$ in a strata $S$ of $V$,  we are reduced to the zero dimensional case for the induced stratification  on  a normal section to $S$ at $v$. 

\subsection
{Logarithmic complexes}\label{LC}
We give in  section  $4$ an account of Hodge theory with coeficients in an admissible variation of mixed  Hodge  structure   $(\LL, W, F)$  
on the  complement of a normal crossing divisor  $Y$  in a smooth  complex projective variety $X$, using   logarithmic complexes.

Moreover, we enlarge the theory to logarithmic complexes asssociated  with a NC subdivisor $Z \subset Y$, which is necessary in the inductive step (subsection \ref{Log Z}). 

{\it Let $\widetilde \LL$  be a variation of mixed  Hodge  structure, $\LL := \widetilde \LL[m]$ the associated  perverse sheaf on $X-Y$ with the conventional shifted degree}, 
   $j := (X- Y) \to X$, $j_Z := (X-Z) \to X$, and  $i_Z: Z \to X$.
   
We  describe logarithmic complexes with weight  and Hodge filtrations,  obtained for   various functors  applied to the intermediate extension of the  admissible variation of MHS $\LL$ on $X-Y$ as
$ j_{Z !} (j_{!*}\LL)_{\vert X-Z}, Rj_{Z *} (j_{!*}\LL)_{\vert X-Z}$,
$ i_Z^*Rj_{Z *} (j_{!*}\LL)_{\vert X-Z}, i_Z^!j_{!*}\LL$ and $ i_Z^*j_{!*}\LL$,
 defining thus a MHS on the corresponding hypercohomology groups.
 
 \n ($\star$) {\it We refer to such complexes as bifiltered logarithmic complexes}.

\n The starting point is   the  realization of the direct image $ Rj_{Z *}( j_{!*}\LL)_{\vert X-Z} $
as a sub-complex   $IC^*\LL(Log Z) $ of the logarithmic  complex   $ \Omega^* \LL := \Omega^*_X(Log Y)\otimes \LL_X$  (denoted as   $\Omega^*( \LL, Z)$  in \cite{C} definition (8.3.31)).  
We refer back to a local study by Kashiwara \cite{K} to construct the weight filtration with a property of local decomposition.

 We refer to the last sections $4$ and $5$ for the details of the constructions. Indeed, the details are not needed earlier. Only the existence  of the  weight and the compatibility with the perverse filtrations are needed  in the earlier sections, as well  the following general properties:

\n 1) the statement and the proof of local purity at the end of section $2$.

\n 2) the condition $w \leq a+i$ on the  weight  of $\H^i(Z ,\ilm)$ for a pure $\LL$ of weight $a$.

\n 3) the condition  $w \geq a+i$ on the  weight  of $\H_Z^i(X,\ilm)$ for a pure $\LL$ of weight $a$.
\subsubsection
{Compatibility of Hodge structure with   perverse filtration}\label{compa}

 Let $f: X \to V$ be a  morphism, 
   and $K$ a complex on $V$ with constructible cohomology. The
     topological middle
perversity truncations on $ K  $ on $V$ (\cite{BBD} section $2$ and prop. 2.1.17)
define an  increasing {\em  perverse  filtration} $\pt$ on $ K$, from which we deduce for each    closed sub-variety $W$ of $V$,  an  increasing  filtration $\pt$   on the hypercohomology:
\begin{equation}
    \pt_i {\mathbb H}^k(V-W, K) :=\, \im
    { \,\left\{ {\H}^k(V - W,\pt_i K) \to {\H}^k(V- W, K)
      \right\}}.
\end{equation} 
Let $ Z := f^{-1} W := X_W  \subset X$,  $j_Z: X-Z \to X$ and  $K:= Rf_* R(j_Z)_* j_Z^* j_{!*}\LL$. The  {\em   perverse filtration} $\pt$   on the hypercohomology  groups ${\H}^k(X-Z,  j_{!*}\LL)$   is deduced by the isomorhism ${\mathbb H}^k(X-Z,  j_{!*} \LL) \, ={\mathbb H}^k(V-W, Rf_*  j_{!*}\LL)$
\begin{equation}
    \pt_i {\H}^k(X-Z,  j_{!*} \LL) \,:= \pt_i \H^k(V-W, Rf_*  j_{!*}\LL)
    \end{equation} 
 Similarly,  we define functorially $ \pt_i$ on $ \H_Z^k(X,  j_{!*} \LL), \H_c^k(X-Z,  j_{!*} \LL)$ and 
 $\H^k(Z,  j_{!*} \LL) $, after $\pt$ on $ K$.   
 
When $f$ is projective, we check that the  filtration ${\pt_i}$  is a  filtration  by sub-MHS on the hypercohomology of $X-Z$ (\cite{EY}, theorems 1.1 and 3.2) using a result of \cite{DM}. We say  simply: the perverse filtration ${\pt}$ is compatible with the MHS on the hypercohomology of  $X-Z$.
 \begin
 {prop}\label{global}
 Let $f: X \to V $ be a projective
morphism  on $X$  smooth, $j: X-Y \to X$ the open
embedding of the complement of  a normal crossing divisor $Y$ in
$X$, $\mathcal {L}$ an admissible variation of $MHS$ on  $X - Y$ shifted by  dim.$X$, $W \subset V$ a
closed algebraic subset of $V$ 
 s.t. $X_W := f^{-1}(W)$ is a sub-normal crossing divisor
  of $Y$. Then, the perverse filtration $\pt$ on  the cohomology $\H^* ( X -
X_W,\ilm)$ (resp. $\H_{X_W}^* ( X ,\ilm)$ and  by duality $ \H_c^* ( X -
X_W,\ilm)$, $\H^* ( X_W ,\ilm)$),
is a filtration by sub-$MHS$.
\end{prop}
The proof   is carried in terms of logarithmic complexes in subsection (\ref{pglobal}). 
\subsubsection
{Mixed Hodge structure (MHS) on the Link}

 The basic criteria for a complex $K$ on $V$ to underly a structure of mixed Hodge complex (MHC),   is given locally at a point  $v \in V $ in terms of 
the Link at $v$ (proposition \ref{mainp} below).  The criteria is translated into a property of the tubular neighborhood of $X_v := f^{-1}(v)$,  in terms of  $X^*_{B_v}:= X_{B_v} -X_v$, the inverse image of a small  ball $B_v$ with  center $v$  in $V$ minus  the central fiber $X_v$. 

Let $j_{X_v }: (X - X_v) \to X, \,  i_{X_v } : X_v  \to X,  \, k_v: (V-v) \to V$ and $i_v: v \to V$. 
 We apply the functor $ i_v^* R (k_v)_*k_v^*$ functorially to  the filtration ${\pt}$   on $Rf_* j_{!*}\LL$ to define the perverse  filtration ${\pt_i}$ on   $R\Gamma(X_v , i_{X_v }^*R(j_{X_v })_* j_{X_v }^* j_{!*}\LL )$  via   the isomorphism $i_{X_v }^*R(j_{X_v })_* j_{X_v }^* j_{!*}\LL ) \overset{\sim}{\longrightarrow}  i_v^* R (k_v)_*k_v^*Rf_* j_{!*}\LL$.

The perverse  filtration is considered on the hypercohomology $ \H^r(X^*_{B_v}, j_{!*}\LL) $, which 
 coincide with $R\Gamma(X_v , i_{X_v }^*R(j_{X_v })_* j_{X_v }^* j_{!*}\LL)$ if the radius of $B_v$  is small enough. Thus, it may look  a less abstract  object.

 \begin{prop}[MHS on the Link]\label{mainp}
 
   Let  $v$ be a  point in $ V$ with fibre  $X_v := f^{-1}(v) $. We suppose   $ X_v $ and $ X_v  \cup Y$
 are  NCD in $X$ and $\LL$ defined on $X-Y$; then, for a ball $B_v$ with center $v$ small enough,  the  perverse filtration
 on $\H^r(X^*_{B_v}, j_{!*}\LL) \overset{\sim}{\longrightarrow}  \H^r ((X_v , i_{X_v }^*R(j_{X_v })_* j_{X_v }^* j_{!*}\LL)$   is   compatible with  the  MHS  (that is a filtration sub-MHS).  
\end{prop}

 \n The proof in  subsection  \ref{Link1}  depends on a    local version  of a result in \cite{DM}
 easy to check.
 Once the above structure is defined, we can  express  the notion of local purity in positive characteristic  and give a meaning to the purity theorem \cite{DG}.
\subsection
{Deligne-Gabber's local purity}\label{p}
This  basic result is stated in positive characteristic in \cite{DG} as follows:\\
{\it Let  ${K}$ be a pure complex of
weight $a$ on  an algebraic variety $V$ and $B_v$ a henselization
of $V$ at  a point $v$ in the zero dimensional strata of $V$. The
weight $w$ of the cohomology
 $ \H^i( B_v- \{v \}, K )$,
    satisfy:
\begin{equation*}
w \leq a+i  \; \hbox {if }\;i \leq -1\; \hbox {and }\; w > a+i
\; \hbox {if }\;i \geq 0.
\end{equation*}}
According to Deligne's dictionary between the purity in positive
caracteristic and Hodge theory,   a pure complex corresponds in
the transcendental case to an intermediate extension $k_{!*}\LL $ of a
polarized VHS $\tilde \LL$ on a smooth algebraic  open subset $k:
V^*\hookrightarrow V$ embedded in a
  complex algebraic variety $V$.

To state the  above result,
 the first task is  to construct a $MHS$ on the hypercohomology  $\H^*(B_v - v,
k_{!*}\LL )$ where $ B_v$ is a small ball with center $v \in
V$, with coefficients in an intermediate extension  of $\LL :=
\tilde \LL [n]$ where $n$ is the dimension of $V$, that is the VHS
shifted as a complex  by $n$ to the left, or more intrinsically on $ H^*(i_v^* R (k_v)_*k_v^*k_{!*}\LL )$.
Once we have such $MHS$
we can state the conditions   on the weight  as follows:
\begin{thm}
 [local purity]\label{pure}
Let  ${\tilde \LL}$ be a polarized variation of Hodge structures
 of weight $b$ on the smooth $n$-dimensional open subset
$V^{*}$ of an algebraic variety $ V$, $k: V^* \to V$  the
embedding, and let $v$  be a point in the zero dimensional strata
of $V$, $B_v$ a small ball with center $v$ and $a = b + n$. 

The weight $w$ of the space $ H^i(i_v^* R (k_v)_*k_v^*k_{!*}\LL )$,   isomorphic to the Intersection cohomology $ \H^i( B_v- \{v \}, k_{!*}\LL )$
 of the Link at $v$,  satisfy the relations:
\begin{equation*}
w \leq a+i  \; \hbox {if }\;i \leq -1\; \hbox {and }\; w > a+i
\; \hbox {if }\;i \geq 0
\end{equation*}
(the intermediate extension $k_{!*}\LL  $,  after the shift $\LL := \tilde \LL [n]$,  is a complex of weight $a = b + n$).
\end{thm}
\subsubsection{Equivalent statement}  
  The $MHS$  is defined  in fact along the NCD $X_v$ in $X$, so we prove  equivalent conditions on the weight above on $X$ as follows
\begin{prop}
[Semi purity]\label{Sp}
 Given a (shifted) polarized variation of Hodge structures $\ilm $ on $X$ of weight $a$, the weights of the mixed  Hodge structure on  the graded-cohomology spaces
\begin{equation*}
Gr_i^{\pt} \H^r(B_{X_v}-X_v,\ilm)
 \end{equation*}
 satisfy the inequalities:
$w \leq a+ r$ if $r-i \leq -1$, $w > a+ r$ if $r-i \geq
0$.
\end{prop}
The crucial case  of the proof is treated in section \ref{cru}.  The original proof in positive characteristic  may be adapted to the transcendental case, but we give a relatively simple new proof based on the inductive hypothesis which assumes Hard Lefschetz theorem on $B_v -v$. In fact,
 both the proposition  and Hard Lefschetz theorem (\ref{HLe}) are proved
 simultaneously by induction on the decreasing dimension of the strata
 of $V$. 
 
 Let $\SS$ be a stratification of $V$, and  $V_j:= \cup_{S \in \SS: \hbox { dim} \, S \leq j} S$; if  we assume by an inductive argument the decomposition theorem established on $ V- V_j$,  then we deduce first the local purity theorem at a point $v$ of a strata $S \subset V_j$  of dimension $j$.
Let $N_v$ be a normal section to $S$ at $v$ in general position, the subvariety $f^{-1}(N_v)$ is smooth in $X$ (we can suppose $X$ smooth) and intersects the NCD   $X_S $ transversally. On $N_v$ the  point $v$ is  isolated  in the lowest strata of dimension zero; hence, we reduce the proof of   local purity to  the case of an  isolated point.
 On its turn, the semi-purity at $v$ is used to extend the proof of the decomposition from  $V-V_j$ to $ (V-V_j) \cup S$, as well Hard Lefschetz theorem.

  Thus, we obtain the decomposition on $V-V_{j-1}$, from which we repeat the argument of local purity at points of $V_{j-1}$.
Finally, we obtain by induction, a  simultaneous proof   of the local purity and the  decomposition at the same time. 

\section{Proof of the decomposition theorem}

 {\it 
 We present here a    proof of theorem \ref{P} and corollary \ref{CP}  for  projective morphisms, based  on a reduction to a stratification by  NCD over the strata (definition \ref{St}, proposition \ref{RSt}, section \ref{St1}). Three lemma (\ref{L3}, \ref{L1}, \ref{L2}) below use Hodge theory for which we refer to sections $4 - 5$}.
  
\subsection{Proof of theorem \ref{P}}\label{deco}   Let $f: X \to V$ be a projective morphism of complex algebraic varieties,    $\SS$  a Thom-Whitney stratification   adapted to $\ilm $ on $X$ and $K:= Rf_*\ilm$.  Set $V_k$  the union  of all strata of $V$ of dimension $ \leq k$.
  
 \subsubsection
{The big strata} By construction, the restriction of $f$ to any strata of $V$ is a topological fibration.  The union of the big  strata of  $V$  form an open smooth set $U$ of dimension  $n:= dim. V$  as  we   suppose  the   singularities of  $V$ contained in  $V_{n-1}$. The restriction to $U$   of the local cohomology sheaves ${\HH}^{i}(K)$    are locally constant,   hence they coincide with the  perverse  cohomology sheaves: $\ph{i}{K} = {\HH}^{i}(K)$.
\begin{lem}[Initial step] \label{L3}
The relative Hard Lefschetz theorem apply over the big strata $U$ of $V$. 
\end{lem}
The proof follows from Hodge theory in section \ref{rl}, where we suppose  $\LL $ defined on the complement of $Y_U$  an horizontal relative  NCD in $X_U$. Then, $R^if_* \ilm$ underlies a VHS such that at each point $v \in U$,  the fiber $(R^if_* \ilm)_v \simeq \H^i ( X_v, \ilm)$  is isomorphic to the intersection cohomology of the restriction of $\LL$ to $U_v$ the fiber of $U$ at $v$ with its HS. The family $R^if_* \ilm$  for various degrees satisfy the relative Hard Lefschetz isomorphisms.
Hence,  the decomposition theorem  for the smooth induced morphism  $f_U: X_U \to U$ follows  from  the  results of Deligne \cite{DL}. 

\subsubsection
{The inductive step}\label{induction}  
   The  decomposition theorem is proved by descending induction on the dimension of the strata. We suppose the {\it decomposition} proved   over  the open subset $U_j:= V-V_j$ for some $j < n$, then we extend 
 the decomposition to  $U_{j-1}$ along $V_j - V_{j-1}$, that is across the union of  smooth strata $S$ of dimension $ j$. Let $N_v$ denote a general normal section to the strata  $S$ at a point $v \in S$, and $f_{N_v}: X_{N_v} \to N_v$ the restriction of $f$.
 
  By construction  of $\SS$, $f$ is locally trivial along $S$: let  $B_v$ be a small ball on $S$ with center   $v$, then  $f$ is locally homeomorphic to $f_{N_v}\times Id_{\vert B_v}: X_{N_v}\times B_v \to N_v \times B_v $ (subsection \ref{TW}) , and   the proof  may be  reduced to the case of $f_{N_v}: X_{N_v} \to N_v$ with the induced  stratification on $N_v$; that is  the  case  of an isolated point $ v \in N_v$ with $j=0$.
\subsubsection{The case of  zero  dimensional strata $V_0$}
  We prove now the case of  the zero  dimensional strata in   theorem  \ref{P}. Since the proof is local, we may suppose $V_0 = v $ reduced to one point $v$. Let 
 $ k_{v} : (V- v) \to V$ and $ i_{v} : v \to V$, $\LL$ of weight $a$ on a Zariski open subset of $X$: 
  \begin {prop}
  \label{prop}
  Let
   $K = {R f}_{ *}  \ilm$ on an algebraic variety $V$ of dim.$n$,  $\SS$ a  Whitney stratification of $V$ underlying a Thom  stratification of $f$, $v \in V_0 $ a point in the strata of dimension $0$, and suppose the following conditions on $V-v$ satisfied:
   
 \smallskip 
   1)  There exists  for each $i \in \Z $,
a  decomposition over the open subset $ V -v$  into a
direct sum  of intermediate extensions of  shifted local systems $\LL^i_{S_l}[l]$ (formula \ref{E1}) 
on the various strata $S_l \subset V-v$ of dimension $l \leq n$: 

 \centerline{$ \ph{i}{ K}_{\vert V - v} :=  k_{v}^* \ph{i}{ K}_{\vert V - v} \overset{\sim}{\longrightarrow} 
 \oplus_{S_l \subset V-v } \;\;  k_{v}^* {i_{S_l}}_{! *}\LL^i_{S_l}[l]$}
 
 2) Hard Lefschetz: $  {^p\!\HH}^{-i}(K)_{\vert V-v}
\stackrel{\eta^i}{\rightarrow}{^p\!\HH}^i(K)_{\vert V-v}$ is an isomorphism on $V-v$.
 
 3) The local purity  theorem   apply to $K$ at $v$ (theorem \ref{pure},  subsection \ref{pure1}).
 
 \medskip 
 \n Then    the decomposition as well Hard Lefschetz extend over $v$.
 
 Precisely, let $\rho: K \to {R k_v}_* K_{\vert V-v}$ and $ \alpha: {i_{v,*} R{i}_v^{!} K}
 \to K$,
 and consider  the long exact sequence
of perverse cohomology on $V$
 \begin{equation}\label{elf}
\stackrel{^p\!\delta^{i-1}}{\rightarrow}\ph{i}{i_{v,*} R{i}_v^{!} K}
\stackrel{^p\!\alpha^i}{\rightarrow}\ph{i}{K}
     \stackrel{^p\!\rho^i}{\rightarrow}\ph{i}{{R k_v}_* K_{\vert V-v}}
    \stackrel{^p\!\delta^i}{\rightarrow}
\end{equation}
 
 \n i) The perverse image of
${^p\!\rho^i}$ is isomorphic to: 
\begin{equation*}
Im \,{^p\!\rho^i}\overset{\sim}{\longrightarrow}  {k_v}_{! *} k_v^*\, \ph{i}{ K} \overset{\sim}{\longrightarrow} 
\oplus_{S_l }{i_{S_l}}_{! *} {\LL}^i_{S_l}[l].
\end{equation*}
ii) We have $\ph{i}{i_{v,*} R{i}_v^{!} K} \overset{\sim}{\longrightarrow} 
i_{v,*}H^i ( i_v^! K) $. Let ${\LL}^i_v:= \im(H^i ( i_v^! K) \to H^i ( i_v^* K))$, then $\ke \, {^p\!\rho^i} = \im \,{^p\!\alpha^i}\simeq i_{v,*}{\LL}^i_v$ is a polarized Hodge structure of weight $a +i$.

\smallskip
\n iii) 
The perverse cohomology
decomposes on $V$  
\begin{equation}\label{elf1}
 \ph{i}{K} \overset{\sim}{\longrightarrow}  \ke\,
{^p\!\rho^i} \oplus Im \,{^p\!\rho^i}. 
\end{equation}
In particular, we can extract from the sequence (\ref{elf}) an exact sub-sequence 
\begin{equation} 0 \to
  \oplus_{S_l \subset V-v, i-1-j \geq 0} R^{i-1-j} k_{v *}
 k_v^*({i_{S_l}}_{!*} \LL^j [l] \stackrel{^p\!\delta^{i-1}}{\rightarrow}\H^i_v(X, K)\stackrel{^p\!\alpha^i}{\rightarrow}i_{v,*}{\LL}^i_v \to 0
\end{equation}
iv)  Hard Lefschetz:  $  {^p\!\HH}^{-i}(K)
\stackrel{\eta^i}{\rightarrow}{^p\!\HH}^i(K)$ is an isomorphism on $V$.  
\end{prop}
The proof  occupy the rest of subsections \ref{deco} and \ref{Hard}.

 In fact the third condition on local purity follows  from the first condition on the decomposition on $V-v$ (see subsection \ref{pure1});  its proof is just postponed to ease the exposition. The statement is given in full generality but the proof is given first for fibrations by NCD from which case it is deduced in general.
\subsubsection{Perverse cohomology of ${R k_{v}}_* K_{\vert
V-v}$ and ${R k_{v}}_! K_{\vert V-v}$}
In order to compute   successively: $Im \, {^p\!\rho_i}$,  $ Im\,  {^p\!\alpha_i}$ and to prove the splitting of  $\ph{i}{ K}$, 
we rely on the next calculus of  the  perverse cohomology of $ R k_{v *} (K_{\vert V-\{v\}})$ under the hypothesis  of the decomposition  on  $V- v$. 

\smallskip 
\begin{lem} \label{lem}
 Let $i: S \to V$ be a fixed strata of $V$ of dim.$l$, $\LL^j
[l] $ a family  of  local systems on  $S$ shifted by $l$ for  $j\in \Z$,  and  $K' := \oplus_j i_{ !*}(\LL^j[l])[-j]$ the direct sum, hence $\ph{i}{ K'} =  i_{ !*}\LL^i [l]$.
Let $v$  be a point in the closure of $S$ and
 $k_v: V-v \to V$. 

\smallskip 
i)    $\ph{i}{R k_{v *}
k_v^*i_{ !*} \LL^j[l]} =   R^{i} k_{v *}
(k_v^*i_{ !*}{\LL}^j [l])$ for $i > 0$, vanish for $i < 0$, and we have a short exact sequence of perverse sheaves for $i = 0$
 \begin{equation*}
 0 \to
i_{ !*}(\LL^j[l]) \to \ph{0}{R
k_{v *}k_v^*  i_{ !*}(\LL^j[l])}\stackrel{h}{\rightarrow}  R^0 k_{v *}
(k_v^*i_{ !*}\LL^j [l]) \to 0 
 \end{equation*}
 in particular ${\underline H}^0(  \ph{0}{R
k_{v *}k_v^* i_{ !*}\LL^j [l])})\overset{\sim}{\longrightarrow}  R^0 k_{v *}
(k_v^*i_{ !*}\LL^j [l])$, and 
a morphism  of   perverse sheaves   $\varphi:  \PP \to \ph{0}{R
k_{v *}k_v^*  i_{ !*}
\LL^j [l]}$
   factors through  $i_{ !*}
{\LL}^j [l]$ if  and only if $h \circ \varphi = 0$. 

\medskip
 ii) We deduce
 $\ph{i}{R k_{v *} k_v^*K'} = \oplus_{j\leq i} \ph{i-j}{R k_{v *}
k_v^*i_{ !*} \LL^j[l]}  =  \ph{i}{R k_{v *}
k_v^*  {^p\!\tau}_{\leq i}K'} $

\n  and  short exact sequences of perverse sheaves for all $i$
 \begin{equation*}
  0 \to
i_{ !*}({\LL}^{i}[l]) \to \ph{i}{R
k_{v *}k_v^*  K'} \stackrel{h}{\rightarrow} \oplus_{j\leq i} R^{i-j} k_{v *}
(k_v^*i_{ !*}{\LL}^j [l]) \to 0
 \end{equation*}
 where the last term  is a sum of vector spaces supported on $v$ in degee zero.
 
\smallskip 
 There are isomorphisms of sheaf cohomology  concentrated on $v$ in degree $0$
 \begin{equation*}\label{miss}
 {\underline H}^0(  \ph{i}{R k_{v *}k_v^* K')}\overset{\sim}{\longrightarrow}   \oplus_{j \leq i} R^i k_{v *} ( k_v^*i_{ !*}{\LL}^j [l-j])  
\end{equation*}
  and 
a morphism  of   perverse sheaves   $\varphi:  \PP \to\ph{i}{R k_{v *}k_v^* K'} $ factors through  $i_{ !*}
({\LL}^i [l])$ if  and only if $h \circ \varphi = 0$. 

\smallskip 
iii)  Dually,   $\ph{i}{R k_{v !} k_v^*K'} =  \oplus_{j\geq i} \ph{i-j}{R k_{v !}
k_v^*i_{ !*}\LL^j[l]}$ and  we have short exact sequences
\begin{equation*}
 0 \to  \oplus_{j\geq i} i_{v *} H^{i-j-1}( i_v^*i_{ !*}{\LL}^j_{S_l}[l] )  \stackrel{h'}{\rightarrow} \ph{i}{R
k_{v !}k_v^* K'} \to  i_{ !*}{\LL}^i [l] \to 0
 \end{equation*}
where the first term is 
 supported on $v$ in degee zero.

A  morphism $\varphi$ defined on  $\ph{i}{R k_{v !}k_v^* K' }$  factors through  $i_{ !*}
({\LL}^{i}[l])$ if and only if $ \varphi  \circ h'= 0$.
  \end{lem}
 \begin{proof} 
 In the case of  a unique  local system $K' = i_{!*}\LL'[l]$,  we have  by definition  $ i_{
!*}{\LL}'[l] = \tau_{\leq -1} Rk_{v *} k_v^*i_{ !*}{\LL}'[l]$; moreover $ \tau_{\geq 0} Rk_{v *} k_v^*i_{ !*}{\LL}'[l]$ is supported by $v$. 

Then we have the following exact sequence and  isomorphisms 
\begin{equation*}
\begin{split}
&0 \to  i_{!*}{\LL}'[l] \to \ph{0}{Rk_{v
*}k_v^*i_{S_l!*}{\LL}'[l]}\to i_{v,*} R^0 k_{v *}k_v^*(i_{!*}{\LL}'[l]\to 0\\
 & \ph{i}{R k_{v *} k_v^*i_{!*}{\LL}'[l]} = 0 \,\, {\rm for}\,\,  i < 0\\
& \ph{i}{Rk_{v *} k_v^*i_{!*}
 {\LL}'[l]} = R^i k_{v *}k_v^*i_{!*}
({\LL}'[l]) \, {\rm for} \, i >0 \\
&{\underline H}^0(\ph{0}{ Rk_{v *} k_v^*i_{!*}{\LL}'[l])}\overset{\sim}{\longrightarrow}   R^0 k_{v *}k_v^*(i_{!*}{\LL}'[l] 
 \end{split} 
 \end{equation*}
 respectively, we have  dual statements for $Rk_{v
!}k_v^*$
 \begin{equation*}
\begin{split}
&0 \to  i_{v *} H^{-1}( i_v^*(i_{!*}{\LL}'[l]) \to  \ph{0}{Rk_{v
!}k_v^*i_{S_l!*}{\LL}'[l]} \to i_{!*}{\LL}'[l]  \to 0\\
 & \ph{i}{R k_{v !} k_v^*i_{!*}{\LL}'[l]} = 0 \,\, {\rm for}\, \, i >  0\\
& i_{v *} H^{i-1}( i_v^*i_{!*}{\LL}'[l] )\overset{\sim}{\longrightarrow}  
\ph{i}{Rk_{v !} k_v^*i_{!*}
 {\LL}'[l]}  \, \,{\rm for} \,\, i < 0 \\
& i_{v *}H^{-1}( i_v^*(i_{!*}{\LL}'[l]) \overset{\sim}{\longrightarrow}  
{\underline H}^0(\ph{0}{Rk_{v !} k_v^*i_{!*}{\LL}'[l])}
 \end{split} 
 \end{equation*}

 If we set  $\LL' = \LL^j$ and since for any complex $C$, $\ph{i}{ C [-j]} = \ph{i-j}{C}$,  we deduce the contribution of  the component $\LL^j[l]$  in  the sum in $K'$:\\
$\ph{i}{R k_{v *}k_v^* i_{!*}\LL^j[l-j]} = R^i k_{v *}
 (k_v^* i_{!*}\LL^j[l-j]) $ for $i > j$,\\
 $\ph{i}{Rk_{v *}k_v^* i_{!*}\LL^j[l-j]} = 0 $ for $i <
j$,   and the dual statements for $R k_{v !}k_v^*$.

Moreover, we deduce  a short exact sequence \\
$ 0 \to  i_{!*}\LL^j[l] \to
\ph{j}
{R k_{v *}k_v^* i_{!*}\LL^j [l-j]} \xrightarrow{h} R^j k_{v *}
 k_v^* i_{!*}\LL^j [l-j] \to 0$
 
\n and $ {\underline H}^0(  \ph{j}{Rk_{v
*}k_v^* i_{!*}{\LL}^j[l-j]} \overset{\sim}{\longrightarrow}   R^j k_{v *}k_v^* i_{!*}{\LL}^j[l-j]$. Dually\\
$ 0 \to   i_{v *} H^{j-1}( i_v^*i_{!*}\LL^j[l-j] )  \stackrel{h'}{\rightarrow} \ph{j}{R
k_{v !}k_v^* i_{!*}\LL^j[l-j] } \to  i_{ !*}\LL^j [l] \to 0$\\
\n and 
 $i_v^* H^{j-1}( i_v^*i_{!*}\LL^j[l-j] )  \overset{\sim}{\longrightarrow}  {\underline H}^0 (\ph{j}{R
k_{v !}k_v^* i_{!*}\LL^j[l-j] }).$
 \end{proof}
 
  {\bf Proof of the proposition \ref{prop}.}
 Although the statement is for any projective $f$,  the proof is given first in the case of fibrations by NCD.
 By the above lemma  (\ref{lem}, ii) we have 

\smallskip 
\centerline{$ {k_v}_{!*} k_v^*  \ph{i}{K} = \oplus_{S_l \subset V-v }{i_{S_l}}_{! *} \LL^i_{S_l}[l] \subset  \ph{i}{R k_{v *} k_v^*  \ph{i}{K}[-i]} \subset  \ph {i}{R k_{v *} k_v^*K}$}

\smallskip 
 i) {\it  We prove ${k_v}_{!*}k_v^*\ph {i}{ K} \subset \im \, {^p\!\rho^i} \subset  \ph{i}{R k_{v *} k_v^*  \ph{i}{K}[-i]}$}:

\n We deduce from the canonical morphism $\phi: R
{k_v}_!  \to R
{k_v}_*$ and the decomposition hypothesis 
 $ k_v^*K = \oplus_j k_v^* \ph {j}{ K}[-j] $ on $V - v$, the first equality below
 \begin{equation*} 
\begin{split}
&\im\, \biggl(\ph {i}{R
{k_v}_! k_v^*K}  \xrightarrow{^p\!\phi_i}  \ph {i}{R
{k_v}_* k_v^*K }\biggr) =
 \\
   \oplus_j & \im\, \biggl(\ph {i}{ R {k_v}_ ! \ph {j}{ k_v^*K}[-j]}  \xrightarrow{^p\!\phi^j_i}  \ph {i}{ R {k_v}_ * \ph {j}{ k_v^*K}[-j]}\biggr) =\\
 &  \im\, \biggl(\ph {0}{ R {k_v}_ ! k_v^*\ph {i}{ K}}  \xrightarrow{^p\!\phi^i_i}  \ph {0}{ R {k_v}_ *k_v^*\ph {i}{ K}}\biggr) = {k_v}_{!*}k_v^*\ph {i}{ K}
\end{split}
\end{equation*}
 In the direct sum over $j$, only the term for $j = i$ is significant since for  $j<i$
 
 \n  $\ph {i}{R {k_{v}}_! k_v^* {i_{S_l}}_{! *} {\LL}^j_{S_l}[l-j]} = 0$,  and   $\ph {i}{R {k_{v}}_* k_v^* {i_{S_l}}_{! *} {\LL}^j_{S_l}[l-j]} = 0$ for   $j>i$; hence the second equality follows.     
The last equality follows from  the construction of ${k_v}_{!*}k_v^*\ph {i}{ K}$ in (\cite{BBD}, subsection 2.1.7) as image of  the morphism $ {^p\!\phi^i_i}$. 

Since $ {^p\!\phi^i_i}$ factors through $\ph {i}{K}$, its  image  ${k_v}_{!*}k_v^*\ph {i}{ K}$
is contained in the image of  $\ph {i}{ K}$.
 Since   $k_v^*  \ph{i}{K}$ is a component of $k_v^*  K[i]$, we remark that $ \ph {0}{ R {k_v}_ *k_v^*\ph{i}{ K}}$ is a component of $\ph{0}{R k_{v *} k_v^*K[i]} = \ph{i}{R k_{v *} k_v^*K } $.

The existing  morphism on $V-v$: $k_v^*  \ph{i}{K} \rightarrow k_v^* K[i] $ extends by $R k_{v *}$ and $R{k_{v}}_!$ such that 
 $ {^p\!\phi^i_i}$ is compatible with the factorization
 
\centerline{$\ph {0}{R{k_{v}}_! k_v^* K[i]} \rightarrow \ph {0}{K[i]}
\xrightarrow{\ph {0}{\rho^i}} \ph {0}{R k_{v *} k_v^*K[i]}$}

\smallskip
\n  Hence $\im 
\, {^p\!\rho^i}$ in the formula \ref {elf} contains ${k_v}_{!*}k_v^*\ph {i}{ K}= \oplus_{S_l \subset V-v }{i_{S_l}}_{! *} \LL^i_{S_l}[l] $.

\n   To show that   ${^p\!\rho^i}$     factors  through
 $  \oplus_{S_l \subset V-v }{i_{S_l}}_{! *} \LL^i_{S_l}[l] \subset \ph {i}{R k_{v *} k_v^*K}$, we   prove,  in view of  lemma (\ref{lem}, ii), that the induced morphism  $ \rho^i_0 := H^0 (i^*_v \, {^p\!\rho^i})$   vanish in degree $0$.
 \begin{lem}\label{L1}
  The  morphism 
 $ \rho^i_0 $ induced by $i^*_v \, {^p\!\rho^i}$   on the cohomology in degree zero
vanish

\smallskip
\centerline{ $ \rho^i_0:  H^0( i^*_v \ph {i}{K})
{\longrightarrow} 
H^0( i^*_v \ph {i}{ R k_{v *} k_v^*K}) $}
 \end{lem}
 This is a basic argument where Hodge theory is needed in the proof. We assume  that $f $ is a fibration by NCD over the strata (section \ref{St1}), in which case  we refer to section $4$ for  Hodge theory   and to (section \ref{pure} below),  for the semi-purity theorem  needed here.

  First, we remark the following geometrical interpretation in terms of a small ball $B_v  \subset V$ and its inverse $B_{X_v} \subset X$. 
 By lemma (\ref{lem} i), we have

\centerline{ $ H^0(i_v^* \ph {0}{ R k_{v *} k_v^*\ph {i}{K}}) \simeq H^0(\oplus_{S_l \subset V-v }k_v^* {i_{S_l}}_{! *}{\LL}_l^i [l])
\overset{\sim}{\longrightarrow}     \H^0 (B_v - v, \ph {i}{K})  $}
\n Second we deduce from the triangle:
  
\centerline {   $ \pt_{< 0}(K[i]) \to   \pt_{\leq 0}(K[i]) \to \ph {i}{K} \xrightarrow{[1]}$} 
 
\n  an exact sequence:
 
\centerline { $\H^{0} (B_v , \pt_{< 0}(K[i]))\to \H^{0} (B_v , \pt_{\leq 0}K[i])  \to \H^0 (B_v , \ph{i}{K})$
$ \to \H^1 (B_v , \pt_{< 0}K[i])$} 
\n where $\H^{r} (B_v , \pt_{< 0}(K[i])) = 0$ for $r =0, 1$,  hence

\smallskip
 \centerline{$ \H^{0} (B_v , \pt_{\leq i}(K)) \simeq   \H^0 (B_v , \ph{i}{K}) \simeq   H^0(i^*_v\,\ph {i}{ K}) $}

\smallskip
\n which enables the   factorization of $\rho^i_0$  as a morphism 

\smallskip
\n $ H^0(i^*_v\,\pt_{\leq i}K)=  \H^{0} (B_v, \pt_{\leq i}K) \to  \H^{0} (B_v -v , \pt_{\leq i}K)  \longrightarrow  \H^{0}
(B_v - v, \ph {i}{K})$.

\smallskip
\n 
 In other terms, we have an interpretation  of  $\rho^i_0$ as a composition morphism via $\rho'$
$$  \H^{0} (B_v, \pt_{\leq i}K)  \overset{\rho'}{\longrightarrow} \pt_{\leq i}\H^{i} (B_{X_v} - X_v, \ilm){\longrightarrow}  Gr^{\pt}_i\H^{i} (B_{X_v} - X_v, \ilm
).$$

\n  and through $\H^{i} (X_v,  \ilm)$ as follows:

\smallskip 
\n  $  \H^{0} (B_v, \pt_{\leq i}K)  \to \pt_{\leq i}\H^{i} (X_v,  \ilm)\to  \pt_{\leq i}\H^{i} (B_{X_v} - X_v, \ilm
)$
 
 \smallskip 
 \n where the space $\H^{i} (X_v,  \ilm) $ has weight $w \leq a+i$ (corollary \ref{wi}). By assumption  the semi purity  apply to $ Gr^{\pt}_i\H^{i} (B_{X_v} - X_v,  \ilm)$ of weight
   $w > a+i$ (proposition \ref{Sp}), hence we deduce  $\rho^i_0 =0$. 
 
 \medskip
 \n  ii) {\it Proof of } $Im \,{^p\alpha^i}\overset{\sim}{\longrightarrow}  i_{v,*}{\LL}^i_v$.
 We deduce from  lemma (\ref{lem} ii) as above:

$H^0({^p\!\mathcal { H}}^{i-1} ( {{R k}_{v}}_*
      K_{\vert V-\{v\}}))
\overset{\sim}{\longrightarrow}  R^{i-1} k_{v *}k_v^* ({^p\!\tau_{\leq i-1}}K)$.
  
\n    We deduce from the formula (\ref{elf}) and the computation of ${^p\!\rho^{i-1}}$ above, a sub-exact sequence of perverse sheaves :
   \begin{equation*}
 0 \to  R^{i-1} k_{v *}k_v^* ({^p\!\tau_{\leq i-1}}K)
 \stackrel{^p\!\delta^{i-1}}{\rightarrow} \H^i_{X_v} (X,\ilm)
\stackrel{^p\!\alpha^i}{\rightarrow}{^p\!\mathcal H}^i ( K)
     \stackrel{^p\!\rho^i}{\rightarrow}
\end{equation*} 
where   the perverse
  cohomology of $i_v^{!}(K)$ coincides with its cohomology 
   $\ph{i}{i^!_v K}= H^i_{X_v}(X, \ilm)$ as a complex of vector spaces, hence $ \im \,{^p\!\alpha^i} = \coke \, {^p\!\delta^{i-1}}$.
  
   On the other hand, by definition  ${\LL}^i_v$
is the image of   $I^i$ in the exact sequence 
\begin{equation}\label{I}
H^{i-1} ({i^*_v}{ R k_{v}}_* K_{\vert V-\{v\}} )
\stackrel{\delta^{i-1}}{\rightarrow}
 \H^i_v (V, K) \stackrel{I^i}{\rightarrow} H^i({i^*_v}  K)
  \stackrel{\rho^i}{\rightarrow}
   H^i ({i^*_v} R k_{v*} K_{\vert V-\{v\}} ).
\end{equation}
hence ${\LL}^i_v :=  \im \, I^i = \coke \, \delta^{i-1}$.
  
We need to prove $\im \, \delta^{i-1} = \im \,{^p\!\delta^{i-1}}$,
 to deduce the result in the form of  an  exact sequence
\begin{equation}\label{P1}
  0 \to i_{v,*}{\LL}^i_v \to {^p\!\mathcal { H}}^{i}(K) \to
\oplus_{S_l \subset V-v } {i_{S_l}}_{! *} {\LL}^j_{S_l}[l] \to 0.
\end{equation}

 We deduce from  (lemma \ref{lem}
  ii) and the decomposition on $V-v$:

\smallskip
\n $R^{i-1} k_{v *}k_v^* ({^p\!\tau_{\leq i-1}}K)  \overset{\sim}{\longrightarrow}  
H^0({^p\!\mathcal { H}}^{i-1} ( {{R k}_{v}}_*
      K_{\vert V-\{v\}}))
 = ({^p\!\tau_{\leq i-1}}\H^{i-1}(B_{X_v}-X_v, \ilm)) $

\smallskip
\n  and the following interpretation  of the image:
 
 \smallskip
  \n  $ {^p\!\delta^{i-1}} ( R^{i-1} k_{v
*} k_v^* ({^p\!\tau_{\leq i-1}}K)) \overset{\sim}{\longrightarrow} 
\delta^{i-1}({^p\!\tau_{\leq i-1}}\H^{i-1}(B_{X_v}-X_v, \ilm))  $.
 
 \smallskip
\n 
 Since $H^{i-1} ({i^*_v}{ R k_{v}}_* K_{\vert V-\{v\}} ) = \H^{i-1}(B_{X_v}-X_v, \ilm))$,  we need to prove 
\begin{lem}\label{L2}

$\delta^{i-1}({^p\!\tau_{\leq i-1}}\H^{i-1}(B_{X_v}-X_v, \ilm)) = \delta^{i-1}(\H^{i-1}(B_{X_v}-X_v, \ilm)) $ 

\end{lem}
The proof is based again on  the semi purity theorem.
Indeed, the quotient space
$\H^{i-1}(B_{X_v}-X_v, \ilm))/ {^p\!\tau_{\leq i-1}}$ has weight
$w < a+i$ by the semi purity theorem, hence the image   of $
{^p\!\delta^{i-1}} $ and $ \delta^{i-1}$ in $\H^i_{X_v} (X, \ilm)$
of weight $w \geq a+i$, are the
same. 

iii)  {\it Proof of the splitting $  \ph {i}{K} \overset{\sim}{\longrightarrow}  i_{v *} {\LL}^i_v \oplus (\oplus_{S_l \subset V-v } {i_{S_l}}_{! *}
{\LL}^i_{S_l}[l])$.}
We consider the morphisms
\begin{equation}\label{import}
\ph {i}{R k_{v!} k_v^*K}
  \stackrel {^p\!\beta^i}{\rightarrow}
   \ph {i}{K} \stackrel{^p\!\rho^i}{\rightarrow}  \, \im\,
{^p\!\rho^i} \subset \ph {i}{R k_{v*} k_v^*K}
\end{equation} 
 where $\beta^i$ is dual to 
 $\rho^{-i}$ as $\beta: R k_{v!} k_v^*K \to K$ is dual to $\rho$. We deduce from  lemma (\ref{lem}, iii) a  dual argument to the proof  in ii) above to assert that $\beta^i$
 factors through $\oplus_{S_l \subset V-v } {i_{S_l}}_{! *}
{\LL}^i_{S_l}[l]$, hence  ${^p\!\rho^i}$ induces an isomorphism $ \im \, {^p\!\beta^i} \to \im \, {^p\!\rho^i}\circ {^p\!\beta^i} := \oplus_{S_l \subset V-v } {i_{S_l}}_{! *}
{\LL}^i_{S_l}[l]$ and  defines a splitting.

 In the sequence of morphisms: $\H^i_v(V, K) \xrightarrow {^p\!\alpha^i}
 \ph {i}{K}\xrightarrow {^p\!\gamma^i}H^i(i_v^* K)$ we have $\im \,{^p\!\alpha^i} = \ke \, {^p\!\rho^i}$ and $ \ke \,{^p\!\gamma^i}= \im\, {^p\!\beta^i}$, hence
${^p\!\gamma^i}$ induces  on $\im \,{^p\!\alpha^i}$ an isomorphism to $\im \,{^p\!\gamma^i}\circ {^p\!\alpha^i}:= i_{v *}{\LL}^i_v$.
 
\begin{rem}\label{special} 
i) We deduce from the above proof  (formula 2.3 and 2.4) 

\smallskip
\centerline{ $ \im \ {^p\delta^{i-1}}  \simeq  \oplus_{S_l \subset V-v , i-j > 0 } H^{i-j}_v ({i_{S_l}}_{! *}
\LL^j_{S_l}[l]), \, \im \ {^p\delta^{i-1}} = \ke \ {^p\alpha^i} = \ke I^i =   \im \ \delta^{i-1}$}

\smallskip
\n and  the following  isomorphism induced by ${^p\delta^{i-1}}$:

$ \oplus_{S_l \subset V-v, i-1-j \geq 0} R^{i-1-j} k_{v *} 
 k_v^*{i_{S_l}}_{!*} \LL^j [l]
\overset{\sim }{\longrightarrow}  \oplus_{S_l \subset V-v , i-j > 0 }
 H^{i-j}_v ({i_{S_l}}_{! *}\LL^j_{S_l}[l])$

\medskip
ii) Hodge theory is used in the proof of the three lemmas \ref{L3}, \ref{L1}, \ref{L2}.
{\it We use semi purity at $X_v$ or equivalently   local purity at   $v$ (section \ref{pure1})  to extend the decomposition property  across $v$}.
\end{rem}
 \subsection{Hard Lefschetz} \label{Hard}
  We  check  for all $i \geq 0$, $ \eta^{i}: \ph {-i}{K} \overset{\sim}{\longrightarrow}  \ph {i}{K}$ is an isomorphism. By assumption the restriction  of $ \eta^{i}$ to  $V-v$ on $  k_v^*\ph{-i}{K}$
  is an  isomorphism, hence it remains an  isomorphism  on the intermediate extension 
  
\smallskip
\centerline{ $  (k_v)_{!*} k_v^*\ph{-i}{K} \overset{\sim}{\longrightarrow}  \oplus_{S_l \subset V-v, j} {i_{S_l}}_{! *}{\LL}^j_{S_l}[l] \subset   \ph {-i}{K} $}

\n across the point $v$ in the strata $V_0$ of dimension $0$.
  It remains to prove Hard Lefschetz   $ {\LL}^{-i}_v   \overset{\sim}{\longrightarrow} {\LL}^{i}_v$
 where ${\LL}^i_v := \im \, (\H^i_{Y_v}
(X, \ilm) \stackrel{I^i}{\rightarrow} \H^i(Y_v, \ilm))$.
\begin{lem} i) The cup-product with the class of an hyperplane section 
induces isomorphisms  $\eta^{i}:{\LL}^{-i}_v \to {\LL}^{i}_v$ for $i  > 0$.\\
ii) The $HS$ on ${\LL}^{i}_v$   is Poincar\'e dual to
${\LL}^{-i}_v$ for $i \geq 0$. 
 \end{lem}
In the case of a big strata $S$, we have   a local system  $ {\LL}^i_S = R^if_* \ilm$ and the lemma follows from Hard Lefschetz for intersection cohomology  then  the decomposition follows by Deligne's argument in \cite {DL} applied to the VHS on $ {\LL}^i_S = R^if_* \ilm$ (subsection \ref{rl}).  The general statement of Verdier duality for proper morphisms apply to prove the duality between $ {\LL}^i_S$ and $ {\LL}^{-i}_S$.
 
 At a point $v$ of a lower dimensional strata, Verdier duality between perverse cohomology sheaves in degree $i$ and $-i$ is compatible with the decomposition established above and apply to $  (k_v)_{!*} k_v^*\ph{i}{K}$ by induction.  The duality between 
 ${\LL}^{i}_v$ and ${\LL}^{-i}_v$ in ii) is deduced from the duality of  $ R i_{X_{v}}^{!}$ and $R  i_{X_{v}}^*$
  and Verdier duality on $X_v$ in the definition of ${\LL}^i_v$ (formula  \ref{E1}).

   {\it To prove i)  at the point $v$} we consider an hyperplane
section $H$  intersecting all
NCD in  $X$ normally and proceeds  by induction. Let $i_H$ denote the closed embedding of $H$ in $X$; the cup
 product    defines a morphism $\eta$ equal  to the composition of
 the morphisms $ \ilm  \stackrel{ \rho}{\rightarrow} i_{H*}i^{*}_H \ilm
\overset{\sim}{\longrightarrow}  i_{H*}Ri^{!}_H ( \ilm) [ 2] \stackrel{G}{\rightarrow}  \ilm
[ 2]$. We apply the functors $R i^{!}_{X_v}$ and
$i^{*}_{X_v}$ to the above morphisms $\rho,\, G$, and $ \eta$ as in the commutative diagram
$$\begin{array}{ccccc} \H^i_{X_v}(X, \ilm)
&\xrightarrow { \rho_i ^!} &\H^i_{X_v \cap H} (H, \ilm)
 &\xrightarrow{ G_{i+2}^!}& \H^{i+2}_{X_v}(X, \ilm)\\
I^i{\downarrow}\quad  &   &I^i_{ H}{\downarrow}\quad  && I^{i+2}{\downarrow}\qquad  \\
\H^i(X_v, \ilm)
 &\xrightarrow { \rho_i ^*}&\H^i ({X_v \cap H} ,\ilm)
 &\xrightarrow{ G_{i+2}^*}&  \H^{i+2}(X_v, \ilm)
 \end{array}$$
where  on the first line: $(\eta^{!}_v)_ i = G_{i+2}^! \circ  \rho_i ^!$
and on the second line: $(\eta^*_v)_i = G_{i+2}^* \circ  \rho_i^*$ are functorially induced by $\rho$
and $G$, while by definition 
${\LL}^i_v := \im \, I^i,\, {\LL}(H)^i_v := \im \, I_H^i,\,
  {\LL}^{i+2}_v := \im \, I^{i+2} $ are the images of the vertical
maps induced by $I$ from the top line to the bottom line. 

     The morphisms $\rho^*_i$ and the dual morphisms  $G^*_{i+2}$ 
   induce $\rho'_i$ and $G'_{i+2}$ below 

\smallskip
\centerline {
 ${\LL}^i_v \xrightarrow { \rho'_i }
 {\LL}( H)^i_v \xrightarrow{ G'_{i+2}}
{\LL}^{i+2}_v $ \quad  where ${\LL}( H)^i_v := \im \, I_H^i$}

\smallskip
\n with composition equal to $\eta'$ induced by $\eta$.
The proof continue  by induction on $H$.
A striking point however, there is no crucial case as  in (Weil II \cite {WII}, th\'eor\`eme 4.1.1)
 unless $v $ is on the generic strata.

  \begin{lem}  The induced morphism $\rho'_i:{\LL}^i_v \to {\LL}(H)^i_v$
 is an isomorphism for $i < 0$
and by duality $G'_{i+2}: {\LL}(H)^i_v \to
  {\LL}^{i+2}_v $ is an isomorphism for $ i+2 > 0 $.
 \end{lem}
 We prove first the isomorphisms for $i < 0 $
\begin{equation}
  \rho_i^!:
\H^i_{X_v}(X, \ilm) \rightarrow \H^i_{X_v\cap H} (H,
\ilm)
\end{equation}
We  apply Artin- Lefschetz theorem  (\cite{BBD}, corollary 4.1.5) to the affine open subset
  $X_v- (H\cap X_v)$  with coefficients in a complex of  sheaves $\KK \in {^p\!D}^{\leq 0}(X_v)$ to prove the vanishing of $\H^i(X_v-(H\cap X_v), \KK)$ for $i > 0$ and its dual statement:  $\H_c^i(X_v-(H\cap X_v), \KK') = 0$ with coefficients in a complex of  sheaves $\KK' \in {^p\!D}^{\geq 0}(X_v)$ for $i < 0$. 
  
 Since $  i_{X_v}^{*}  (\ilm)[-1] \in {^p\!D}^{\leq 0}(X_v)$  and  its dual   $  i_{X_v}^{!}  (\ilm)[1] \in {^p\!D}^{\geq 0}(X_v)$ we deduce
  $\H_c^i(X_v-(H\cap X_v), i_{X_v}^{!}(\ilm)[1]) = 0 $ for $i < 0$, hence we have for $i < 0$
$$\begin{array}{ccc}
 \H^i_{X_v}(X, \ilm)
&= &\H^{i-1} (X_v, i_{X_v}^{!}\ilm[1])\\
 \rho_i^!\downarrow { \simeq} &   &\downarrow { \simeq}   \\
\H^i_{X_v \cap H}(H , \ilm)
 &=&
 \H^{i-1} (X_v\cap H, i_{X_v}^{!}\ilm[1])
 \end{array}$$
\n The  last equality is the main point   and  follows from  the  isomorphisms 
\begin{equation*}
   i_{X_v \cap H}^*R i_{X_v}^! (\ilm)\overset{\sim}{\rightarrow}  R i_{X_v \cap
H}^! i_H^* (\ilm), \quad   i^{*}_H (\ilm[-1]) \overset{\sim}{\rightarrow} 
 { j_H}_{!*} i^{*}_{H\cap (X-Y)} \LL[-1])
\end{equation*}
due to the transversality of the intersection  of $H$ and $X_v$.

\smallskip
{\it Proof of the lemma.}   We extend the  diagram above by  introducing the  kernel 
 of  the Intersection morphisms $I$ to get  two columns of  short  exact sequences 
 $$\begin{array}{ccc}
 H^{i-1} (i^*_v R(k_v)_* K_{\vert V-\{v\}})\qquad  &
\xrightarrow{ \rho^*_{i-1}} & \qquad H^{i-1} ({i^*_v} R (k_v)_ * {K_H}_{\vert V-\{v\}})\\
 \delta^{i-1} \downarrow \qquad && \delta^{i-1}_{ H} \downarrow \qquad \\
 \H^{i}_v (V,K) \xrightarrow{\sim}  \H^i_{X_v}(X, \ilm)&\xrightarrow{ \rho^!_i}&  \H^i_{X_v \cap H} (H,\ilm)
 \xrightarrow{\sim}  \H^{i}_v (V,K_H) \\
  I^i \downarrow &&
I^i_{ H} \downarrow \quad  \\
{\LL}^i_v = \im \, I^i& \xrightarrow { \rho'_i }&
 {\LL}( H)^i_v= \im \, I^i_H
 \end{array}$$
where  $K_H = R(f_{\vert H})_*
 ({j_{\vert H}})_{!*}\ilm_{\vert H}$, 
 ${\LL}^i_v \subset
 \H^{i} (i^*_v K) =   \H^i(X_v, \ilm)$ and  
 ${\LL}( H)^i_v \subset  \H^i ({X_v \cap H} ,\ilm)
  =  \H^{i}(i_v^* K_H)$.
   
On the right column the short  exact sequence is defined  by
the  perverse shifted restriction ${\LL}_{\vert H}[-1]$ of ${\LL}$, for which we  suppose Hard Lefschetz  by induction. 

  We prove that for $i < 0$, the morphism $\rho^!_i$, which  is an isomorphism by the lemma, 
induces isomorphisms 
\begin{equation}
\rho^!_{i \vert}: Im \, \delta^{i-1} = \ke \, I^i \overset{\sim}{\longrightarrow}  \im \,
\delta^{i-1}_H = \ke \, I_H^i
\end{equation}
 hence induces isomorphisms: $\rho'_i:
{\LL}^i_v \xrightarrow{\sim} 
 ({\LL}{\vert H})^i_v $ for $i < 0$.
 
{\it Proof of the isomorphism $\rho^!_{i \vert}:  \ke \, I^i \overset{\sim}{\longrightarrow}   \ke \, I_H^i$}.
By the  remark (\ref{special} i)
 
 \smallskip
 \centerline{
   $ \ke (I^i) = Im (\delta^{i-1})
= Im (\oplus_{S_l \subset V-v , j < i} H^{i-1-j} ( {i^*_v}
 {R k_v}_* k_v^*{i_{S_l}}_{!
  *}\LL^j_{S_l} [l]))$}
  
 \smallskip
  \centerline{ $  \ke (I_H^i) = Im (\delta^{i-1}_H) =
   Im (\oplus_{S_l \subset V-v, j < i}
 H^{i-1-j}(i^*_v R (k_v)_* k_v^*{i_{S_l}}_{! *}
 (\LL_{\vert H})_{S_l}^j [l])) $.}
 
  \smallskip
 \n The proof is reduced to the comparison of $\LL^j_{S_l}$ and  $(\LL_{\vert H})_{S_l}^j $
  on each component  $S_l$.
 
 We consider for each $S_l$,   a normal section $N_{v_l}$ to $S_l$ at a general point
$v_l \in S_l$, then:
 \[
 (R^{-l+j} f_*i_{X_{S_l}}^{! }{\ilm})_{v_l}\overset{\sim}{\longrightarrow} 
\H^{-l+j}_{X_{v_l}}(X_{N_{v_l}},{\ilm})\overset{\sim}{\longrightarrow} 
\H^{j}_{X_{v_l}}(X_{N_{v_l}},{\ilm}[-l])
 \]
 where ${\ilm}[-l]$  restricts  to a perverse sheaf on $X_{N_{v_l}}$
 of  dimension dim.$X - l$.  Since $i < 0$ and $ j
< i $ we have $ j < -1 $, moreover $X_{v_l}-(H\cap X_{v_l})$ is affine,
hence  the restriction of $ \rho^!_j$ to $X_{N_{v_l}}$
 is an  isomorphism by 
the hyperplane section theorem on the  fibres of $f$ over $S_l$: $\rho^!_{j\vert}:  \H^{j}_{X_{v_l}}(X_{N_{v_l}},{\ilm}[-l])\to \H^{j}_{X_{v_l}\cap H}(X_{N_{v_l}}\cap H,{\ilm}[-l])$
where the last term is isomorphic to $ (R^{-l+j}
f_*i_{H \cap X_{S_l}}^! \ilm_{\vert H})_{v_l}$.

 \begin{cor}  The iterated cup-product with the  class  of a relative  hyperplane
section induces isomorphisms $\eta^i:{^p\HH}^{-i}(K) \to
{^p\HH}^i(K) $ for $i
> 0$.
 \end{cor}
 \begin{cor} \label{Vpolar}
For each strata $S$ of $V$, the local system $\LL^i_S$ is a polarized variation of Hodge structure on the smooth variety $S$.
\end{cor}
 We consider here the case of a fibration  $f: X \to V$ by  NCD  over the strata with $X$ smooth,  we consider at a point $v \in S$   open neighborhoods  $B_v$ of $v$ in $V$ and $S_v $ in $S$, and a projection $p_v: B_v \to S_v$, inducing the identity on $S_v$, such that $p_v \circ f : X_{B_v} \to S_v$ is smooth.  

The fibers  $N_a :=  p_v^{-1}(a) $ at  points $a \in S_v$ form a  family of normal sections to $ S_v$ such that $X_{N_a} := f^{-1}(N_a) $ is a smooth sub-variety of $X_{B_v}$. By an argument based on the relative complex $\Omega^*_{X_{B_v}/ S_v} \LL $ as in section \ref{rl}, the families $R^k f _*(i_{X_{S_v}}^!(j_{!*}\LL)$
(resp. $R^k f _*(i_{X_{S_v}}^*(j_{!*}\LL)$) form variations of MHS on $S$ 
 of weight $\geq a+k+l$  (resp. $\leq a+k+l$) (corollary \ref{wi}).

Then  the image of the Intersection morphisms $\LL^k_{S_v}$   are pure variation of HS of weight $a+k+l$ on $S_v$. 

 \subsection
{Proof of  the local purity theorem \ref{pure}}\label{pure1}
We did assume in the statement of the proposition \ref{prop} the local purity at $v$ to simplify the exposition. We show now the local purity at $v$ follows in fact from the decomposition on $V-v$.
We assume  here that $f$ is  a fibration by NCD over the strata in order  to apply  the construction 
of Hodge theory  in sections $4 - 5$.  We refer precisely to the subsection (\ref{tube}) for the construction of the mixed Hodge structure used in this theorem and to the compatibility with the perverse filtration.

The  proof of the equivalent semi purity at $X_v$ (proposition 
 \ref{Sp}) differs from \cite{DG}
as we use the polarization of the Intersection cohomology of $X$. 

Let $\ilm $ be a shifted polarized VHS of weight $a$ on the smooth and compact
 variety $X$, 
    $X_v := f^{-1} (v)$  the fiber at $v \in V$, and  $ B_{X_v} =  f^{-1} (B_v)$ the inverse of a small neighborhood $B_v$ of $v$.

\subsubsection{Inductive hypothesis}\label{indu}
{\it   We assume by induction the decomposition  theorem on $V -v$,  hence we have an isomorphism:
\begin{equation} 
Gr^{\pt}_i \H^r(B_{X_v}- X_v,\ilm)\overset{\sim}{\longrightarrow} 
\H^{r-i}(B_v-\{v\},\ph{i}{Rf_* \ilm}).
 \end{equation}
 This isomorphism is used to carry  the MHS from the left  term (subsection \ref{C6}) to  the right}.
 Under such isomorphism, we  prove the following inequalities  on the weights $ w$ of the $MHS$  
  
  i) $w >  a+r $ on ${\pt}_{\leq r} \H^r(B_{X_v} -
 X_v,\ilm) $, and
 dually: 
 
 i') $w
\leq a+r $ on $ \H^r(B_{X_v}-X_v,\ilm) / {\pt}_{\leq r}
\H^r(B_{X_v}-X_v,\ilm) $.\\
 {\it or equivalently  for  $j =r-i $:
 
 ii) $w > a+i+j $ on $H^j(B_v-v,\ph{i}{Rf_*\ilm})$ if $j \geq 0$,
and dually:

 ii') $w \leq a+i+j $ on $H^j(B_v-v,\ph{i}{Rf_*\ilm})$ if $j
\leq -1$}.
\begin{ex}
 The dual of $Gr^W_{a+l} \H^{-1}(B_v- v, \ph{i}{Rf_*\ilm} ) $  for  all $i$ and $l \leq i-1$ in the assertion ii'),  is
 $Gr^W_{a-l} \H^0(B_v- v,  \ph{-i}{Rf_*\ilm}) $ 
 for all $i$ and  $ a-l \geq a-i+1$  in the assertion ii); we remark the  change of the variable $i$ into the variable $-i$.
\end{ex}
\subsubsection{
 Duality}  
 Let $K := Rf_* \ilm$ and $k_v: V-\{v\} \to V$. We have:
$$ K(*v) := i_v^*Rk_{v*}k_v^* K \overset{\sim}{\longrightarrow}  R\Gamma (B_v - v, K) \overset{\sim}{\longrightarrow} 
 R\Gamma (B_{X_v} - X_v, \ilm). $$
 The duality isomorphism $D (K(*v)) [1]
\overset{\sim}{\longrightarrow}  K(*v) $ where $D$ stands for Verdier dual, is deduced from the duality between
$Rk_{v*}$ and $Rk_{v !}$ (resp. $i_v^*$ and $R i_v^!$) as
follows. We apply to $k_v^* K$ two sequences of functors 
$R k_{v!} \to Rk_{v *} \to i_{v *}i_v^* R k_{v *}$ and 

\n $ Rk_{v!}\to Rk_{v *} \to i_{v *}Ri_v^! Rk_{v !} [1]$, 
 defining dual triangles,   from which we deduce 

\centerline {$ D(K(*v))[1]  \overset{\sim}{\longrightarrow}  i_{v *}Ri_v^! Rk_{v !} k_v^* K[1]
\overset{\sim}{\longrightarrow}  K(*v)$.}

\n which corresponds to the duality of the Intesection cohomology of  the ``link'':  $ \partial B_{X_v} = f^{-1} ( \partial B_v) $ with coefficients in  the restriction of  $\ilm$. 

Since the perverse filtration is compatible with the MHS over $V-v$, we deduce
\begin{equation*}
 D( Gr^W_{a-q} H^{-j}( K(*v))) \overset{\sim}{\longrightarrow}  Gr^W_{a+q} H^j(D K(*v))\overset{\sim}{\longrightarrow} 
   Gr^W_{a+q} H^{j-1}( K(*v)).
\end{equation*}
\n and since the duality: $D(\ph{-i}{K}) \overset{\sim}{\longrightarrow}  \ph{i}{K}$ follows from the auto-duality of $\ilm$ by Verdier's direct image theorem for $f$ proper, we deduce
\begin{equation*}
 D( Gr^W_{a-q} Gr^{^p\!\tau}_{-i} H^{-j}( K(*v)))
  \overset{\sim}{\longrightarrow} 
   Gr^W_{a+q}Gr^{^p\!\tau}_{i} H^{j-1}( K(*v)).
\end{equation*}
Hence, the proof is reduced by the above duality, to one of the two  cases ii) or ii')  in degree $j \geq 0$ or 
$j \leq -1$.

\subsubsection{Proof by induction on dim.$X$} 

Let $H$ be a general hyperplane section  of $X$
transversal to all strata, $i_H: H \to X$ and  $j_H:(X-H) \to X$.  
The restriction to $H$ of the intermediate extension $  i_H^* \ilm$   is  equal to $j_{!*}(\LL_{\vert H})$ by transversality. 

 {\it We assume   the local purity theorem for the perverse cohomology sheaves of  $ Rf_* j_{!*}(\LL_{\vert H}[-1])$, and we use Artin-Lefschetz vanishing theorem to deduce the local purity  for the perverse cohomology shaves of  $ K$ in degree $i \not= 0$ as follows.}
 
\begin{lem}\label{D} 
Let $ K_c = Rf_* ( j_H)_! j_H^*\ilm$ (resp.  $ K(*) = Rf_* ( j_H)_* j_H^*\ilm$), then
the complex  $ K_c \in {^pD}_V^{\geq 0}$ , or equivalently   $ \ph{i}{ K_c} = 0$  for $i < 0$
(resp.  $ K(*) \in {^pD}_V^{\leq 0}$  or equivalently   $ \ph{i}{ K(*)} = 0$  for $i > 0$).
\end{lem}

\begin{proof}
  Since the morphism $f\circ j_H: (X -H) \to X$ is affine, it follows that  the functor $Rf_*\circ j_{H !} =  R(f \circ j_H)_ ! $ is left t-exact (\cite{BBD}, corollary 4.1.2), which means that it transforms  ${^pD}_{X-H}^{\geq 0}$ into $ {^pD}_V^{\geq 0}$.   It  applies to $Rf_*(j_H)_! j_H^*\ilm$ and shows that $K_c \in  {^pD}_V^{\geq 0}$. The result  is a version of Artin-Lefschetz vanishing theorem.
  The statement $ \ph{i}{ K(*)} = 0$  for $i > 0$ follows by duality.
\end{proof} 

Let  $ K = Rf_*\ilm$,  and $ K_H = R(f \circ i_H)_* i_H^* \ilm$.
The restriction morphism  $\rho:  K \to   K_H$, and dually the Gysin morphism: $G:  K_H[-2] \to  K$, induce morphisms compatible with the $\pt$ filtration
\begin{equation*}
\rho_i: \ph{i}{K} \rightarrow \ph{i}{K_H},\, G_i: \ph{i-2}{K_H}\rightarrow \ph{i}{K}, \, L_i:\,{^p{\HH}}^i(K) \lorw \,{^p{\HH}}^{i+2}(K), 
\end{equation*}
where $ L_i = G_{i+2} \circ \rho_i$ (resp. $L: \, K \lorw  K [2]$), is defined by the cup product with the Chern class $c_1$ of a relative ample bundle.
\begin{cor}\label{surj0}
 The restriction $ \rho_{i}:  \ph{i}{Rf_*\ilm}  \to \ph{i}{ Rf_* ( i_H)_* i_H^* \ilm} $ is an isomorphism  for each integer $ i < -1 $ and injective for $i = -1$.

 \n Dually,  the Gysin morphism $G_i: \ph{i-2}{ Rf_* ( i_H)_* i_H^* \ilm} \to \ph{i}{ Rf_*\ilm} $
 
 \n  is an isomorphism for $i >1$,  and it is surjective for $i = 1$.
\end{cor}
We deduce from  the triangle 
$  ( j_H)_! j_H^*\ilm \to \ilm \to i_{H *}i_H^* \ilm \xrightarrow{(1)}$,
 and its  direct image by $Rf_*$, an exact sequence of perverse cohomology 
\begin{equation}\label{es}
 \cdots \to \ph{i}{ K_c} \xrightarrow{\can_i}  \ph{i}{K} \xrightarrow{\rho_i}   \ph{i}{K_H}
 \xrightarrow{\partial_i} \ph{i+1}{ K_c} \to   \cdots
\end{equation} 
The corollary follows from   lemma \ref{D} as the complex  $ K_c \in {^pD}_V^{\geq 0}$ , or equivalently   $ \ph{i}{ K_c} = 0$  for $i < 0$. 

{\it Next,  we use  the inductive assumptions over $V-v$ in the  crucial case for $i = 0$}.
\begin{lem}
 If we suppose  the relative Hard Lefschetz theorem for  the morphisms $f $ and  $f \circ i_H$
 restricted  to $V-v$,
 then $G_0 $ is injective and $ \rho_{0}$ is surjective on $V-v$. 
  Moreover, we have a decomposition
 \begin{equation}\label{1+}
   \ph{0}{Rf_*\ilm}_{\vert V-v}   \overset{\sim}{\longrightarrow} \im \, G_0  \oplus \ker \, \rho_{0}.
\end{equation}
Respectively, $ \rho_{-1}$ is injective, $G_1 $ is surjective and 
we have a decomposition
\begin{equation} \label{2+}
  \ph{-1}{Rf_* ( i_H)_* i_H^* \ilm}_{\vert V-v}  \overset{\sim}{\longrightarrow} \im \,  \rho_{-1}  \oplus \ker \, G_1 
\end{equation}
 Moreover  $G_0: \ph{-2}{Rf_* ( i_H)_* i_H^* \ilm}_{\vert V-v}  \overset{\sim}{\longrightarrow} \im \,  G_0 \subset \ph{0}{ Rf_*\ilm}_{\vert V-v}$
 is   an isomorphism onto its image.
\end{lem} 
\begin{proof}
Recall that $\LL$ is defined on $X-Y$. Let $ j': (H - H\cap Y) \to H$. We have  by  transversality:  
  $i_H^*\ilm \overset{\sim}{\longrightarrow}  j'_{!*} (i_H^* \LL)$, hence $ \ph{i}{ K_H} =  \ph{i}{Rf_* j'_{!*} (i_H^* \LL)}$. Since by the inductive hypothesis, 
  Hard Lefschetz theorem apply on $V-v$ to $K_H [-1]$, image of  the intermediate extension  $ j'_{!*}i_H^* \LL[-1]$ of the shifted restriction of $\LL$ on $H$, we deduce that the composition morphism  $\smile c_1 = \rho_0 \circ G_0$ in the diagram is an isomorphism on $V-v$ 

\smallskip
\centerline {$
   \ph{-2}{ K_H}_{\vert V-v} \overset{G_0}{\longrightarrow} \ph{0}{ K}_{\vert V-v} \overset{\rho_0}{\longrightarrow} \ph{0}{ K_H}_{\vert V-v} $} 

\smallskip
\n where the first term is $ \ph{-1}{Rf_* j'_{!*}i_H^* \LL[-1]}$ and the last $\ph{1}{Rf_* j'_{!*}i_H^* \LL[-1]}$.
  Hence  $\rho_{0}$ is surjevtive, $G_0$ is injective and the decomposition \ref{1+} follows. 

Respectively, the composition morphism  $\smile c_1 = G_{1} \circ  \rho_{-1}$ in the diagram
 
 \smallskip
\centerline {$
 \ph{-1}{ K}_{\vert V-v} \overset{\rho_{-1}}{\longrightarrow} \ph{-1}{ K_H}_{\vert V-v} \overset{G_{1}}{\longrightarrow} \ph{1}{ K}_{\vert V-v}$}

\smallskip
\n
is an isomorphism and the decomposition \ref{2+}  follows.
\end{proof}
The next result is based on Hodge theory as in (proposition \ref{mainp}, definitions \ref{dual0} and 
\ref{dual})
\begin{cor}\label{cor0}
 If the local purity theorem applies for $H$, then it applies for $X$, except eventually for $(\ke \, \rho_0)_{\vert V-v} \subset \ph{0}{Rf_*\ilm}_{\vert V-v}$. 
\end{cor} 
\begin{proof}
The  restriction morphisms $ \rho_i$  are compatible with MHS for all $i$ and $j$ 

\centerline {$H^j(B_v-v,\ph{i}{K}) \to H^j(B_v-v,\ph{i}{K_H})$}
\n and they are isomorphisms for $i < - 1$ by the corollary above.  Then,   the conditions of local purity on $X$ are satisfied  for $i < -1$, since they are satisfied on  $H$ by induction. The dual argument for $G_i$ apply, and we are left with the cases $i = 0, -1, 1$.

The case  $i = -1$ follows from the decomposition (formula \ref{2+}), as $\rho_{-1}$ is an isomorphism onto a direct summand $ \im \,  \rho_{-1}$ of  $ \ph{-1}{Rf_* ( i_H)_* i_H^* \ilm}_{\vert V-v}$. 
 
 By Lefschetz isomorphism on $V-v$, we deduce the case $i = 1$ from the case $i = -1$.
 For $i = 0$, as $G_0$ is an isomorphism by the  lemma, the  inequality holds for $\im\, G_0$.   Only the case of $\ke \, \rho_0$ can not be deduced by induction from $H$. 
 \end{proof}
\subsection
{The crucial case} $w \leq a+j $ on $H^j(B_v-v,  \ph{0}{Rf_*\ilm})$ for $j <0$.
\label{cru}
\

\n The result is local at $v$, so we can  suppose $V$ affine, then  choose a projective embedding to allow the use of  the polarization of the cohomology. 
 
The proof is subdivided in many steps.   
  We  apply the next lemma for $ \ph{i}{Rf_*\ilm}$, in the case of $i = 0$, but it is equallly proved next for all $i$ .

\begin{lem}
 Let $\ilm$ be of weight $a$, then the $MHS$ on
  $$Gr_i^{\pt}H^{i+j}(B_{X_v}-X_v, \ilm)\simeq H^j(B_v-v, \ph{i}{Rf_*\ilm})$$
   is  of weight $\omega >  a+i+j$ for $j > 0$ and dually of weight  $\omega
\leq a+i+j$ for $ j <  -1$.
\end{lem}

\begin {proof}  Let  $H$ be a general hyperplane
   section of $V$  containing  $v$, $H_v = B_v \cap H $ and $ K = Rf_*\ilm $. We suppose $H$ normally embedded outside $v$ so  that the perverse truncation
 commutes with the restriction to $H-v$ up to a shift in degrees;
 We have a Gysin  exact sequence
 \begin{equation*}
{\H}^{j-2} (H_v- \{v\},\ph{i}{K})(-1) \stackrel{G_j}{\rightarrow}  {\H}^j
(B_v- \{v\}, \ph{i}{K})\rightarrow
 {\H}^j (B_v - H_v ,\ph{i}{K})
\end{equation*}
Since
   $ \ph{i}{K} $ is in the category $^pD_c^{\leq 0}V$,
 and $B_v -H_v$ is Stein, we apply Artin Lefshetz hyperplane section theorem
to show that $ {\H}^j (B_v - H_v , \ph{i}{K}) \simeq 0$ for $ j > 0$.
 Then $G_j$ is an isomorphism for $j > 1$ and it is surjective for $j = 1$. 
 
 The smooth strict transform $H'$  of $H$ intersects transversally  in $X$ the various
 subspaces $Y_l$ inverse of the strata $S_l$ so that  Gysin
  morphisms are 
    are compatible with  the $MHS$.
  
  Hence, we deduce the statement for $B_v-v$ in the lemma for $j > 0 $ from the inductive
 hypothesis on the local purity of the $MHS$ structure
 on 
 
 \n   $ H^{j-1} (H_v- v, \ph{i}{K}[-1]) $ for $j-1 \geq 0$, since $H^{j-2} (H_v- v, \ph{i}{K})(-1) \simeq
  H^{j-1} (H_v- v, \ph{i}{K}[-1])$.
\end{proof}

\subsubsection{The large inequality: $w \leq a $ on $H^{-1}(B_v-v,  \ph{0}{Rf_*\ilm})$}
This case also is easily deduced by induction 
\begin{lem}
i)  ${\H}^0(B_v-\{v\}, \ph{i}{Rf_*\ilm})$ is  of weight
 $ \geq a+i$.\\
 ii) Dually: ${\H}^{-1}(B_v-v,\ph{i}{K})$ is of weight
 $ \leq a+i$.
\end{lem}
\begin{proof} Let $k_v: (V-v) \to V$ and consider  a general  hyperplane section $H_1$ not
containing $v$. We  deduce from the triangle of
complexes:\\ $ R k_{v !} k_v^* \ph{i}{K}\rightarrow R k_{v *} k_v^*
\ph{r}{K} \to i_{v *} i_v^* R k_{v *} k_v^* \ph{i}{K}$,
 the exact sequence
\begin{equation*}
\begin{split}
 {\H}^0(V - H_1,R k_{v *} k_v^*  \ph{i}{K})
   \xrightarrow{\gamma_0} H^0( i_v^* R k_{v *} k_v^* \ph{i}{K})
      \rightarrow  {\mathbb H}^1( V-H_1,
R k_{v !} k_v^* \ph{i}{K})
\end{split}
\end{equation*}
which shows that $\gamma_0$ is surjective,  since 
${\H}^1( V-H_1, R k_{v !}  k_v^* \ph{i}{K})$  vanishes as
$V-H_1$ is affine.
 We deduce the assertion i) since the weights
$w$ of the cohomology  ${\H}^0(V - H_1,R k_{v
*} R k_v^* \ph{i}{K})$ of the open set $V-H_1$ satisfy $w \geq a+i$.\end{proof}
\subsubsection
{The crucial step:  $w \not= a$}It remains  to exclude the case $w = a $, which  does not follow by induction, that is
   \begin{equation}\label{c}
Gr^W_a Gr^{\pt}_0
\H^{-1}(B_{X_v}-X_v, \ilm )  \overset{\sim}{\longrightarrow}  Gr^W_a \H^{-1}(B_v-v,
\ph{0}{K}) \overset{\sim}{\longrightarrow}  0.
\end{equation}
from which we  deduce by duality $Gr^W_a
Gr^{\pt}_0 \H^0(B_{X_v}-X_v, \ilm )\overset{\sim}{\longrightarrow}  0  $.
  
  By corollary \ref{cor0}, it remains to consider the case of the sub-perverse sheaf $\ke \rho_0 \subset \ph{0}{K}$ in which case we use the following lifting of  cohomology classes
  
  \begin{lem}\label{ext0}
Let  $B_v^* := B_v-v$ and $H$ a general hyperplane section of $X$. The following map is surjective:

$ W_a  \tau_{\leq 0}\H^{-1}(B_v^*,   Rf_*j_{H!}j_H^*\ilm) \rightarrow   Gr^W_a Gr^{\pt}_0 \H^{-1}(B_v^*,  \ke \,\rho_0)  $.
 \end{lem}
  The proof  relies  on  Hodge theory  appied to
  $K_c$ (formula \ref{es}) as follows.
 \subsubsection
 { Cohomology with compact support  of the fibers of $f_{\vert X-H}$}
  By transversality of $H$ with all strata of $X$, 
 the restriction to $H$ on $\ilm \simeq IC^*\LL$ is defined in terms of  complexes of logarithmic type (subsection \ref{globe})
 
  \n $ \rho: (IC^*\LL, F) \to i_{H *}i_H^*(IC^*\LL, F)$, where 
   $(IC^*(\LL_{\vert H}[-1]), F)[1] = i_H^*(IC^*\LL, F)$.

 We introduce  the  cone $(C(\rho)[-1], F)$ and let $ (K_c,   F):= Rf_* (C (\rho)[-1],   F)$. From the  triangle  $K_c \to K \to K_H \xrightarrow{(1)}$,
 we deduce  an   exact sequence of perverse cohomology 
\begin{equation*}
 \ph{i-1}{K} \xrightarrow{\rho_{i-1}} \ph{i-1}{K_H}\xrightarrow{\partial_{i-1}}\ph{i}{K_c} \xrightarrow{\can_i}  \ph{i}{K} \xrightarrow{\rho_i}   \ph{i}{K_H}
 \end{equation*} 
{\it Proof of lemma \ref{ext0}.}
In the local case, over a neighborhood of $v$, we consider  the long exact sequence 
  $$  \to  \H^{i-1}(B_v^*, K_H)\to \H^i(B_v^*, K_c) \xrightarrow{\gamma}     \H^i(B_v^*, K)  \rightarrow  \H^i(B_v^*, K_H) \to \cdots $$
We put a  MHS on the terms as follows (subsection \ref{tube}).  The mixed cone $C(I_X)$ (resp. $C(I_H)$) puts MHS on $ \H^* (B_v^*, K)$ (resp.  $\H^*(B_v^*, K_H)$) (lemma \ref{dual1}) in terms of  complexes of logarithmic type.
  The restriction   $ \rho:  C(I_X) \to i_{H *} C(I_H)$ is well defined and compatible   with the weight filtration $W$ and the Hodge filtration $F$. 
 We introduce  the mixed cone $(C(\rho)[-1], W, F)$ satisfying 

\smallskip
 \centerline {$(Gr^{W}_i C(\rho)[-1], F) = Gr^{W}_i (C(I_X)\oplus Gr^{W}_{i+1}( i_{H *}C(I_H)[-1])$.}

\smallskip
\n  We remark that  $C(\rho) [-1] \simeq ( j_H)_! j_H^* C(I_X)$,
 and we have  a triangle 
\begin{equation*}
\begin{split}
C(\rho)[-1]\to C(I_X) \to i_{H *}i_H^*C(I_H) \xrightarrow{(1)}
\end{split}
\end{equation*}
inducing   MHS on the terms of the exact sequence.
 
 Such MHS is compatible with the perverse filtration 
(proposition \ref{mainp}, sections \ref{C6}, \ref{Link1}), that is the subspaces  $\H^i(B_v^*, \pt_\leq j K), \,  \H^i(B_v^*,  \pt_\leq j K_H),  \H^i (B_v^*,  \ph{j}{K})$,  and $ \H^i (B_v^*,  \ph{j}{K_H})$ underly corresponding MHS, as well 
$\H^i(B_v^*, \pt_\leq j K_c)$ and  $\H^i(B_v^*,  \ph{j}{K_c})$.
We apply the previous results to the exact sequence 
 $$ \cdots \to \H^{-1}(B_v^*,  \ph{0}{K_c}) \xrightarrow{\gamma}     \H^{-1}(B_v^*, \ke \rho_0)  \xrightarrow{\partial}  \H^0(B_v^*, \coke \rho_{-1}) \to $$
 Taking $Gr^W_a$, we have an  exact squence
  \begin{equation*}
  Gr^W_a \H^{-1}(B_v^*,  \ph{0}{K_c}) \xrightarrow{\gamma_1}   Gr^W_a   \H^{-1}(B_v^*, \ke \rho_0)  \xrightarrow{\partial} Gr^W_a  \H^0(B_v^*, \coke \rho_{-1})  \to 
\end{equation*}
By the inductive hypothesis  on $H$, $ Gr^W_a  \H^0(B_v^*, \coke \rho_{-1}) = 0 $,  hence $\gamma_1$ is surjective so that  elements of  $Gr^W_a   \H^{-1}(B_v^*, \ke \rho_0)$ can be  lifted 
to elements in 

\n $W_a   \H^{-1}(B_v^*, \ph{0}{K_c})$.
 Moreover,   $   \tau_{ \leq 0} K_c =  \ph{0}{K_c}     $
  as $  \tau_{<0} K_c = 0 $, hence  we can lift the elements
to $ W_a\H^{-1}(B_v^*,   \tau_{ \leq 0} K_c)$. This ends the proof of lemma \ref{ext0}.

\subsubsection{Polarization} In this step we use the polarization of Intersection cohomology.
    The following proof  is  based on the idea that the  cohomology of  $B_v-v$ fits in two exact sequences  issued from the  two triangles for $K:= Rf_*\ilm$
\begin{equation}\label{eq}
 i_v^* R k_{v !} k_v^*K \to i_v^*K \to i_v^* R k_{v *} k_v^* K, \quad    R k_v^!  K \to i_v^*K \to i_v^* R k_{v *} k_v^* K 
\end{equation}
 from which
we deduce the commutative 
diagram  (see also the diagram \ref{dia})
$$\xymatrix{
 \H^{ -1}(B_{X_v}-X_v, \ilm )\ar[d]^{ \partial }\ar[dr]^{\gamma}\ar[r]^{\partial_X}&
 \H_c^0(X - X_v, \ilm) \ar[d]^{\alpha_X} \\
   \H_{X_v}^0(X, \ilm)\ar[r]^{A}&
   \H^0(X, \ilm) }$$
  Let $ \gamma $  denotes the composition morphisms 
   $ \gamma = \alpha_X  \circ \partial_X = A \circ \partial $.
   \begin{lem} 
\n   Let $u \in  Gr^W_a \H^{
-1}(B_{X_v}-X_v,\ilm) $ be in the image of   the canonical 

\n  morphism
$ Gr^W_a\H^{ -1}(B_{X_v}- X_v,  j_{H !}j_H^* \ilm ) \xrightarrow{\can} 
 Gr^W_a\H^{ -1}(B_{X_v}-X_v, \ilm )$, 
  then 
  
  \n $ \gamma (u) = 0 $
   in $Gr^W_a \H^0(X, \ilm ) = \H^0(X, \ilm )$.
\end{lem}
\begin{proof} 1) {\it $\gamma (u)$ is primitive}.  Let  $B^*_{X_v} := B_{X_v}-X_v$,   $b \in Gr^W_a\H^{ -1}(B_{X_v}^*,  j_{H !}j_H^* \ilm )
 $ such that $ u =  \can (b)$, and consider the diagram corresponding to an hyperplane section $H$
   \begin{equation*}
Gr^W_a\H^{ -1}(B_{X_v}^*,  j_{H !}j_H^* \ilm )
\stackrel {\partial_!} \rightarrow
 Gr^W_a\H_c^0(X^*,  j_{H !}j_H^* \ilm ) \stackrel {\alpha_!} {\rightarrow}
 Gr^W_a\H^0(X,  j_{H !}j_H^* \ilm ) 
 \end{equation*}
 Let  $\gamma_! :=  \alpha_! \circ \partial_!$, then  $\gamma (u) =  \gamma (\can \, b) = \can \, \gamma_! (b)$,  hence the restriction to $H$: $\rho_H(\gamma (u) )= \rho_H(\can \, \gamma_! (b)) = 0 $,  and  $\gamma (u)$ is a primitive element.
  
2) {\it $\gamma (u) = 0 $}.   We consider the diagram
  
$$\xymatrix{
 &
Gr^W_a \H^0(X, \ilm) \ar[dr]^{A^*}& \\
  Gr^W_a \H_{X_v}^0(X, \ilm)  \ar[ur]^{A} \ar[rr]^{I}&
  &Gr^W_a \H^0(X_v, \ilm) }$$
  
\n where $A^*$ is the dual of $A$.
 Let $P $ denotes the scalar product defined by Poincar\'e duality on $ \H^0(X, \ilm)$ and   $C $ the
Weil operator defined by the HS. The polarization $Q$ on the primitive part of $ \H^0(X, \ilm)$
is defined by $Q(a,b) := P (Ca , \overline b)$.

A non-degenerate pairing $P_v$ 

\smallskip
\noindent \centerline { $P_v:  Gr^W_a \H_{X_v}^0(X,
\ilm)\otimes Gr^W_a \H^0(X_v, \ilm) \to \C$.}

\smallskip 
\noindent  is  also defined by duality.
The duality between $ A$ and $ A^*$ is defined for all
$b \in Gr^W_a \H_{X_v}^0(X, \ilm)$  and $ c \in
Gr^W_a\H^0(X, \ilm)$ by the formula :

\smallskip
\noindent \centerline {$\; P ( A b ,  c ) = P_v( b, A^* c ).$}

\smallskip
\noindent  
Let  $C $ be the Weil operator defined by the $HS$ on $
Gr^W_{a+i}\H_{X_v}^i(X, \ilm)$ as well on $ Gr^W_{a+i}
\H^i(X, \ilm)$, then:

\centerline{ $\, P (C.
A(\partial u), A (\overline { \partial u}) ) = P_v(C \partial  u , A^*
\circ A ( \overline {\partial u })) = P_v(C.\partial (u),  I (\overline {\partial u }))$}

\n and since $I (\partial  u )= A^*\circ A\circ \can (b) =  0$, we  deduce $P (C.
\gamma ( u), \overline {\gamma ( \partial u) }) ) = $ \\
$P (C.
A(\partial u), A(\overline { \partial u }) ) = 0$,
 hence $\gamma (u) = 0$ by
polarization as $\gamma (u)$ is primitive.
\end{proof}
\subsubsection{}The last step of the proof of the crucial case is based on two lemma.
  \begin{lem}\label{del}
The connecting morphism $\H^{i -1}(B_{X_v}-X_v, \ilm
) \stackrel{\partial_X}{\rightarrow} \H_c^i(X-X_v, \ilm
)$   induced by  the triangles in formula (\ref{eq}) is  injective on $Gr^W_i$ \\
  $Gr^W_{a+i}Gr^{\pt}_i\H^{i -1}(B_{X_v}-X_v, \ilm )
\stackrel{ 
\partial_X}{\rightarrow} Gr^W_{a+i}Gr^{\pt}_i\H_c^i(X-X_v, \ilm ).$
\end{lem}
 We prove equivalently $ Gr^W_{a+i}\H^{-1}(B_v -v, \ph{i}{K})
\stackrel{\partial_V}{\rightarrow}
Gr^W_{a+i}\H_c^0(V-\{v\},\ph{i}{K})$  (precisely  $Gr^W_{a+i}
(\partial_V)$) is injective, by considering
the long exact sequence:

\smallskip
\centerline {  $ \H^{-1}( V-\{v\}, \ph{i}{K})
\stackrel{\partial_V}{\rightarrow} H_c^0(V-\{v\},\ph{i}{K}) \to
\H^0(V -v, \ph{i}{K}) $.}

\smallskip 
 \noindent It is enough to prove that  $ H^{-1}( V-\{v\}, \ph{i}{K})$   is pure of weight
 $ a+i-1$ or by  duality $ H^1_c( V-\{v\},
\ph{i}{K})$ is pure of weight $ a+i+1$.  We consider as above the hyperplane section $H_1$ not containing $v$ and the 
  exact sequence:\\
$\H_{H_1}^1 (V, {R k}_{v
!}k_v^* \ph{i}{K}) \xrightarrow{\varphi} \H^1 (V,{R jk}_{v !}k_v^*
\ph{i}{K})
\rightarrow \H^1 (V - H_1 , R k_{v !}k_v^* \ph{i}{K}) $ where  $\varphi$ is surjective since the last term  vanish; the space \\
  $\H^{-1} (H_1,{R k}_{v !}k_v^* \ph{i}{K})(-1)\overset{\sim}{\longrightarrow}  \H_{H_1}^1 (V,{R k}_{v
!}k_v^* \ph{i}{K}) $ \\ is a pure HS of weight  $a+i+1$ (as a sub-quotient of the pure intersection  cohomology  of the inverse image $H'_1$ smooth in $X$).
 \begin{lem}\label{2}
 The following morphism is injective\\
 $ Gr^W_{a+i}\H_c^i(X-X_v, \ilm )
\stackrel {\alpha_X} {\rightarrow}
 Gr^W_{a+i}\H^i (X, \ilm )\overset{\sim}{\longrightarrow}  \H^i(X, \ilm ) $.
 \end{lem}
 In the long  exact sequence
 $ H^{i-1}(X_v,\ilm) \rightarrow \H_c^i(X-X_v,  \ilm)
\stackrel{\alpha_X}{\rightarrow} \H^i(X, \ilm)$,
 the weight $w$  of    $H^{i-1}(X_v,\ilm) $
satisfy  $w < a+ i$ since $X_v$ is closed,  then the morphism $\alpha_X$ (precisely  $Gr^W_{a+i}
(\alpha_X)$) is  injective and  $ H^i(X,
 \ilm)$ is pure of weight $a+i$.
 \subsubsection{}
   The morphisms
 $\partial_X $ and  $\alpha_X$ are compatible with  ${\pt}$ and $W$.
 We still  denote by  $\partial_X$ (resp. $\alpha_X$) their restriction  to
 ${\pt}_r \H^{r -1}(B_{X_v}-X_v,
\ilm )$ (resp. ${\pt}_r \H_c^{r}(X - X_v,  \ilm)$), then the
composed morphism has value in $ \H^{r}(X ,  \ilm)$ ( in fact in
the subspace ${\pt}_r$ but we avoid the filtration $\pt$ as   we have no information at the point $v$ yet).

\begin{cor} $ Gr^W_a Gr^{\pt}_0
\H^{-1}(B_{X_v}-X_v, \ilm )  = 0$. 
\end{cor}
It is enough  to apply
the lemma \ref{ext0} to  $H^{-1}(B_v-v, \ke \, \rho_0)$,  since

  $Gr^W_a Gr^{\pt}_0
\H^{-1}(B_{X_v}-X_v, \ilm ) \simeq Gr^W_a H^{-1}(B_v-v, \ke \, \rho_0 \oplus \im \, G_0)$ 

\n splits ( formula \ref{1+}) and    $Gr^W_a H^{-1}(B_v-v, \im \, G_0) = 0$
by induction. 
 
   By lemma \ref{ext0}, each element $\overline {u} \in Gr^{\pt}_0 Gr^W_a
\H^{-1}(B_v^*, \ke \, \rho_0 ) $ is the class  modulo ${\pt}_{ -1}$ of an element $u \in ({\pt}_0 \cap W_a) \H^{
-1}(B_{X_v}^*,\ilm) $, where $u = \can (b)$ with $b \in ({\pt}_0 \cap W_a)\H^{-1}(B_{X_v}^*, j_{H !}j_H^* \ilm ) $.
    
 We use the commutative diagram to lift the elements as needed
$$\xymatrix{
 W_a \H^{-1}(B_v-v,   \tau_{\leq 0} Rf_*j_{H!}j_H^*\ilm)
 \ar[d]^{ }\ar[r]^{\quad \can}& 
  Gr^W_a H^{-1}(B_v-v, \ke \, \rho_0)
 \ar[d]^{} \\
 W_a  \tau_{\leq 0}\H^{-1}(B_{X_v}-X_v,   j_{H!}j_H^*\ilm)
  \ar[r]^{\quad \can}&
 Gr^W_a   Gr^\tau_0\H^0(B_{X_v} - X_v, \ilm) 
    }$$
  Since $ \gamma (u) = \alpha_X  ( \partial_X u) = 0$, we deduce  $\partial_X u = 0$ in $Gr^W_a\H^0_c(B^*_{X_v}, \ilm )$ by  lemma (\ref {2}), hence  the class $\overline {\partial_X u}= 0$ in $Gr^{\pt}_0 Gr^W_a
\H^0_c(B_{X_v}^*, \ilm ) $.  Finally, since  
 $\partial_X (\overline {u}) = \overline {\partial_X (u)} = 0$,
 we deduce  $  \overline {u} = 0$ by  lemma (\ref {del}). This ends the proof of the crucial case and the proposition \ref{prop} for $f$ a fibration by NCD.
 
 \begin{rem}\label{imp}
The  following relation follows from  the proof of the  proposition \ref{prop}
\begin{equation*}
 \im\, \biggl(\ph {i}{ R {k_v}_ ! { k_v^*K}}  \rightarrow  \ph {i}{ K}\biggr) =
 {k_v}_{!*}k_v^*\ph {i}{ K}
 \end{equation*}
 \end{rem}
 \subsubsection{} 
  The inductive step theorem \ref{P} and most of the corollary \ref{CP} follow for $f$ a fibration by NCD.  We complete now the proof of equation \ref {De3}.
\begin{lem}\label{De4}
On  a projective variety $V$, we have an orthogonal decomposition of a polarized HS of weight $a+i+j$
 \begin{equation*}
Gr^{\pt}_i \H^{i+j}(X,   \ilm) \simeq \H^j(V,  \ph{i}{ R f_ *  \ilm}) \overset{\sim}{\longrightarrow} 
 \oplus^{S_l \subset V_l^*}_{ l \leq n } \,\, \H^j(V,  {i_{S_l}}_{! *}{\LL}^i_{S_l}[l]).
\end{equation*}
\end{lem}
The proof is by induction on $dim.$V and for fixed $V$ on $dim.$X. The case of $dim.$V = $1$
is clear and may be proved by the technique of the inductive step.  We suppose the result true for  $dim.$V = $n-1$. For $dim.$V = $n$, 
by corollary \ref{surj0},
 the case of  $ \ph{i}{ R f_ *  \ilm}$ for $i <0$ follows from the case $\ph{i}{ Rf_* ( i_H)_* i_H^* \ilm} $ by  restriction to a general ample  hyperplane $H$, and by duality the case for $i > 0$. It remains to prove the crucial case  of  $ \ph{0}{ R f_ *  \ilm}$. 
 First we prove that for each non generic component $S_l$ with $l < n$, the component  
 $ \H^j(V,  {i_{S_l}}_{! *}{\LL}^0_{S_l}[l])$ is a sub HS.
 
 We may suppose there exists a projection $\pi: V \to V_l $ to a projective variety $V_l$
 inducing a finite morphism on $\overline {S_l}$. By induction, the property apply for the decomposition of  $\pi \circ f$ as $dim.V_l < n$, which contains  the component $\pi_* {i_{S_l}}_{! *}{\LL}^0_{S_l}[l]$. 
  
 It follows that the direct sum  $  \oplus^{S_l \subset V^*_l}_{ l < n } \,\, \H^j(V,  {i_{S_l}}_{! *}{\LL}^i_{S_l}[l]) \subset Gr^{\pt}_0 \H^{j}(X,   \ilm)$  is a sub-HS.  Then, it is a standard
argument to show that its orthogonal subspace, which coincides with $ \H^j(V,  {i_{S_n}}_{! *}{\LL}^i_{S_n}[l])$ (the generic component) is a sub-HS.  This is  a clarification of the proof for constant coefficients in \cite{DM0} and a generalization to coefficients.
\begin{rem}
For a quasi-projective variety, we must consider MHS. This apply in general for 
coefficients in a mixed Hodge complex generalizing the notion of  variation of MHS. The pure case in this paper is the building bloc.
\end{rem}
 \subsubsection
{ Proof for any projective morphism}\label{gen}
The above results establish the case of fibrations by NCD. Given $\ilm$ on $X$ and a diagram $ X' \overset{\pi}{\longrightarrow}  X \overset{f}{\longrightarrow} V$ such that the desingularization $\pi $  and $f':= f\circ \pi$ are  fibrations by NCD over the strata.
We apply the above result to $\pi $ and the extension $j'_{! *}\LL$ on $X'$ to deduce the decomposition of $K:= R\pi_* j'_{!*}\LL$ on $X$ into a direct sum of intermediate extensions.

We indicate here, how we can extend Hodge theory and deduce the structure of polarized VHS on $\LL^i_S$ in all cases  without reference to the the special fibration case.

Let $\LL$ be defined on a smooth open set $\Omega $ of $X$, $j: \Omega \to X$ and $j':  \Omega \to X'$ an open embedding into a desingularization of $X$.

\begin{lem}\label{polar}
Let  $\pi: X'  \to X$ be a desingularization defined by a fibration by NCD over the strata,  $ j'_{!*}\LL $  the intermediate extension  of $ \LL $ on $X'$. If $X$ is projective, 
$\H^i(X, \ilm )$ is a sub-Hodge structure of $Gr^{\pt}_0 \H^i(X', j'_{!*}\LL)$.
The induced Hodge structure on $\H^i(X, \ilm )$ is independent of the choice of $X'$.
\end{lem}

\begin{proof} We apply the above result to  $\pi: X'  \to X$  to deduce  the decomposition  of $\ph{0}{R\pi_* j'_{!*}\LL} \simeq \oplus_{S \in \SS}\LL^0_S$  on $X$ into a direct sum  consisting of intermediate extensions of polarized VHS. 

\n It follows from lemma \ref{De4}, that the  decomposition 

\centerline {$Gr^{\pt}_0 \H^i(X', j'_{!*}\LL) \simeq  \H^i(X, \ph{0}{R\pi_* j'_{!*}\LL}) \simeq \oplus_{S \in \SS} \H^i(X, {i_S}_{! *}
\LL^0_S)$}

\n is compatible with HS.   Moreover,  on the big strata  $U$ (as we can  suppose $V$ irreducible), we have $\ilm = {i_U}_{! *}\LL^0_U$; from which we deduce a HS  on $\H^i(X, \ilm)$ as a a sub-quotient HS of $\H^i(X', j'_{! *}\LL)$.

The uniqueness is deduced by the method of  comparison of  the two 
desingularizations $X'_1$ and $X'_2$ with a common desingularization $X'$ of the fiber product $X'_1  \times_X   X'_2$.
\end{proof} 
 
\subsubsection{
 Proof  of the corollary \ref{Vpolar} and lemma \ref{De4} for any projective morphism}

We describe the underlying structure  of polarized VHS  on $\LL^i_v$. 

Let $f: X \to V$, $\pi: X' \to X$ such that $\pi$ and $f\circ \pi$ are fibrations, $v \in S$  a general point of a strata $S$ on $V$, and $N_v$ a general normal section to $S$ at $v$ such that the inverse image  $ X_{N_v}$ is normally embedded in $X$ (resp.  $ X'_{N_v}$
in $X'$). By definition the fiber $\LL^i_{S, v}$ is the image of the Intersection morphism: $\H^i_{X_v}(X_{N_v}, \ilm)  \overset{I}{\rightarrow} \H^i(X_v, \ilm) $. 

 By transversality,
 the perverse truncation $\pt_X$ on $K:= R \pi_*  j'_{!*}\LL$ on $X$, is compatible with the restriction of $K$ to $ X_{N_v}$.  Hence, we can suppose $v $ in the zero  dimensional strata.  As the decomposition is established for $\pi$ and $f\circ \pi$, it follows for $f$, and $\H^i_{X_v}(X, \ph{0}{K})  = Gr_0^{\pt_X}\H^i_{X'_v}( X', j'_{!*}\LL) $ carry an induced MHS (see subsection \ref{C6}). 
 
 The natural decomposition on $\ph{0}{K}$ induce a decomposition of  $\H^i_{X_v}(X, \ph{0}{K})$ into direct sum of MHS where $ \H^i_{X_v}(X, \ilm)$ is  a summand, on which the MHS is realized as sub-quotient of the MHS on $\H^i_{X'_v}( X', j'_{!*}\LL) $.

A similar argument puts on $ \H^i (X_v, \ilm)$ a MHS as a sub-quotient of the MHS on $Gr_0^{\pt_X}\H^i (X'_v, j'_{!*}\LL) $ such that the image MHS of the intersection morphism $I$  is pure on $\LL^i_v$.

Verdier  duality between $\LL^{-k}_S$ and $ \LL^k_S$ follows from the auto-duality of $\ilm$
by Verdier formula for the  proper morphism  $ f $. 
Since Hard Lefschetz isomorphisms between $\LL^{-k}_S$ and $ \LL^k_S$
are also satisfied, we  deduce that  the VHS $ \LL^k_S$ are polarized.

Similarly,  the decomposition for
$\ph{i}{Rf_* \ph{0}{R\pi_* j'_{!*}\LL}}$ is part of the decomposition for
$\ph{i}{R(f \circ \pi)_*  j'_{!*}\LL}$ and contains the decomposition of $\ph{i}{Rf_* j_{!*}\LL}$.
{\it This ends the proof of lemma \ref{De4} and altogether
the proof of corollary  \ref{Vpolar}.}
 
\section{Fibration by normal crossing  divisors}\label{St1}
The  proof of the decomposition is reduced to the case of a fibration  by NCD (definition  \ref{St}) over the strata, in which case the proof along a strata is reduced to the zero dimensional strata by intersecting with a normal section,  such that we can 
 rely on logarithmic complexes in all arguments based on Hodge theory.  
 
 The proof of the proposition \ref{RSt} is divided   into many steps,  {\it   first we transform the morphism $f$} and then,   simultaneously,  a desingularization $\pi: X' \to X $. 
 \subsubsection{
 Thom-Whitney stratification}\label{TW}
 Let $f:X\to V$ be an algebraic   map, and   ${\SS} =  (S_\alpha)$ a Whitney stratification  by a family of strata $S_\alpha$ of $V$.  The subspaces
$V_l=\cup_{\dim S_\alpha \leq l} S_\alpha$ form an increasing  family of 
 closed algebraic sub-sets of
$V$  of dimension $ \leq l$, with  index $l \leq n$,  where $n$ is the dimension of $V$, 
 $$V_0 \subset  V_1 \subset \cdots \subset V_{n-1} \subset  V_n = V$$ 
 The inverse   
of a  sub-space $Z\subset V$, is denoted  $X_Z:=f^{-1}(Z)$, in particular 
 $X_{S} = f^{-1}(S)$ for a strata $S$ of $\SS$. 

A  Thom-Whitney stratification of $f$ has the following  properties:
\begin{itemize}
\item (T) Over a strata $S$  of $ V$ the  morphism $f_\vert: X_{S} \to S$ induced by $f$,
 is a  locally trivial topological fibration.
 \item (W) The link   at any point of a strata  is a  locally constant topological invariant of the strata \cite {Ma}, \cite {L-T}.
\end{itemize}
 \begin{lem}
  \label{rec}
 
 Let $f: X \to V $ be a projective morphism, and $Y$ a   strict  closed algebraic subset
  containing   the singularities  of $X$. There exists 
 a commutative diagram:
 \[
\begin{array}{cccccccccc}
 X &  \overset{\pi_1}{\longleftarrow} & X_1    &\cdots 
 & \overset{\pi_i}{\longleftarrow} & X_i&  \cdots & X_k &\overset{\pi_{k+1}}{\longleftarrow} &X_{k+1}\\
 f \downarrow \quad & &  f_1 \downarrow  \quad&   && f_i \downarrow  \quad& & f_k \downarrow  \quad && f_{k+1} \downarrow \quad \\
  V & \overset{id}{\longleftarrow}&\,\, V & \cdots 
  &  \overset{id}{\longleftarrow} &\,\, V&  \cdots &\,\, V  & \overset{id}{\longleftarrow} &\,\, V\\ 
\end{array}
\]
such that $ f_{i+1} := f_i \circ \pi_{i+1} $ and a decreasing sequence $V^i$  for $0 < i \leq k$ of closed  algebraic  subspaces of $V$ of dimension $d_i > 0$  
such that for $j < i$, the  inverse  image $f_i^{-1}(V^j)$  are  NCD in the   non-singular variety $X_i$, 
and 
$f_i^{-1}(V^{j}) -f_i^{-1}(V^{j+1})$ is a relative NCD over $V^j -V^{j+1}$.

 Moreover,  there exists a    Whitney stra\-tification $\SS^i$ of $V$ adapted to $V^j$ 
 for $j \leq i$, satisfying the following   relation: the strata of $\SS^i$ and of  $\SS^{i-1}$ coincide outside  $V^i$.
  
 The morphism $\pi_{i+1}$, obtained by blowing-ups over   
 $ f_i^{-1}(V^i)$ in $ X_i$, transforms $ f_i^{-1}(V^i)$ into a NCD and $\pi_{i+1}$  induces an isomorphism: 
 $$(\pi_{i+1})_\vert: (X_{i+1} -  f_{i+1}^{-1}(V^i)) \overset{\sim}{\longrightarrow}  X_i -f_i^{-1}(V^i).$$
 
 Moreover, the morphisms $\pi_s$  are modifications over  $f_i^{-1}(V^i)$  for all $s \geq i+1$. 
 
 Let $\rho_i:=\pi_1\circ\cdots\circ\pi_i$ for all $i$. The open  subset $\Omega := f^{-1}(V -V^1) \subset X$ is   dense  in $X$,  the restriction of $\rho_i$ to $\rho_i^{-1}(\Omega) \subset X_i$ is an isomorphism, and the restriction of the morphism
$f_i$ to  $\rho_i^{-1}(Y \cap \Omega)$ is a fibration by  relative NCD  
over $V - V^1$.

\smallskip
  Moreover, we can suppose the family of subspaces $V^i$  maximal: the dimension of $V^i$ is $n-i$ for $i > 0$ and $ k+1 = n$.

We refer to  the  morphism $f_{k+1}:X_{k+1} \to V$ (resp. diagram)   as an admissible fibration with respect to the family $V_i$.
   \end{lem}
 
 We can always  suppose $V = f(X)$.  
Let $\pi_1: X_1 \to  X $ be a desingularisation morphism of $X$ such that $ \pi_1^{-1}(Y)$, 
 as well the inverse image of the irreductible components of $Y$, are NCD in $X_1$. 
 Let $f_1:= f \circ \pi_1$ with $X_1$ smooth;
 there  exists an open subset 
 $U \subset V$ such that the restriction of $f_1$ over $U$ is smooth.
 
 If the dimension of $ f(Y)$ is strictly smaller than  
 $n :=\dim V$,  let   
 $U \subset V - f(Y)$ over which the  restriction of $f_1$ is smooth and let  $V^1:= V - U$. 
 
 In the case $f(Y) = V$,   we suppose moreover 
  the restriction of $f_1$ to $\pi^{-1}(Y) \cap f_1^{-1}(U)$ is a fibration by  relative NCD over $U$ and let  $V^1 := V - U$. 
  We remark that  $d_1 := \, \dim \, V^1$  is strictly smaller than   $\dim V = n$ and we can always choose $U$, hence $V^1$, such that $d_1 = n-1$. 
 
   We consider a  Thom-Whitney stratification  of the morphism $f_1:X_1\to V$; in particular  the image by $f_1$ of a strata of  $X_1$ is a strata of $V$. We suppose also that the stratification is compatible with the divisor $ \pi_1^{-1}(Y)$ and
 let $\SS^0$ denotes such  stratification of $V$.
 
We construct now the algebraic sub-spaces $V^i$ of $V$, the morphisms $f_i= X_i \to V$,  and $\pi_i: X_i \to X_{i-1} $ in the lemma. 

  Let  $\SS^1$ be a stratification compatible with $V^1$ and a refinement of $\SS^0$. 
Let $D^1:= f_1^{-1}(V^1)$; we  introduce the horizontal divisor: $Y^1_h := \overline { ( \pi_1^{-1}(Y)  \cap f_1^{-1}(V - V^1)}$.
   
We construct $\pi_2:X_2 \to X_1$ by blowings-up over  $D^1 \cup Y^1_h$, without modification of $X_1- D^1$, such that the inverse  image $ D^2 := \pi_2^{-1}(D^1)$
of $D^1$ and the union  $D^2 \cup Y^2_h$  with $ Y_h^2 := \pi_2^{-1}(Y_h^1)$ are NCD in the smooth variety $X_2$.

 Let $f_2:= f_1\circ \pi_2$, hence $ D^2 := f_2^{-1}(V^1)$,
 the next argument   allow us to construct a sub-space $V^2$  such that  $D^2 \cap f_2^{-1}(V^1-V^2)$
 is a relative NCD over  $V^1-V^2$.
  
  \begin{lem}\label{dense}
  Let $f: X \to Z$ be a projective morphism on a   non-singular space $X$, $T$ an algebraic sub-space 
  of $Z$  such that $D:= f^{-1} (T)$  is a NCD in $X$
as well the inverse image of each irreductible  component of $T$.

  Then, there  exists a non singular  algebraic subset $S_0$ in  $T$, such that the  dimension of $T - S_0$ is  strictly smaller than $\dim T$, over which  $f$ is a relative NCD.
  \end{lem}

\noindent {\bf Proof}. i)  Let $D_{\beta_i}$ denote the components of the divisor $D$ and $D_{\beta_1,\ldots,\beta_j}$ the intersection of $D_{\beta_1}$, $\ldots, D_{\beta_j}$. 
Let $ d := \dim T$ be the  dimension of $T$; there  exists an open subset    $S_{\beta_1,\ldots,\beta_j}$ complementary of an algebraic  sub-space  of $T$ of dimension
strictly smaller than  $ d$, such that $f^{-1}(S_{\beta_1,\ldots,\beta_j}) \cap D_{\beta_1,\ldots,\beta_j}$ is either empty, or a  topological fibration over $ S_{\beta_1,\ldots,\beta_j}$.
Over the open subset $S'_0:=\cap_{\beta_1,\ldots,\beta_j}S_{\beta_1,\ldots,\beta_j}$, the divisor $D \cap f^{-1}(S'_0)$ is a  topological fibration, which is needed to satisfy the assertion (2) of the definition \ref{St}.

ii) Still, we need to check the assertion (3) of the definition \ref{St}. Let $\overline{S_1}$
be an irreductible  component of   $T$ of dimension $d$. There exists a dense open subset $S'_1$ in $\overline{S_1}$ with inverse image a sub-NCD of $D$, by the above hypothesis.
As the argument is  local, we consider an open affine subset $\mathcal U$ of $Z$  with non empty intersection  $S_1 := \overline{S}_1\cap {\mathcal U} \subset S'_1$  and a 
 projection $q:  {\mathcal U}\to{\mathbb C}^{d}$
   whose restriction to $S_1 $ is a finite projection. 
 There  exists an open affine dense subset  $U_1$ 
of ${\mathbb C}^{d}$ such that
$q\circ f$ induces a smooth  morphism over $U_1$.

Considering $S'_0$ in i) above, we notice that $f$ induces over  $q^{-1}(U_1)\cap S'_0$  a fibration by NCD: indeed,
let $x$ be a point in $q^{-1}(U_1)\cap S'_0 \cap S_1$. 
There  exists an open neighborhood ${\mathcal U}_x $ 
of  $x$ in $(q^{-1}(U_1)\cap S'_0) \subset  {\mathcal U}$, small enough such that the
 restriction of $q$  to ${\mathcal U}_x\cap S_1$ is non ramified on its  image.

 On the other hand, for all $y \in  {\mathcal U}_x $, we have:  $f^{-1}(q^{-1}(q(y)) \cap {\mathcal U}_x)$ is smooth, since
$q(y)\in U_1$. The dimension of $f^{-1}(q^{-1}(q(y)\cap {\mathcal U}_x)$ is $\dim X-d$. As
$x$ is in $S'_0$ the dimension of $f^{-1}(y)\cap D$ is $ \dim X-1-d_1$ and it is a divisor with
 normal crossings in   $f^{-1}(q^{-1}(q(y)\cap {\mathcal U}_x)$ of dimension: $ \dim X- d$.

We remark that $q^{-1}(q(x))\cap {\mathcal U}_x$ is a normal  section of $U_1\cap S'_0$
at $x$ in $V$. Hence $f$ induces on  $q^{-1}(U_1)\cap S'_0$  a fibration by NCD in the normal sections.

 The open  subset $S_0$ is the union of open  subsets   $q^{-1}(U_1)\cap S'_0$ for the various   irreductible components  $\overline {S_1}$ of $T$ of maximal dimension  $d$, which proves the assertion (3) 
 of the  definition and ends the proof of the  lemma \ref{dense}.

\vskip.1in
{\it End of the proof}. Let $d_1$ be the dimension of $V^1$. After the lemma \ref{dense} there exists an open algebraic subset $S_0$ of $V^1$, over which the restriction 
of  $f_2$ induces a  relative NCD  over $S_0$, moreover, the dimension $d_2$ of the algebraic set $V^2:=V^1-S_0$ is   $d_2 < d_1$ strictly less than $d_1$.
 If $d_1 = n - 1$, we can always  choose $S_0$ such that $d_2 = n-2$.  

We define  a new  Whitney stratification $\SS^2$ of $V$, compatible with $V^2$
 which coincide with $\SS^1$ away of $V^1$, by keeping the same strata 
away of 
  $V^1$ and adding   Thom-Whitney strata including 
  the connected  components of $S_0$ just defined
  and strata in their   complement $V^2$ in $V^1$. 
   
  We complete the proof by repeating this argument 
for the closed algebraic sub-space  
   $V^2$ of $V^1$.    This process is the beginning of an inductive  argument as follows.  

\vskip.1in 
{\it  Hypothesis of the inductive  argument:} Given projective morphisms $f_j: X_j  \to V $ for $1 \leq j \leq i $ as in the diagram above, and $\pi_j: X_j \to X_{j-1}$, $f_j := f_{j-1} \circ \pi_j$
where $X_j$ is non-singular,
 a Thom-Whitney  stratification $\SS^i$ of $V$ compatible with a family of algebraic sub-spaces  $V^j \subset\cdots \subset  V^{1}\subset V$ for $j \leq i $   such that there exists a stratification
 associated to $X_i$, stratifying $f_i$,   and such that $f_{i-1}^{-1}(V^j)$ are NCD in $X_i$, as well the inverse image of each   irreductible component of $V^j$   
for $1 \leq j < i$, and moreover $f_i$ induces an admissible morphism over  $V-V^i$.  Also, let $Y_h^i$ be a  divisor of  $X_i$ such that  $f_i^{-1}(V^1) \cup Y_h^i$ is a NCD and $Y_h^i 
 \cap (X_i- f_i^{-1}(V^1)) $ is a  relative NCD.
 
\vskip.1in 
 A sequence of blowings-up centered over $f_i^{-1}(V^i)$  leads to the construction of 
 $\pi_{i+1}:X_{i+1}\to X_i$
  such that $X_{i+1}$ is non-singular and the inverse images of  $ V^j$ by
 $f_i \circ \pi_{i+1}$,   as well its irreductible components for $1\leq j\leq i$, 
 and their  union with  $ Y_h^{i+1} := \pi_{i+1}^{-1}(Y_h^i)$  are NCD in  $X_{i+1}$.
 
  Let $f_{i+1} := f_i \circ \pi_{i+1}$.
The  stratification $\SS^i$ of $V$ underlies a stratification of $f_{i+1} $. 
 It follows from lemma \ref{dense} that in each  maximal strata $S$ of $V^i$ 
 in $\SS^i$, there exists an open dense subset $S^i_0(S)$ 
 over which $\pi_{i+1}$
 is  a relative NCD. Let  $V^i_0$ be the union of all $S^i_0$. The complement $V^i-V^i_0$ of $V^i_0$ is a closed algebraic strict sub-space   $V^{i+1}$ of $V^i$, and $f_{i+1}$ is admissible 
 over $V - V^{i+1}$. 
We construct a  refinement of the stratification $\SS^i$ and then  a Thom-Whitney stratification of $\SS^{i+1}$ compatible with $V^{i+1}$,
 keeping the same  strata away from  $V^i$ and  introducing as new strata the connected components of the open subset $V^i_0$  of $V^i$,  then completeing by a Thom-Whitney    stratification of the  
 complement in  $V^i$.
 
The inductive  argument ends when  $V^{k+1}=\emptyset$, which occurs after a  finite number of  steps since  the family $(V^i)$ is decreasing. 
\begin{rem}\label{c} 
Let $d_i :=$ dim.$V^i$,  we can suppose $V_{d_i} = V^i$ of dimension $d_i$ for $0 < i \leq k$. 
 \end{rem}
\begin{cor}
i)  In the  lemma  \ref{rec}, the morphism $f_{k+1}$  is a  fibration by NCD over the strata, moreover we can  suppose $k = n$.

ii) For any $X$, there exists a modification $\pi: X' \to X$ which is a fibration by NCD over the 
the strata, and  transform an algebraic  sub-space  $Y $ containing the singularities of $X$
into a divisor with normal crossings  $Y'  := \pi^{-1}( Y)$.

iii)  We can  suppose in the preceding  diagram that each morphism $\pi_i$ is a
 fibration by NCD over the  strata. 
\end{cor}

   The assertion  i) is clear.
   In the case where $f$ is the identity of   $X$ in  the lemma \ref{rec}, we construct a modification $X'$ of $X$ compatible with $Y$, which is a  fibration by relative NCD over $X$, which prove ii).
By the same argument, we can suppose all $\pi_i$ in the  lemma admissible which prove iii).
\begin{proof}[ Proof of the proposition \ref{RSt}]\label{pr}
Going back to the inductive argument for $f$ in the lemma, we apply at each step of the induction the assertion (ii) of the  corollary, in particular  we can start with the  desingularization by an admissible modification. 

 \smallskip
 \n \emph{ Hypothesis of the induction.} We suppose there  exists:\\
 1)  A diagram of morphisms
 $D_i$:
   $$ X   \overset{\pi'_i}{\leftarrow}  X_i \overset{f_i}{\rightarrow}  V, \quad f_i = f \circ \pi'_i$$
 where  $\pi'_i$ is admissible. \\
2) A decreasing  family of  algebraic sub-spaces  $V^i \subset V^j $   for  $j <  i$ with   inverse image  $ f_i^{-1}(V^j)$  consisting of  NCD in $X_i$ for  $j < i$, and a  stratification $\SS^i$ of $V$ compatible with the family  $V^j$, such that the restriction of $f_i$
to $X_i - f_i^{-1}(V^i)$ over  $V- V^i$  is a  fibration by relative NCD   over the strata.  

\smallskip
 \n \emph{Inductive step.} Let $d_i$ (resp. $n-i$) be the dimension of $V^i$. We want to define   a sub-space $V^{i+1} \subset V^i$ of dimension  strictly smaller $d_{i+1} < d_i $ (resp. $n- i-1 $) and to extend the diagram over the open subset    $V^i - V^{i+1}$, that is, to construct a diagram of morphisms $D_{i+1}$: 
   $$ X   \overset{\pi'_{i+1}}{\leftarrow}  X_{i+1}  \overset{f_{i+1}}{\rightarrow}  V, \quad f_{i+1} = f \circ \pi'_{i+1}$$
       such that:
 
 1)    $\pi'_{i+1}: X_{i+1}\overset{\pi_{i+1}}{\rightarrow} X_i \overset{\pi'_i}{\rightarrow} X$ is admissible and defined as a composition of  $ \pi'_i: X_i {\rightarrow} X$ with
 a  modification $ \pi_{i+1}$ inducing an  isomorphism: $ X_{i+1} - f_{i+1}^{-1} (V^i) \overset{\sim}{\longrightarrow}  X_i - f_i^{-1} (V^i)$.
       
     2)  $f_{i+1}^{-1} (V^i)$ is a relative NCD over the open subset  $V^i - V^{i+1}$. 

\smallskip  
To achieve this step, we apply a slightly modified   version of the lemma \ref{rec}.
 \begin{lem}
[Relative case]\label{relatif}  
 Let $f: X {\rightarrow} V $ be a projective  morphism and $Z$ a strict algebraic  sub-space  of $V$ of dimension $\ell$.
 
  With the notations of the definition \ref{St},  there exists an  admissible diagram in the  sense of the  lemma \ref{rec} and an  index $i$ such that $ Z \subset V^i$, dim.$V^i =$ dim. $Z$,  such that $Z' := f_i^{-1}(Z)$  is a   NCD in $X_i$  relative  over $V^i - V^{i+1}$.
  
 Then, we may construct  a stratification of $f$, such that
 $Z$ is a  sub-space of $V_{\ell}$, union of the strata of  dimension $\leq {\ell}$, and 
 $Z' := f_{\ell}^{-1}(Z)$  is a NCD  in $X_{\ell}$  relative over  $V_{\ell} - V_{\ell - 1}$. 
Hence, there exists a diagram
 $$X \xleftarrow{\pi'} X' \xrightarrow{f'}  V, \quad  f' := f \circ \pi'  $$
such that $f'^{-1} (Z) = \pi'^{-1} (f^{-1} (Z))$ is a  relative NCD over  $Z - (Z \cap V_{\ell - 1}) $.
 \end{lem}
\begin{proof} Indeed, in the preceding  construction, we can suppose that all  sub-spaces $V^i$
of dimension $d_i \geq \ell$ contain  $Z$. Let  $V^k$ and $V^{k+1}$ such that $d_k > \ell \geq d_{k+1}$, then we choose   stratifications $\SS^i$ of $V$  compatible with $Z$   and replace $V^{k+1}$ by $V^{k+1} \cup Z$, hence we are obliged to replace all $V^{i}$ for $i \leq k+1$ by the sub-spaces $V^i \cup Z  \supset V^{k+1} \cup Z$, then we  construct $\pi_{\ell}: X^{\ell} \to X^k$ such that $\pi_{\ell}^{-1} (f_k^{-1}(V^{k+1} \cup Z)) \subset  X^{\ell}$ is a NCD. 
\end{proof} 
We remark that the variety $f_i^{-1} (V^i)$ in $X_i$ is over  $ f^{-1} (V^i)$  in $X$.  Then, we  apply the lemma  (\ref{relatif} ) for  $\pi'_i: X_i \to X $  instead of
    $ f: X \to V$  in the lemma and  $f^{-1} (V^i) \subset X$  instead of $Z \subset V$,
 to construct the admissible morphism 
    $\pi'_{i+1}$
  \smallskip
  \centerline{$\pi'_i: X_i \to X,  \quad f^{-1} (V^i) \subset X, \quad  X_i \xleftarrow{\pi_{i+1}} X_{i+1}  \xrightarrow{\pi'_{i+1}} X,  \quad \pi'_{i+1} := \pi'_i \circ \pi_{i+1}$}
      \smallskip
    \n   That is we  develop the constructions of the lemma  \ref{rec} to  construct $\pi'_{i+1}$ over $X$
  {\it by  modification only of sub-spaces over $f_i^{-1} (V^i)$} to transform  $f_i^{-1} (V^i)$ into a NCD, hence $X_{i+1}$ in the  diagram $D_{i+1}$ differs
  from $X_i$ in the diagram $D_i$ only over $f_i^{-1} (V^i)$.
  
    At this stage, we go back to the construction in the  lemma to extend $f_i$ over an open subset  $\Omega$ of $V^i$ such that 
  $V^{i+1} := V^i - \Omega$ is of dimension smaller than $d_i$ and we  define $f_{i+1}: f_i \circ \pi_{i+1}$, then we can complete the diagram $\DD_{i+1}$ and the inductive step.
 At the end we  define $\pi'_m$ and $f_m$  both admissible for  some index $m$.
   \end{proof}     
 \begin{cor}  The decomposition theorem for  $f $  can be  deduced from both cases  $\pi'$
 and $f'$ in the proposition \ref{RSt}. 
  \end{cor} 
  Let $\LL$ be a local system on $X-Y$, there exists an  open    algebraic set $\Omega \subset X-Y$ dense in $X$ such that 
  $ \Omega':= \pi'^{-1}(\Omega) \subset X'$ is isomorphic to $\Omega$,  which carry the   local system $\LL$. Let $j': \Omega' \to X'$, then   the  decomposition theorem  for  $j_{!*}\LL$ 
  with respect to the orginal proper algebraic morphism  $f$ follows from both cases of $f'$ and $\pi'$ (\cite{DL} proposition 2.16).
     \begin{rem} 
 In the  construction of a relative NCD  $X_{V_i}$  overe $V_i - V_{i-1}$, for more clarity we ask for  both conditions:
 
i) The restriction of $f$ to $ X_S:= f^{-1}(S)$ over each strata $S$ of $V_i - V_{i-1}$ is a topological  fibration: $f_{\vert}: X_S \to S$.
 
ii) For each point $v \in  \subset  V_i - V_{i-1}$ smooth in $V_i$, let $N_v$ be a normal section  to $V_i$ at $v$ in general position, then $f^{-1}(N_v)$ is smooth in $X$ and its intersection with the NCD    $X_{V_i} $ is  transversal.

These two conditions may be equivalent.
    \end{rem}  
\section
{Logarithmic complexes}\label{4}
We develop    the construction of Hodge theory by logarithmic complexes with coefficients in an admissible graded  polarized variation of mixed Hodge structure (VMHS): $(\LL, W, F)$ with  singularities along a normal crossing divisor $Y$. We refer to \cite {K} and \cite{EY} for basic computations and to \cite{K, C, S-Z} for admissibility.   

The admissibility  on $X-Y$ refers to  asymptotic properties of   $(\LL, W, F)$ along the NCD.
Such asymptotic properties are expressed on Deligne's  extension $\LL_X$ defined in terms of the
  ''multivalued'' horizontal sections of $\nabla$ on $X-Y$, on which
   the connection $ \nabla$ extends on $\LL_X$ with logarithmic singularities along $Y$.
 
 The extension $\LL_X$ is a  locally free analytic sheaf of modules, hence  algebraic if $X$ is projective. It is  uniquely characterized by the  residues of the logarithmic singularities  of   the connexion $\nabla$, and
defined in terms of a  choice of the logarithm of the eigenvalues of the monodromy  (\cite{H} Th\'eorme d'existence Proposition 5.2). 

The fibre of the vector bundle $\LL(x):= \LL_{X,x}\otimes_{\OO_{X,x}} \C$ is  viewed as the space of the '' multivalued''  horizontal sections  of $\LL$ at $x$ (sections of a universal covering of the complementary of $Y$ in a ball $B_x$ at $x$).  

The extension over $Y$ of the Hodge filtration of a polarizable variation of HS by Schmid \cite{Sc, G-S, CKS} is a fundamental asymptotic property,
that is required  by assumption for a graded polarizable variation of MHS as a condition  to admissibility .

 The local monodromy $T_i$ around a component $Y_i$ of $Y$, defines a nilpotent endomorphism $N_i := Log \,T_i^u$ logarithm of the unipotent part of $T$ preserving the extension of the  filtration $W$ of $\LL$ by sub-bundles $\WW_X \subset \LL_X$. 

Deligne pointed out the problem of the existence  of the relative monodromy filtration $M( \sum_i N_i, W))$  in (\cite{WII}, I.8.15). The required properties are proved in the case of geometric variation of MHS over a punctured disc in \cite{E} and  studied axiomatically as conditions of admissibility in \cite{S-Z}. The  definition of admissibility in \cite {K} along a NCD is by reduction to the case of a punctured disc.

We construct below the weight filtration directly on the logarithmic complex (\ref{log}), generalizing the case of constant coefficients in \cite{HII}. 
 
 We remark that the construction of the Intersection complex (\ref{note}), the intermediate extension of a local system on $X-Y$, as well the development of mixed  Hodge theory, involve the behaviour at ``infinity'', along the NCD $Y$.
 
\subsubsection {Notations}\label{ls} 
 Let  $Y := \cup_{i\in I} Y_i $ be a NCD, union of smooth irreducible components with index $I$, and for $J
\subset I$, set $ Y_J:= \cap_{i\in J} Y_i$, $ Y_J^*:= Y_J -
\cup_{i\notin J} (Y_i \cap Y_J)$ ($ Y_{\emptyset}^*:= X - Y$). We denote uniformly the various embeddings by $j:Y_J^* \to X$. 

The local system $\LL$ on $X^*:= X-Y$ is defined by a connection  $\nabla$ on the fibre bundle  $\LL_{X^*} = \LL \otimes_\C \OO_{X^*} $ with
horizontal sections given by  $\LL$. The  extension of $\LL_{X^*}$ with a regular singular connection
is a couple consisting of a fibre bundle  $\LL_X$ and a connection $\nabla: \LL_X \rightarrow {\Omega}^1_X (Log Y) \otimes 
\LL_X $ (\cite{H}, \cite{M} definition 3.1).
The residue of $\nabla$ is defined along a component  $Y_i$  of  the NCD $Y$ as an endomorphism of the restriction
$ \hbox{Res} _{Y_i} \nabla: {\LL}_{Y_i} \rightarrow {\LL}_{Y_i}$,
 of $\LL_X$ to $Y_i$. 
 
  The eigenvalues of the  residue are constant along a connected component of $Y_j$ and related to the local monodromy $T_j$ of $\LL$ at a general point of $Y_j$ by the formula: $   Log \, T_j  =  -2i\pi  \hbox{Res} _{Y_j} \nabla$ (\cite{H}, theorem 1.17, proposition 3.11). 

 The construction of $\LL_X$ is local.  Deligne's idea is to fix the choice of the residues of the connection    by 
the condition that the eigenvalues of  the  residue  belong  to  the image of a section  of the projection $\C \to \C/ \Z$, determined by fixing the real part of $z: \,  n \leq \RR(z) < n + 1$, hence fixing the determination of the logarithm  $Log: \C^* \rightarrow \C$, and forcing  the uniqueness of the construction. Hence the local constructions glue into  a global bundle.  In Deligne's extension case  $n = 0$.  

Let $X(x) \overset{\sim}{\longrightarrow}  D^{n+l}$ be a neighborhood of a point $x$ in $Y$ isomorphic to a
product of complex discs,
such that $X(x)^* = X(x) \cap (X-Y) \overset{\sim}{\longrightarrow}  {(D^*)}^n \times D^l$ where $D^*$ is
the disc with the origin deleted. 
For simplification,  we  often set $l = 0$, then $X := D^n$ is a ball in $  \C^n$ with  $Y$ defined by $y_1\cdots y_n =0$.
The fundamental group $\pi_1 (X(x)^*)$ is a free abelian group
generated by $n$ elements representing  classes of closed paths
around the hypersurface $Y_i$,
defined locally by the equation $y_i=0$, at a general point of $Y_i$, one for each index $i$.

 The restriction of the local system $\LL$ to $X(x)^*$ corresponds to a representation of $
{\pi}_1(X(x)^*)$ in a vector space  $L$, hence to the action of commuting automorphisms
$T_i$ of $L$  for $i \in [1,n ]$ indexed by the local components $Y_i$ of $Y$ and called
local monodromy action around $Y_i$. 

Classically $L$ is viewed as the fibre of $\LL$ at the
base point of the fundamental group $\pi_1 (X^*)$, however to
represent the fibre  of  Deligne's extended bundle at $x$, we view
{\it  $L$ as the vector space    of multivalued horizontal sections of
$\LL$}, that is the sections of the inverse of $\LL$  on a universal cover $\pi: \widetilde D^n  \rightarrow D^n := X(x)^*$ defined for $z = (z_1,\cdots, z_n) \in \widetilde D^n \subset \C^n$ by $y = \pi(z)$ with components $  y_j := e^{2i\pi z_j}$, then $
L := H^0( \widetilde D^n, \pi^{-1}\LL)$.

The automorphisms $T_i$  are defined over $\Z$ and decompose as a product
of  semi-simple $T_i^s$  and  unipotent $T_i^u$  commuting automorphisms $T_i = T_i^s T_i^u$. On the complex
 vector space $L$, $T^s_i$ is diagonalizable and  represented over $\C$  by the diagonal matrix of its eigenvalues.
    The logarithm of $T_i$ is  defined as the  sum
$
Log \, T_i  =  Log \, T_i^s + Log \, T_i^u =   D_i + N_i 
$, 
 where  $  D_i = Log \,T_i^s$ is diagonalizable  over $\C$
with entries $Log \,{a}_i$  on the diagonal for all
eigenvalues ${a}_i$ of $T_i^s$ and for a fixed determination of $Log$ on
$ {\C}^* $, while $  N_i := Log \, T_i^u$ is a nilpotent endomorphism, defined by  $ N_i = - {\Sigma}_{k \geq 1} (1/ k) {(I-T^u_i)}^k $ as a polynomial function of the nilpotent morphisms
$(I - T^u_i)$, where the sum is finite. At a point $x$ on $\cup_{i\in J \subset I}Y_i$ a product of  closed paths corresponds to a sum of various
$N_i$.

 Locally, at a point $x \in Y$ on the intersection of $n$-components $Y_i, i \in [1,n]$, 
 a spectral  decomposition into a finite  direct sum is defined on $L := \LL(x)$ with index 
 sequences   of eigenvalues $(a.)$,  with one component $a_i$ for
each $T_i$

\centerline {$L = {\oplus }_{(a .)} L^{a .}, \quad  L^{a .}  = \cap_{i \in [1,n ]}
({\cup}_{j>0} \; \hbox{ker} \; {(T_i - {a}_i I)}^j) 
$.}
\subsubsection{} \label{tilda}
For a detailed description of $\LL_X$ at $x \in Y$, let $\alpha_j \in [0,1 [$ for $ j \in [1,n]$ such that $ e^{- 2i \pi {\alpha}_j} = a_j $ is an eigenvalue for $T_j$,  and $L $ the vector space of multivalued sections of $\LL$ on $X(x)^*$, then 
the fiber $ \LL_{X,x}$ is generated as an $\OO_{X,x}$-submodule of $(j_*\LL_{X^*})_x$ by   the image of the embedding of the space   $ L $ of multivalued horizontal sections into $\LL_{X,x} $ by the correspondence $ v \to  \widetilde v  $, defined for $y$ near $x$  as 
\begin{equation*}
\tilde{v} (y)  = (exp (\Sigma_{j \in J}  (log y_{j})({\alpha}_{j} -\frac{1}{2i\pi} N_{j}))). v =
{\Pi}_{j \in J} {y_{j}}^{{\alpha}_{j}} \, \hbox{exp} ({\Sigma}_{j \in J} -\frac{1}{2i\pi}(log  y_{j})  N_{j}). v,
\end{equation*}
 in (\cite{H}, 5.2.1- 5.2.3), then $\tilde{v} (y)$ is a uniform analytic section on $X(x)^*$. 
 It is important to stress that a basis $v_a$ of $L$ is sent onto a basis $\tilde {v}_a $ of
${\LL}_{X,x}$, and if $X$ is projective ${\LL}_{X}$ is an algebraic bundle by Serre's general correspondence. In the text, we omit in general the analytic notations with $X^{an}$ as our applications to the proof of the decomposition are on projective varieties. The action of $N_i$ on $L$   determines 
  the connection (\cite{H}, theorem 1.17, proposition 3.11) as  
 \begin{equation*}
 \nabla \tilde {v} = \Sigma_{j \in J} [\widetilde{(\alpha_j v) }  -\frac{1}{2i\pi}
\widetilde {(N_j v)} ] \otimes \frac {dy_j}{y_j}.
\end{equation*}
\subsection
 {The logarithmic complex  $\Omega^* \LL := \Omega^*_X(Log Y)\otimes \LL_X$}\label{log} 

\

Let $\LL$ be a  local system (shifted by $n$)
on the complement $U$ of the NCD  $Y$ in a smooth complex  algebraic variety $X$,  and  $(\LL_X, \nabla)$ Deligne's extension of $\LL \otimes \OO_U$ with logarithmic singularities.   The connection $\nabla:\LL_X \to \Omega^1_X(Log Y)\otimes \LL_X$ extends naturally into a complex, called the (shifted by $n$)
logarithmic complex  $\Omega^*_X(Log Y)\otimes \LL_X$ .

When $\LL$ underlies a  variation of MHS $(\LL, W, F)$, 
 the filtration by sub-local systems $W$ of $\LL$
extends  as a filtration by canonical  sub-analytic bundles $ \WW_X
\subset \LL_X$. By the condition of admissibility the
filtration $\FF_U$ extends by sub-bundles  $\FF_X \subset \LL_X$. Both $ \WW_X$ and  $\FF_X$
are combined here to define the structure of mixed Hodge complex.

 \begin{thm}
 Let $\LL$ be a shifted admissible graded polarized variation of MHS on $X-Y$. There exists a weight filtration  $W$ on  the logarithmic complex with coefficients  $\LL_X$
  by perverse sheaves,
 and a Hodge  filtration $F$ by  complexes of analytic sub-sheaves  such that the bi-filtered complex
\begin{equation}\label{log}
\Omega^* \LL := (\Omega^*_X(Log Y)\otimes \LL_X, W, F)
\end{equation} 
underly a structure of mixed Hodge complex and induces
  a canonical MHS on the cohomology
 groups $\H^i(X-Y, \LL)$.
\end{thm}

 The weight is defined by  constructible sub-complexes, although it consists  in each degree, of analytic sub-sheaves
of $ \Omega^i_X (Log Y) \otimes\LL_X$. 

The filtration $F$ is  classically  deduced on the
logarithmic complex from the  sub-bundles $\FF^p_X $ in $ \LL_X$ satisfying Griffith's transversality:
 \[
 F^p  = 0 \rightarrow
\FF^p{\mathcal L}_X \cdots \rightarrow
 \Omega^i_X (Log Y) \otimes
\FF^{p-i}{\LL}_X \rightarrow \cdots  \]
{\it In the rest of this section the  direct definition of the weight filtration $W$
 as well its properties in the case of a NCD is based on  the local study in \cite {K} and  \cite{EY}}.
\subsection 
{The  direct image $Rj_* \LL \simeq  \Omega^*\LL $}
To represent the complex  $Rj_* \LL\otimes \C$ in the derived category, we use  its de Rham realization $ \Omega^*\LL := \Omega_X^*(Log Y) \otimes \LL_X $. Indeed, the quasi-isomorphism  $Rj_* \LL\overset{\sim}{\longrightarrow}   \Omega^*\LL$
follows from   Grothendieck's algebraic de Rham cohomology   \cite{Gr} and its  generalization  to   local systems by Deligne (\cite{H}, definition 3.1).

We also describe a sub-complex $IC^*\LL$ representing the intermediate extension \cite{G-M,  B, BBD}. Various related definitions  given here in terms of  local coordinates are  independent of  the choice of coordinates. This approach  is fit for calculus.

\subsubsection{The (higher) direct image $Rj_{*}{\LL} $}
 The residue of the connection $ \nabla$ on the analytic restriction ${\LL}_{Y_i}$ of $\LL_X$   decomposes into  Jordan sum ${D}_i -\frac{1}{2i\pi} {N}_i$ where ${N}_i$ is  nilpotent and ${D}_i$ diagonal with eigenvalues $\alpha_i \in \lbrack 0,1 \lbrack$ such that the eigenvalues of the monodromy $T_j$ are   $ a_j = \e(\alpha_j) := e^{- 2i \pi {\alpha}_j} $.
The   de Rham complex with coefficients  ${\LL}_X$, is quasi-isomorphic to  $R j_* {\LL} $ (\cite{H}, section II.3):
\begin{equation}
R j_* {\LL} \, \overset{\sim}{\longrightarrow}   \,  \Omega^*\LL \,
:= \, \Omega^*_X (Log Y) \otimes {\LL}_X 
\end{equation}
 In the local situation (\ref{ls})  with  $Y$ defined near a point $x$ at the origin in
$D^{n+l}$ by $y_1\cdots y_n =0$, the fiber of the
complex $ R j_* \LL $ 
 is quasi-isomorphic to a Koszul complex.
 We associate to the component of the spectral decomposition,  a strict simplicial
vector space
 $(L^{\e(\alpha.)}, \alpha_i Id -\frac{1}{2i\pi} N_i), i \in [1,n ])$
 such that for all sequences $(i.) = (i_1 < \cdots < i_p)$: 
 $L(\alpha., i.) = L^{\e(\alpha.)}, \,  \alpha_{i_j}Id -\frac{1}{2i\pi} N_{i_j} \colon L(\alpha., i. - i_j)
\rightarrow L(\alpha., i.)$.

 The Koszul complex is the sum   of  this simplicial vector space; it is denoted by  $s(L^{\e(\alpha.)}, \alpha_i Id -\frac{1}{2i\pi} N_i)_{i \in [1,n ]}$. 
\begin {defn} 
   The direct sum of the  complexes $s(L^{\e(\alpha.)}, \alpha_j Id -\frac{1}{2i\pi} N_j)_{j \subset [1,n ]}$  over all sequences $(\alpha.)$ is   denoted by
$s(L, \alpha. Id -\frac{1}{2i\pi} N.) $.

\n It  is also denoted  as an exterior algebra  
\begin{equation}\label{dec}
\Omega^*L := \Omega (L, \alpha . Id -\frac{1}{2i\pi} N.)  =    
{\oplus}_{\alpha .} \Omega (L^{\e(\alpha .)},  \alpha_i Id -\frac{1}{2i\pi} N_i), i \in [1,n ]).
\end{equation}
where $ \e(\alpha_j) = e^{- 2i \pi {\alpha}_j} = a_j $ is an eigenvalue for $T_j$. 
 \end {defn}
\subsubsection{The tilda emmbedding}\label{tilda}
For $M \subset I$ of length $n:= \vert M \vert $ and  $x \in Y^{*}_M$,  {\it the
 above correspondence $v \mapsto \tilde {v} $, from $L$ to
${\LL }_{X,x}$,  extends to}:
 
 \smallskip
 \centerline {$ L(i_1, \ldots, i_j) \to
({\Omega}^*_X (Log Y) \otimes {\LL}_X)_x \,$ by $\, v \mapsto
\tilde {v} \frac {dy_{i_1}}{y_{i_1}}\wedge \ldots \wedge  \frac
{dy_{i_j}}{y_{i_j}}$.}

 \smallskip
 It induces quasi-isomorphisms
 \begin{equation}
(R j_* {\mathcal L})_x \cong ({\Omega}^*_X (Log Y) \otimes {\LL}_X)_x \cong
\Omega^*L \cong s(L,\alpha_. -\frac{1}{2i\pi} N.). 
\end{equation}

 The  endomorphisms $ \alpha_j Id$ and $ N_j $  correspond to endomorphisms denoted by the
symbols $ \alpha_j Id$ and $ \NN_j $  on the image  sections $ \widetilde v$ in ${\LL}_X)_x$.
We recall below a proof  of  the following result:
 the subcomplex $\Omega (L^{\e(\alpha .)},  \alpha_i Id -\frac{1}{2i\pi} N_i), i \in [1,n ])$ is acyclic if  there exists an index $i$ such that $\alpha_i \not= 0$.

 \begin
 {prop}\label{support}   
 Let  $x \in Y^*_M$ be  a general smooth point of  $\cap_{i\in M} Y_i$ for $M \subset I$, $L := \LL_X(x)$,  and $L^u \subset L$  the subspace on which the action of each  monodromy   $T_i$ is unipotent.
 
The complex  of $ \OO_{X,x}$-modules $(\Omega^* \LL)_x$  
  is quasi-isomorphic to   the complex of exterior algebra $\Omega^*L^u$  defined by $(L^u, T_i)$:
 \begin{equation*} 
\Omega^*L^u:= \biggl( 0 \to  L^u \to  \cdots \to \oplus_{\{i_1< \ldots < i_{k-1}\}}  L^u \to \oplus_{\{i_1< \ldots < i_k\}}  L^u   \cdots  \to 0 \biggr)  \simeq (\Omega^* \LL)_x
\end{equation*}
 with differential in degree $k-1: \, (d_{k-1} v_\b)_{\{i_1< \ldots < i_k\}} ) = \sum_{j} (-1)^{j}  \frac{ N_{i_j}}{2i\pi} v_{\{i_1< \ldots \hat i_j \ldots < i_k\}} $.
  \end{prop}
 \begin{proof}  We show that  for each  power of the maximal ideal $\m^r$ at $x$ with $r> 0$, the subcomplex $\m^r (\Omega^* \LL)_x$  is acyclic, and for $r =0$ only  the component $\Omega^*L^u$ is not acyclic.
 As $x$ is at the intersection of $n$ components of $Y$, we associate to each monomial $y^{m.} := y^{m_1} \cdots  y^{m_n} $ of degree $ r$ a  complex vector subspace $s L \subset \LL_x$ by the correspondence 
  $v \to y^{m.}\tilde v$ (\ref{tilda}) to get  a  complex of vector subspaces $\Omega^* Ly^{m.}$ of $  \m^r (\Omega^* \LL)_x$
  \begin{equation}\label{s}
\Omega^* Ly^{m.}:= \biggl( 0 \to s L \to  \cdots \to \oplus_{\{i_1< \ldots < i_k\}}  y^{m.}L   \frac{d y_{i_1}}{ y_{i_1}}\wedge \cdots \wedge \frac{d y_{i_k}}{y_{i_k}}  \cdots  \to 0 \biggr)
\end{equation}with   differential defined by:
$ y^{m.} \widetilde v   \to  y^{m.}  ([(m_i + \alpha_i)Id -\frac{1}{2i\pi}N_i]. \widetilde v)
 \otimes  \frac{d y_i}{y_i}$. The complex  $ \m^r (\Omega^* \LL)_x$ is an inductive limit 
 of the sub-complexes  $\Omega^* Ly^{m.}$ for various sections $y^{m.}$ such that $\sum_{i}m_i \leq r$. We  show in particular that each complex $\Omega^* Ly^{m.}$ is acyclic for $r>0$.
 \begin{lem}\label{sup}
 For each  monomial $y^{m.} := y^{m_1} \cdots  y^{m_n} $,  the sub-complex  $\Omega^* L y^{m.}$ of $ (\Omega^*\LL)_x$ (\ref{s})
 is acyclic if the degree $r := \sum_{i \leq n} m_i$ of $y^{m.}$ is strictly positive, hence 
 it   is quasi-isomorphic to the  complex 
   \begin{equation*}
 \Omega^*L:= \biggl( 0 \to  L \to  \oplus_{i\in [1,n]} L  \to \cdots \to \oplus_{\{i_1< \ldots < i_k\}}  L    \cdots  \to 0 \biggr)
 \end{equation*}
 with differential  in degree $k-1$:
 
 \smallskip 
 \centerline{ $  \forall v_{\b}, \,d_{k-1} (v_\b)_{\{i_1< \ldots < i_k\}}  = \sum_j (-1)^j ( (m_{i_j}+\alpha_{i_j})  Id -\frac{1}{2i\pi} N_{i_j} )v_{\{i_1< \ldots \hat i_j \ldots < i_k\}} $.}
 \end{lem} 
 \n  Since  
   $\nabla y^{m.} \tilde {v} = \sum_i m_i y^{m.}  \tilde {v}\otimes  \frac{d y_i}{y_i} + y^{m.} \Sigma_i [\widetilde{\alpha_i v }  -\frac{1}{2i\pi}
\widetilde {N_i v} ] \otimes \frac {dy_i}{y_i} = y^{m.} [(m_i + \alpha_i)Id -\frac{1}{2i\pi}N_i] \widetilde{ v } \otimes  \frac{d y_i}{y_i}$
 
 \n   The complex  $\Omega^* Ly^{m.}$ is isomorphic to the complex $ \Omega^*L$  in the lemma  where the differentials appear as given by morphisms $ (m_i + \alpha_i)Id -\frac{1}{2i\pi}N_i : L \to L$, hence it is  acyclic if  one of such morphisms is an isomorphism, that is  at least one $m_j +\alpha_j \not= 0$ in which case  $(m_j+\alpha_j ) Id  -\frac{1}{2i\pi}  N_j $
is an isomorphism of $L$ as $N_j$ is nilpotent; indeed, the complex may be written as a cone over such morphism. 
We deduce the proposition as $\alpha_j \not\in \Z$ unless $\alpha_j = 0$, then $m_j+\alpha_j = 0 $ iff $m_j =0$ and $\alpha_j = 0$.
 \end{proof}

\begin{rem}
[Reduction to the locally unipotent case]
\label{uni}

 It follows from the proposition that the  cohomology of the restriction to $Y$ is determined locally by the      unipotent  subspace under the monodromy actions.  This is a good reason to reduce the study of the weight filtration to local systems with locally  unipotent monodromy.

In particular, if $Z \subset Y$ is a component with non locally unipotent monodromy (the monodromy $T_Z$ around $Z$  has no eigenvalue equal to $1$), then $Rj_{Z *}(j_{!*}\LL_{\vert X-Z}) = j_{Z !}(j_{!*}\LL_{\vert X-Z})=  j_{!*}\LL$, where  $j_Z: X-Z \to  X$ denotes  the open embedding. In general, only the locally unipotent summand of  $\LL$    is interesting.
\end{rem}

{\it The above description of  $(R j_* {\LL })_x$ is the model for the next
description of various perverse sheaves in the rest of the section}.

\subsection
{The intermediate extension $  j_{!*}\LL \simeq IC^*\LL$} 

\
 We describe the intermediate extension  $ j_{!*}\LL$ at a point at  ``infinity'' along the NCD $Y$, by a sub-complex $ IC^*\LL \subset \Omega^* \LL$  containing the submodule $ \Omega_X^* \otimes \LL_X$.  Let $\II_Y $ denote the ideal product of the ideals of  the components $Y_i$.  The  complex of  $IC^*\LL$ contains the product $\II_Y \Omega^*\LL  $ as an acyclic sub-complex by lemma \ref{sup}. The quotient complex $\Omega^*\LL/IC^*\LL$ is supported by  $Y$.
 
 Due to (remark \ref{uni}), we state the results for locally unipotent local system to simplify the notations and we mention as a remark the case of a general local system.
\subsubsection
{Definition of $ IC^*\LL$ for a locally unipotent local system}\label{note}
 In the  local situation (\ref{ls}) at a general point $x\in Y^*_M$ of the intersection of $Y_i$ for $i \in M$, set  for all $K = \{i_1, \ldots, i_k\} \subset M = [1,n]$,    $A := \OO_x  \overset{\sim}{\longrightarrow}  \C\{\{x\}\}$, $ y_K := y_{i_1} \cdots y_{i_k} $ and $dy_K := dy_{i_1} \wedge \cdots \wedge dy_{i_k}$.
Let  $\widetilde L$ denote the vector space defined by sections $\widetilde  v$ for all flat vectors $v \in L$, generating  the fiber $ \LL_{X,x} := A \widetilde L$ of the sheaf $\LL_X$ at $x$. 
The fiber $ (\Omega^* \LL)_x $ of the sheaf $\Omega^*  \LL$ at $x$, is  generated as an    $(\Omega^*_X)_x $-module by $ \sum_{K \subset M} A \widetilde L \frac{dy_K}{y_K}$.  

   Let $N_J = \Pi_{j \in J} N_j$ denotes a composition of endomorphisms of $L$, and consider  the strict  simplicial  sub-complex
of the de Rham logarithmic complex (\ref{log}) defined by
$Im N_J := N_JL$ in $L(J) = L$. 

   More generally, each subspace $V \subset L$, defines a subspace $\widetilde V \subset \widetilde L$ generating  an  $A$-submodule of the fiber at $x$; for  $K, B \subset M$ we are interested  in the generating  subspaces  $V = N_K L$ image of $N_K$ and  $V = \ke \, N_B$  the kernel $N_B$.
\begin{defn}
[Stalk of the Intersection complex $ (IC^*\LL)_x$]\label{def}

\

Let $x\in Y^*_M $ be a general point in $Y_M$ and $A := \OO_x$. The sub-complex $(IC^*\LL)_x \subset (\Omega^*\LL)_x$ is generated as an $\Omega^*_{X,x}$-algebra 
by the  $ \OO_x$-sub-modules:
\begin{equation*}
 \sum_{B \subset K} y_{B} A (\widetilde {N_{K-B}L}) \frac{dy_K}{y_K}
\end{equation*}
sum for all $K\subset M$ and for each $K$ for all $B\subset K$.
\end{defn}
As an $\OO_{X,x}$-module, it is generated by 
$
  \sum_{B \subset K} y_{B} A (\widetilde {N_{K-B}L}) \frac{dy_K}{y_K}\wedge \Omega^*_{X,x} 
$.
In the proof, by reduction to a normal section to $Y_M$ at $x$, we  erase 
$\Omega^*_{X,x}$ in the formula.
\begin{lem}
[Independence of the coordinates]\label{ind} 
The definition of $(IC^*\LL)_x$ is independent of the choice of  the coordinates defining  $Y$.
\end{lem}
\begin{proof}
We  check the independence, after restriction to a normal section, for a one variable change at a time: $z_j= f y_j$ for  $j\in K$ with $f$ 
 invertible at $x$. We can    suppose $j = 1 $, and we write  $ IC^*L(z_1)$ (resp. $IC^*L(y_1)$) for the stalk at $x$ of the complex when defined with the coordinate $z_1$ and  $ y_j$ for $j > 1$ (resp.  with the coordinate $y_1$ instead of $ z_1$).  
We prove $  IC^*L(z_1)\subset  IC^*L(y_1) $. 

For   $K \subset M = [1, n]$ of length $\vert K\vert = k$, and $K_i \subset K$ of length $\vert K_i\vert = k-i$,  since
$ \frac{dz_1}{z_1} =  \frac{dy_1}{y_1} + \frac {df}{f} $, where $\frac {df}{f} $ is regular at $x$,  a section
$s \in z_{K_i} A (\widetilde { N_{K-K_i}L}) \frac{dz_K}{z_K} \subset IC^*L(z_1)$ is transformed into a  sum 
\begin{equation*}
  s = w_1 + w_2, \quad  w_1 \in IC^*L(y_1), \quad w_2 \in    y_{K_i} A (\widetilde { N_{K-K_i}L}) \,\,\frac {df}{f}  \wedge_{j \in K -  \{1\}} \frac {dy_j}{y_j}
\end{equation*}
to show that  $w_2$ is also in $IC^*L(y_1)$, it is enough to check that $ N_{K-K_i}L \subset 
 N_{K-1-K_i}L$  if $1\not\in K_i$, 
as $\frac {df}{f} $ is regular.  The proof of  $  IC^*L(y_1)\subset  IC^*L(z_1) $ is similar.
\end{proof}
\begin{prop}
[$IC^*L$]\label{IC1} 
 The intersection complex $(IC^*\LL)_x$ at a point $x \in Y_M^*$, for a locally unipotent local system $\LL$, is quasi-isomorphic to the complex:
\begin{equation}
 IC^*L := \biggl( 0 \to  L \to  \oplus_{i \in [1, n]} N_i L \cdots \to \oplus_{\{i_1< \ldots < i_k \subset  [1, n]\}}  N_{\{i_1< \ldots < i_k\}}L    \cdots  \to 0 \biggr)
\end{equation}
with differentials induced by the embedding into $\Omega^*L$.
\end{prop}
We may reduce the proof to  a normal section to $Y^*_M$ at $x$, that is we may suppose $(IC^*\LL)_x$ defined by the complex
\begin{equation*}
\cdots \to \oplus_{K\subset M} \biggl( \sum_{K_i \subset K} y_{K_i} A (\widetilde {N_{K-K_i}L}) \frac{dy_K}{y_K}\biggr)\to \cdots 
\end{equation*}
(instead of being  generated as an $\Omega^*_{X,x}$-module).
In the definition above, 
$N_{\emptyset} = Id$, hence for $K_i = K $, $ N_{K-K_i}L = L$.
The complex $IC^*L$ embeds into $(IC^*\LL)_x$ as 
\begin{equation*} 
  0 \to \widetilde L \to  \oplus_{j \in [1, n]} \widetilde {N_j L } \frac{d y_j}{y_j}  \cdots \to \oplus_{K \subset  [1, n]}  \widetilde {N_KL } \frac{d y_K}{y_K} \cdots  \to 0 
\end{equation*}
We show that the quotient complex is acyclic.  The proof  is similar to the proof of proposition (\ref{support}) and is carried  along $Y$ modulo $\II_Y \Omega^*\LL  $ which is  acyclic. 
 Let    $p \in [1, n]$ be an  integer, and for  each   monomial  $y^{m.} := y_1^{m_{i_1}} \cdots  y_p^{m_{i_p}} $ with all $m_{i_j}> 0$, let $K_{y^{m.}} = \{i_1, \cdots, i_p\} \cap K $. We define the sub-complex   $IC^*L y^{m.}$ of  $\Omega^*L y^{m.}$ (formula \ref{s}) as
\begin{equation*}
IC^*L y^{m.} := (y^{m.} L\cdots \to \oplus_{K, \vert K \vert = k} s  (N_{K-K_{y^{m.}}} L)  \frac{d y_K}{y_K}  \rightarrow \cdots \to 0).
\end{equation*}
\begin{lem}\label{IC2} 
  For each section $y^{m.} := y_1^{m_1} \cdots  y_p^{m_p}$ with at least one $m_i \not= 0$ for $i \in [1, p]$, the complex $IC^*L y^{m.}$ is acyclic: $IC^*L y^{m.}  \overset{\sim}{\longrightarrow}  0$. 
\end{lem}
 For $p = n$, $IC^*L y^{m.} = \Omega^*L y^{m.}$ and  for any $v \in L$, $y^{m.} \widetilde v \in (\II_Y\Omega^* \LL)_x$.
 
 For $p <n $, we introduce the complex of vector spaces  $IC(p)$:
\begin{equation*} 
IC(\check p) :=  \biggl( 0 \to  L \to  \oplus_{j \in [p+1, n]} N_j L \frac{d y_j}{y_j}  \cdots \to \oplus_{K \subset  [p+1, n]}  N_KL   \frac{d y_K}{y_K} \cdots  \to 0 \biggr)
\end{equation*}
where the differential  of  a vector  $ v \in\oplus_{\{i_1< \ldots < i_k\}_ \subset  [p+1, n]\}}  N_{\{i_1< \ldots < i_k\}}L  $ of degree $k $ is: $(d_k  v)_{\{i_1< \ldots < i_{k+1}\}} ) = -\frac{1}{2i\pi} \sum_{j \in [1, k+1] } (-1)^j  N_j v_{\{i_1< \ldots \hat i_j \ldots < i_{k+1}\}} $. 
 
We associate to each subset $B \subset [1, p]$,  the complex  $IC(p) \wedge \frac{d y_B}{y_B}$ and to each index
$j \in [1, p] - B$, a morphism  

\smallskip
\centerline {$(m_jId -\frac{1}{2i\pi} N_j)  \frac{dy_j}{y_j}: IC(p)  \wedge\frac{d y_B}{y_B} \to IC(p)\wedge \frac{dy_j}{y_j} \wedge \frac{d y_B}{y_B} $.}

We check that  $IC^*L y^{m.}$ may be written as  sum of a double complex 

\smallskip
\centerline {$IC^*L y^{m.}:= s \biggl( IC(\check p)  \wedge \frac{d y_B}{y_B}, (m_j Id -\frac{1}{2i\pi} N_j) \frac{d y_j}{y_j} \biggr)_ {B \subset [1, p]}$. }

Then, we conclude  as in the case of $\Omega^*L y^{m.}$.
 
 The quotient complex $(IC^*\LL)_x/IC^*L$ is an inductive limit of direct sum of complexes $IC^*L y^{m.}$ for $ y^{m.} $ is in the ideal of $A$ generated by $y_i$ for all $i \in [1, n]$, in which case $IC^*L y^{m.}$ is acyclic by the lemma;  then the proposition follows.
\subsubsection
{ Local definition of $(IC^*\LL)_x$ for a general local system}
 We introduce in terms of the spectral decomposition of $L$, for each set $\alpha .$, the composition of endomorphisms of $L^{\e(\alpha .)}$:
    $(\alpha .Id -\frac{1}{2i\pi} N .)_J = \Pi_{j \in J} (\alpha_jId -\frac{1}{2i\pi} N_j)$. 
    
    The strict  simplicial  sub-complex
of the de Rham logarithmic complex (\ref{log}) is defined by
$Im (\alpha .Id -\frac{1}{2i\pi} N.)_J$
in $L^{\e(\alpha .)}(J) = L^{\e(\alpha .)}$. 
\begin {defn}[$IC^*L$] The  simple complex defined by the above
simplicial sub-vector spaces  is the intersection complex  $\, IC^*L:= \oplus_{\alpha .} IC^*L^{\e(\alpha .)}$ where
\begin{equation}
IC^*L^{\e(\alpha .)}\colon = s({(\alpha .Id  -\frac{1}{2i\pi} N.)_J L^{\e(\alpha .)}}, \alpha .Id -\frac{1}{2i\pi} N.)_{J \subset [1,n]}.
\end{equation}
\end{defn}
We introduce for each set $\alpha .$
the subset $M(\alpha .) \subset [1,n ]$ such that $j
\in M( \alpha .)$ if and only if ${\alpha}_j = 0$. Let  $N_{J \cap M (\alpha .)} = {\Pi_{j \in J \cap M(\alpha .)}} N_j$ ( it is the identity if $J \cap M (\alpha .) = \emptyset$).
 For each $J \subset [1,n]$, we have the equality of the image subspaces: $  {(\alpha .Id -\frac{1}{2i\pi}N.)}_J
L^{\e(\alpha .)} = N_{J \cap M( \alpha .)} L^{\e(\alpha .)}$ since   the  endomorphism
$ ( \alpha_j Id -\frac{1}{2i\pi} N _j) $  is an isomorphism on $L^{\e(\alpha .)} $ 
whenever  $\alpha_j \neq 0 $, hence
\begin{equation}
 IC^*L \simeq  \oplus_{\alpha .} s(N_{J \cap
M(\alpha .)} {L}^{\e(\alpha .)})_{J \subset M} 
\end{equation}
\begin{lem} Let $L^u$ denote the subspace of $L$ such that all local monodromies are unipotent, then we have quasi-isomorphisms 
$IC^*L^u \simeq  IC^*L \simeq (IC^*\LL)_x$.
\end{lem}
Indeed, if there exists an index $k$ such that $\alpha_k \not= 0$, then $IC^*L^{\e(\alpha .)}$
may be written as a cone over the quasi-isomorphism 

$\alpha_k Id  -\frac{1}{2i\pi} N_k : s(N_{J \cap
M(\alpha .)} {L}^{\e(\alpha .)})_{J \subset M-k} \to s(N_{J \cap
M(\alpha .)} {L}^{\e(\alpha .)})_{J \subset M-k}$.
\subsubsection
{Global definition of the Intersection complex }
The local  definition of $IC^*L$ is well adapted to computations. As it is   independent of the local coordinates defining $Y$ (\ref{ind}), we deduce a global defintion of the Intersection complex as follows.

 A filtration $W$ of $\LL$ by  sub-local systems, extends to a filtration of $\LL_X$ by sub-bundles $\WW_X$.  For all subsets $M$ of $I$, the decomposition  of the restriction to $Y_M$ is  global:
$
\LL_{Y_M}\,  = \, {\oplus}_{(\alpha .)} {\LL}_{Y_M}^{\e(\alpha .)}$. The endomorphisms ${\NN}_i$ for $i$ in $M$ are defined on ${\LL}_{Y_M}$, and ${\NN}_M = {\Pi_ {i \in M}} {\NN}_i$ is compatible with the decomposition. The image ${\NN}_M {\LL}_{Y_M}$ is an analytic sub-bundle  of ${\LL}_{Y_M}$.
The residue of the
connection
$\nabla $ along each  $Y_j$ defines an endomorphism
$ {\alpha}_jId -\frac{1}{2i\pi} {\NN }_j $ on the component
${\LL}^{\e(\alpha .)}_{Y_j} $ of ${\LL}_{Y_j}$
compatible with the filtration by sub-analytic
bundles $\WW_{Y_j}$.  For subsets $J \subset I$ of the set $I$ of
 indices of  the components of $Y$, let:
 \[  ({\alpha. Id} -\frac{1}{2i\pi} \NN.)_J :=
 ({\alpha}_{i_1}Id -\frac{1}{2i\pi} \NN_{i_1})\cdots ({\alpha}_{i_j}Id -\frac{1}{2i\pi} \NN_{i_j}):\LL^{\e(\alpha .)}_{Y_J} \to \LL^{\e(\alpha .)}_{Y_J}  \]
\begin {defn} 
The Intersection complex
$IC^*\LL$ is the sub-analytic complex
 of ${\Omega}^*_X (Log Y) \otimes {\LL}_X \supset {\Omega}^*_X \otimes {\LL}_X$
 whose fibre at
a point $x \in Y^*_M$ is defined, in terms of a set of
  coordinates $y_i $
defining equations of $Y_M$ for $i\in M$,  as an ${\Omega}^*_{X, x}$ sub-module
generated by the sections $\widetilde {v} \wedge_{j \in J} \frac
{dy_j}{y_j}$ for all $v \in \im \ (\alpha.Id -\frac{1}{2i\pi}N.)_J \subset \oplus_{\alpha .} L^{\e(\alpha .)}$    and $J \subset M$
where $\widetilde {v} \in \LL_{X,x}$ and $L = \LL_X(x)$.
 \end {defn}
 This definition is independent of the
choice of coordinates (lemma \ref{ind}) and   a section of $IC^*\LL$  restricts
to a section in $IC^*\LL$ near $x$, since $(\alpha.Id -\frac{1}{2i\pi}N.)_J L \subset
(\alpha. Id -\frac{1}{2i\pi}N.)_{J-i}L $ for all $i \in J$.

For example in the unipotent case, for $M = \{1,2\}$ and $x \in Y_M^*  $, the sections in
$( IC^*\LL)_x \subset (\Omega^2_X (Log Y) \otimes \LL_X)_x $ are generated by $\tilde {v} \frac {dy_1}{y_1}\wedge \frac
{dy_2}{y_2}$ for $ {v} \in N_M L$, $\tilde {v} \frac
{dy_1}{y_1}\wedge dy_2$ for $ {v} \in N_1 L$, $\tilde
{v} dy_1 \wedge \frac {dy_2}{y_2} $ for ${v} \in N_2 L$
and $\tilde {v} dy_1 \wedge dy_2 $ for ${v} \in L$.

\smallskip
We remark that the notation $IC (X, \LL)$ instead of $IC^*\LL$ is used in (\cite{C}, chapter 8).
The intermediate extension $j_{!*} {\LL}$ of ${\LL}$ is topologically defined by
an explicit formula  in terms of the stratification (\cite{B}, \S 3). Locally its  fiber
$(j_{!*}  {\LL})_x $ at a point $x\in Y^*_M$  is quasi-isomorphic to the above complex 
$IC^*L$. 
\subsection
{Weight filtration}\label{ad}

The description of the weight  filtration on the logarithmic complex is based on the work of Kashiwara \cite{K} and \cite{EY}. Although the declared purpose of the paper \cite{K} is a criteria of admissibility in codimension one,  various  local statements and proofs in  \cite{K} are useful to define the weight filtration (\cite{C}, ch.8).

 We  start with a graded polarized variation of MHS 
on the underlying shifted  local system $\LL$ on  the complement of a NCD $Y$, satisfying certain asymptotic properties at points of $Y$, summarized under the condition of the admissibility property assuming as axioms the properties of the asymptotic behavior of  variation of MHS of geometric origin.
In particular we ask for the filtration $F$ on $\LL_{X^*}$ to extend into a filtration of $\LL_X$ by sub-bundles $F$, and the existence of the monodromy filtration relative to the natural extension of the filtration $W$ along $Y$.

Let $x \in Y$ at the intersection of $n$ locally irreducible components with local monodromy action $T_i$, and let $L:= \LL_X(x) = (\LL_X)_x\otimes_{\OO_{X, x}}\C$ denote the  space of ''multiform'' sections of $\LL$ (to emphasize this point the notation  $L := \psi_{x_1}\cdots \psi_{x_n}\LL$ is used in the literature). The conditions of admissibility, stated in terms of the extension of $F$ and the nilpotent endomorphisms $N_i := Log T^u_i$  logarithm of the unipotent  monodromy $T_i^u$, are defined by the structure of Infinitesimal  mixed Hodge structure (IMHS) on $L$, recalled below.
\subsubsection
{Mixed nilpotent orbit and Infinitesimal  mixed Hodge structure (IMHS)} 

\
We deduce from a variation of MHS on the complement of a NCD $Y$, locally  at a point $x \in Y$, the data: $(L, W, F,  N_1, \ldots, N_n)$ with an increasing filtration $W$ s.t.  $N_j W_k \subset  W_k$, a Hodge filtration $F$ and its conjugate $\overline  F$ with respect to the rational structure inherited from $\LL$. It is  called here mixed nilpotent orbit (pre-infinitesimal mixed Hodge module in (\cite {K}, 4.2).

The filtration $W$  is an extension of the weight $W$ on $\LL$. We may call it the finite filtration in opposition to the relative monodromy filtration to be introduced later called the limit filtration.
  
 {\it Nilpotent orbit}. We consider  an integer $w$ and  a data $(L, F, S; N_1, \ldots, N_n)$ where $ L $ is a finite dimensional complex vector  space with a $\Q$-structure,  $F$ is a decreasing $\C$-filtration,  such that: $N_j F^p \subset F^{p-1}$,  $S$ is a non-degenerate rational bilinear form satisfying
  \[   S(x,  y ) = (-1)^w  { S(y, x )}\quad {\rm for }\,  x,y \in L \quad
 {\rm and }\quad S( F^p, { F}^q)) = 0 \,\, {\rm for } \,\,  p+q > w. \]
and  $N_i$ are $\Q$-nilpotent endomorphisms for all $i$,  mutually commuting, such that  
$ S (N_i x, y) = S( x, N_i y)$ and $N_i F^p \subset F^{p-1}$.
 \begin{defn} 
[Nilpotent orbit](\cite {K}, 4.1) 
The above data is called a (polarized) nilpotent orbit of weight
$w$ if :

i)  The monodromy filtration $M $ of the nilpotent endomorphism $ N = \sum_j t_j N_j $  does not depend on the various $t_j $ for $t_j
> 0$ and all $j$.

ii) The data ($L, M, F$) is 
a MHS on $L$ of weight $w$ (a MHS ($L, M, F$) is  of weight $w$, or shifted by $w$,  if  ($ Gr^M_k, F$) is a   HS  of weight $w+k$). The bilinear form $S_k$
such that $ S_k (x, y) = S( x, N^k y)$ polarizes  the primitive subspace $ P_k = \ke(N^{k+1}: Gr^M_k
\to Gr^M_{-k-2})$ with its induced HS  of
weight $w+k$.
\end{defn}
Henceforth, all nilpotent orbits  are polarized. 
\begin{rem}[Nilpotent orbit]\label{Sch}
i) As $N_i$ acts on $F$, the exponential morphism $e^{i \sum t_j N_j}$ acts on $F$  for  $t_i \in \R$. The nilpotent orbit theorem proved by Schmid states that the above definition is equivalent to the statement:

There exists $c > 0$ such that $(L, e^{i \sum t_j N_j }F, e^{-i \sum t_j N_j} \overline F)$ is a Hodge structure of weight $w$ polarized by $S$ for $t > c$.

The HS is viewed as variation of HS for variable $t_i$ and $M$ is called the limit MHS.

ii) On this formula, it is clear that for each subset of indices $K \subset [1, n]$, the filtrations 
 $F_K := e^{i \sum_ {j \in K} t_j N_j}F, \overline F_K := e^{-i \sum_ {j \in K} t_j N_j} \overline F$ define a variation of HS  on $L$ such that $N_i F_K ^p \subset F^{p-1}_K $ with $i \not\in K$. Then $(L, W(\sum_ {i \not\in K} N_i),  F_K)$  is a MHS for  fixed $t_j, j \in K $.
\end{rem}

\begin{ex}
In the case of  a  polarized variation of HS of pure weight $w$  on the product  $n$-times of a small punctured disc  $D^*$, the limit Hodge filtration 
exists  on the  finite
dimensional vector space $L := \psi_{x_1}\cdots \psi_{x_n}\LL$. It defines with the monodromy filtration $M$ a nilpotent orbit approximating the variation of HS (\cite{CK}, \cite{CKS}). 
\end{ex}

 {\it Mixed nilpotent orbit}. We consider a data $(L,W, F, N_1, \ldots, N_n)$, where $ L $ is a finite dimensional complex vector  space with a $\Q$-structure, $N_i$ are $\Q$-nilpotent endomorphisms for all $i$,  
$W$ (resp. $F$) is an increasing $\Q$-filtration (resp. decreasing $\C$-filtration),  such that: $N_j F^p \subset F^{p-1}$ and $N_j W_k \subset W_k$.
  \begin{defn} 
[Mixed nilpotent orbit] 
The above data is called a mixed nilpotent orbit (graded polarized) if the data with restricted structures 

\smallskip
 \centerline {$ (Gr^W_i L, F_{|}, (N_1)_{|},
\ldots, (N_n)_{|})$ }

\n  is a nilpotent
orbit for each integer $i$,  of weight $i$ with some polarization $S_i$ (it is called pre-infinitesimal mixed Hodge module in (\cite {K}, 4.2).
\end{defn}
\begin{defn}
 [IMHS](\cite {K}, 4.3)

\
 A mixed nilpotent orbit $(L,  W, F, N_1, \ldots,
N_n)$ 
 is called an infinitesimal mixed Hodge structure (IMHS) if
the following conditions are satisfied:\\
 i) For each $J \subset I = \{1,
\ldots,n\}$, the monodromy filtration $ M(J)$ of $\sum_{j \in J}
N_j$ relative to $W$ exists and satisfy $N_j M_i(J) \subset
M_{i-2}(J)$ for all $ j \in J$ and $i \in \Z$.\\
ii) The filtrations $M(I), F $ define a graded
polarized MHS. The filtrations $W$ and $ M(J)$ are compatible with the
MHS and the morphisms $ N_i$ are of type ($-1,-1$).
\end{defn}
IMHS are called IMHM in
 \cite {K}; Deligne remarked, the fact  that if the relative monodromy  filtration
$M(\sum_{i \in I} N_i,W)$ exists  in the case of a mixed nilpotent orbit, then it is necessarily
 the weight filtration of a MHS. 
 
 Kashiwara proves  in  (\cite{K}, theorem 4.4.1) the main theorem: it is enough to check the existence of  $W(N_i)$ (in cxdimension $1$). Moreover, IMHS form an abelian category for which the filtrations $W, F, $ and $M$ are strict  (\cite{K}, prop. 5.2.6). 
 
 The morphism $N_i$ is not an auto-morphism of IMHS as it shifts $F$ but not $W$. However, when combined  with  the descent lemma it leads to an important application below (subsection \ref{Idec}).
\subsubsection{ 
{\bf Definition of the filtration }$N*W$ } (\cite{K},  3.4)

 Let ($L, W, N$) denote an
increasing filtration
  $W$ on a vector
 space $L$ with a nilpotent endomorphism $N$ compatible with $W$ s.t.
 the relative monodromy filtration $M(N,W)$ exists. Then
 a new filtration $N * W$  of $L$ is defined 
 by the formula (\cite{K},  3.4)
\begin{equation}\label{*}
 (N*W)_k := N W_{k+1} + M_k(N,W) \cap W_k = N W_{k+1} +
M_k(N,W) \cap W_{k+1}
 \end{equation}
 where the last equality follows from (\cite{K}, Prop 3.4.1).\\
The endomorphism $ N:
 L \to L$ and
 the identity $I$ on $L$ induce  morphisms
  \begin{equation*}
  N: W_k \to (N* W)_{k-1}, \quad I: (N* W)_{k-1} \to W_k .
  \end{equation*}
 We mention two important properties of $N*W$ ( \cite {K},
 lemma 3.4.2):\\
 The monodromy filtration  relative to $N*W$ exists and $M(N, N*W) = M (N,
W)$. Moreover, the following decomposition property  is satisfied
\begin{equation}\label{DP}
 \begin{split}
 Gr^{N*W}_kL &\simeq \im \ (N: Gr^W_{k+1}L \to Gr^{N*W}_kL )\oplus
 \ker \ (I  : Gr^{N*W}_k L \to
  Gr^W_{k+1}L )\\
 &{Im}\, (N: Gr^W_{k+1} L \to Gr^{N*W}_k L) \overset{\sim}{\longrightarrow} 
 {Im}\, (N: Gr^W_{k+1} L \to Gr^W_{k+1} L).
  \end{split}
   \end{equation}
  To  refer to this decomposition,  the couple ($L, N*W$) is said to form
   a (graded) distinguished pair.
\subsubsection{ {\bf The filtration $W^J$ associated to an IMHS} } 
Let $(L, W, F, N_i, i\in M)$ denote the  IMHS at $x \in Y^*_M$ and $ J \subset M$.
The relative weight filtration 
\begin{equation}
M(J, W):=  M( N, W) \,\, \hbox{\rm for }
N \in C(J):= \{ \Sigma_{j\in J}t_jN_j, t_j > 0\},
 \end{equation}
is well defined by assumption. A basic lemma (\cite {K}, cor. 5.5.4) asserts that:
\begin{lem} $ (L, N_1*W, F, N_i, i\in M)$
and $ (L, M(N_1,W), F, N_i, i\not= 1)$ are IMHS.
\end{lem}
\begin{rem}\label{key1}
Geometrically, $N_1*W$   and $M(N_1,W)$ may be  viewed on $\LL_{Y_1}$ near $x$. In term of the equation $z_1$ of $Y_1$, the filtration $N_1*W$   viewed  on $\psi_{z_1}\LL$  is not a variation of MHS. 

\n  In (\cite{K}, lemma 3.4.2 ii and proposition 5.2.5) the following equality   is  proved:
 
\centerline { $M(\sum_{i\in M} N_i,  N_1*W) =  M(\sum_{i\in M- \{1\}} N_i, M(N_1,W)) =  M(\sum_{i\in M} N_i, W)$.}
\end{rem}

\n We deduce a recursive definition of an increasing filtration $W^J$  of $L$ by the star
operation  
\begin{equation}\label{J}
 W^J :=  N_{i_1}*( \ldots (N_{i_j}*
W)\ldots ) \,\,\hbox {for} \,\, J = \{i_1, \ldots, i_j \}
\end{equation}
It is denoted by $\Psi_J
*W$  in
  (\cite {K}, 5.8.2) as it  is  interpreted as defined  repeatedly on $\Psi_{x_i}$ for $x_i \in J$
  in the theory of nearby cycles. 
  
 It followws by induction that  $(L, W^J, F, N_i, i\in M)$ is an  IMHS.  The  filtration $W^J $ does not
depend on the order of composition of the respective
transformations $N_{i_k}*$
 since  $N_{i_p}*
 (N_{i_q}* W )= N_{i_q}* ( N_{i_p}* W) $ for all $i_p, i_q \in J$
  according to a basic result in (\cite{K},   proposition 5.5.5). 
   The star operation has the following
  properties (\cite{K},  formula 5.8.5 and 5.8.6):
  
  \smallskip
\centerline {  $\forall J_1, J_2 \subset M :  M(J_1, W^{J_2} ) = M(J_1, W)^{J_2}$,  $\forall 
 J \subset K \subset M: M(K,  W^J) =  M(K, W)$.}
  
  \smallskip
   \n  The  filtrations $W^J$ fit together to define
  the weight filtration  on $\Omega^*L$:
\begin{defn}\label{lw}
 The filtrations $W$ and $F$ on the de Rham complex
  associated to an IMHS $(L, W, F, N_i, i \in M)$ are defined as 
 simplicial complexes
  \[
  W_k \Omega^*L:= s(W^J_{k-|J|}, N.)_{J \subset M}, \, F^p \Omega^*L:= s(F^{k-|J|}, N.)_{J \subset M}
 \]
\end{defn}    
  \begin{rem}\label{wloc} The induced morphism $N_i: W^J_k \to (N_{i}* W^J)_{k-1}$ drops the degree
of  the filtration $W^J$.
 For further use, it is important to add to the above data, the canonical inclusion
  $I:(N_{i}* W^J)_{k-1} \to W^J_k$.

 Beware that the weight   filtrations $W$  
    on   $\LL$ underlying the variation of MHS,  extends as a constant filtration   $W'$ on $\Omega^*L$. The weight $W$  on $\Omega^*L$ is  the local definition of  a global weight
    on the de Rham complex.
      The relation between $W$ and $W'$ on the terms of $\Omega^*L$ is described  by  (\cite{K} lemma  3.4.2  iii) and corollary  3.4.3).
\end{rem}

 \subsubsection
 {Local decomposition of  $(\Omega^*L, W)$}\label{data}
 
The various  decompositions in each degree, as in the formula \ref{DP}, result into  the decomposition of the whole complex $Gr^W_*(\Omega^*L)$. 
The proof involves the whole  de Rham data $DR(L)$
  attached to an IMHS (\cite{K}, section 5.6):
\[
 DR(L):= \{L_J, W^J, F^J,   I_{J, K}: L_J \to L_K,
 N_{J-K }: L_K \to L_J \}_{K \subset J \subset M}
 \]
where for all $J \subset M$, $L_J := L$, $F_J:= F$,  $W^J$ is the
filtration defined above (formula \ref{J}),  $ N_{J-K }:= \prod_{i \in J-K} N_i$ and
$I_{J, K}:= Id : L \to L$. 

 A set of  properties of  this data stated  by
Kashiwara in seven  formula (5.6.1)-(5.6.7) follow from results  proved in various sections of  (\cite{K}).
   In particular, for each  $J \subset M$, the data: $  (L_J, W^J, F_J,
\{N_j, j \in M \})$ is an IMHS (\cite{K}, formula 5.6.6), which  follows from (\cite{K}, corollary 5.5.4), while formula (5.6.7) with adapted notations states that
  for each  $K \subset J \subset M$,  
  $$ Gr N_{J-K }:Gr^{W^K}_{a}L_K \to Gr^{W^J}_{a- |J-K|}L_J,\, Gr I_{J, K}: Gr^{W^J}_{a- |J-K|}L_J \to  Gr^{W^K}_{a}L_K$$
 form a graded distinguished pair: 
  $$Gr^{W^J}_{a- |J-K|}L_J \simeq \im \, Gr N_{J-K } \oplus \ke  \, Gr I_{J, K}.$$
Now, we are in position to define the ingredients of the local decomposition of  $\Omega^*L$
  
 \begin{lem}[\cite {K}, proposition 2.3.1 and lemma 5.6.2] \label{ld1}

i)  Set for each $J \subset M$
 \[ P^J_k(L):= \cap_{K \subset J, K \not= J}
 \hbox {Ker}\, (Gr I_{J,K}: Gr^{W^J}_k L \to Gr^{W^K}_{k+ \vert J-K\vert} L) \subset
Gr^{W^J}_k L \] then $P^J_k(L)$ has pure weight $k$ with respect to
the weight filtration $M(
\sum_{j \in J} N_j, L)$.\\
 ii ) We have: $ Gr^{W^J}_k L
\overset{\sim}{\longrightarrow}  \oplus_{K \subset J} N_{J-K } P^K_{k + \vert J-K\vert}(L)$.
\end{lem}
 In the definition of  $P^J_k(L)$, we can suppose $\vert K \vert = \vert J \vert - 1$ in i). 
 
 For $i \in J$, let $K = J - i$,  $N_i: Gr^{W^K}_{k+1} L \to Gr^{W^J}_k L$
 and $I_i: Gr^{W^J}_k L \to Gr^{W^K}_{k+1} L$, then $N_i$ on $ Gr^{W^J}_k L $ is equal to  $ N_i \circ I_i $, hence $N_i$ vanish on  $P^J_k \subset \ke \ N_i$ for all $i \in J$, and
  $N:= \sum_{i\in J}N_i$ vanish on $P^J_k$. Then the weight filtration $W(N)$ on the component $P^J_K(L) $ of  $Gr^{W^J}_k L$ coincides with the weight filtration of the morphism $0$, hence $P^J_K(L) \subset Gr^{W(N)}_k \ke \ N  \subset Gr^{W(N)}_k Gr^{W^J}_k L $  is a polarized pure HS.

\n Remark that $W(N)_{k-1} \ke N$  is  a subset of  the component $  \im \,  N_i \subset Gr^{W^J}_k L $.

The statement (ii)
 is proved in (\cite {K}, proposition 2.3.1) with a general terminology. 
  We illustrate here the proof in the case of two nilpotent endomorphisms.  Set  $W^i := N_i*W$ for $i =1, 2$, $
 W^{12} = N_1*N_2*W$ in the following diagram 
   \[  
\xymatrix{      &Gr^{W^1}_i L  \ar[dr]^{N'_2}  & \\
     Gr^{W }_{i+1} \ar[ur]^{N_1} L   \ar[dr]^{N_2}           &  &  Gr^{W^{12}}_{i-1} L    \\
          &   Gr^{W^2}_i   L   \ar[ur]^{N'_1}      & }
          \, 
         \xymatrix{     & \ar[dl]^{I_1} Gr^{W^1}_i L    & \\
     Gr^{W }_{i+1}           &  &  Gr^{W^{12}}_{i-1} L  \ar[ul]^{I'_2} \ar[dl]^{I'_1} \\
          &   Gr^{W^2}_i   L   \ar[ul]^{I_2}    & }
\]
where $ I'_i$ and $N'_i$ are induced by $I_i$ and $N_i$.  The morphisms $ N_i$ and $I_i$ (resp. $ N'_i$ and $I'_i$ ) form  distinguished pairs.
 Let $P^{12,1}_{i-1} L := \ke \, I'_2:  Gr^{W^{12}}_{i-1} L \to Gr^{W^1}_i L $
 
 \n  (resp.  $P^{12,2}_{i-1} L := \ke \, I'_1, \, P^{1}_{i} L := \ke \, I_1, \, P^{2}_{i} L := \ke \, I_2$),
  then
 
\smallskip

 $  Gr^{W^1}_i L \simeq N_1 Gr^{W}_{i+1} L \oplus P^{1}_{i} L$, and $  Gr^{W^2}_i L \simeq N_2 Gr^{W}_{i+1} L \oplus P^2_{i} L$

$ Gr^{W^{12}}_{i-1} L \simeq N'_2  Gr^{W^1}_i L \oplus P^{12,1}_{i-1} L \simeq N'_2 N_1 Gr^{W}_{i+1} L \oplus  N'_2 P^{1}_{i} L \oplus  P^{12,1}_{i-1} L$, and 

 $ Gr^{W^{12}}_{i-1} L \simeq N'_1 N_2  Gr^{W}_{i+1} L \oplus  N'_1 P^2_i L \oplus  P^{12,2}_{i-1} L$,
 
\smallskip
\n  Let $P^{12}_{i-1} L := P^{12,2}_{i-1} L \cap P^{12,1}_{i-1} L$. It remains to prove  

\smallskip
\centerline{ $ P^{12,1}_{i-1} L =   N'_2 P^1_i L \oplus  P^{12}_{i-1} L $ and $ P^{12,2}_{i-1} L =   N'_1 P^2_i L \oplus  P^{12}_{i-1} L$. }

\smallskip
\n which follows from the linear algebra relations:

\smallskip
\centerline{ $I'_2  \circ N'_1 = N_1 \circ I_2$, 
 and  $I_2  \circ I'_1 = I_1  \circ I'_2$.}
\begin{rem}
In the global decomposition of the logarithmic complex, the HS $P^J_k(L)$  is viewed as a fibre of a VHS on $Y_J$ (corollary \ref{decom} below, and theorem \ref{PPP}).
\end{rem}
The following corollary
 is a basic step to check the structure of MHC on the logarithmic
 de Rham complex
\begin{cor}[local decomposition]\label{decom}

The graded vector space of the filtration $W$  on $\Omega^*L$ (definition \ref{lw})
satisfy  the decomposition  property into a direct sum of
Intersection  complexes
\begin{equation}
 Gr^W_k(\Omega^*L)\overset{\sim}{\longrightarrow} 
\oplus IC^*(P^K_{k-|K|}(L)[-|K|))_{K \subset M}
\end{equation}
\end{cor} 
The corollary  is the local statement
of the structure of MHC on the logarithmic de Rham complex.

It follows  from the lemma by definition of  $IC^*$ (proposition \ref{IC1}) by a general  description of
perverse sheaves in the normally crossing divisor case (\cite {K},
section 2). The description is in terms of a family of vector spaces and a diagram of morphisms. The minimal extensions are characterized by relations of distinguished pairs as in the formula \ref{DP}. 

 The family of vector spaces is deduced from a perverse sheaf by repeated application  of  the functors nearby and vanishing cycles with respect to the coordinates. 

\subsubsection{ Decomposition of the 
Intersection  complex $IC^*L$}\label{Idec}

We need to prove a similar decomposition of the Intersection complex $IC^*L$ associated to
 ($L, W, F, N_i, i \in M$)  an IMHS. 
For each $j \in M$, we consider  the  data  ($N_j L, N_j*W, F, N_i, i \in M$) with  induced filtrations from the IMHS ($L, N_i*W, F, N_i, i \in M$) (we should write $ (N_j*W)_{\vert N_jL}, F_{\vert N_jL}, (N_i)_{\vert N_j L}$ but the symbol of restriction is erased). We have a graded exact sequence defined by the induced   filtrations  on $N_jL$:
\begin{equation*}
0  \longrightarrow  Gr^{ W^j}_k N_j L \longrightarrow  Gr^{ W^j}_k  L \longrightarrow  Gr^{ W^j}_k L/ N_j L \longrightarrow  0
\end{equation*}
\begin{lem}[descent lemma for IMHS] i) Let ($ L,  W, F, N_i, i \in M$) be an IMHS, then ($N_j L, W^j:=  ((N_j*W)_{\vert N_j L}, F, N_i, i \in M$)  is an IMHS for all $j \in M$.

ii)   The couple ($ L, W,  F,  N_i $) $ \xrightarrow{N_j}$
($N_j L,  W^j, F, N_i$) form a graded distinguished pair:  $Gr^{ W^j}_k N_jL$ splits into a direct sum 
\begin{equation*}
  \hbox {Im}\,(N_j : Gr^{ W}_{k+1} L \rightarrow Gr^{  W^j}_k NjL ) \,
 \oplus \,  \ke \, (I_j : Gr^{ W^j}_k N_jL  \rightarrow Gr^{  W}_{k+1}L)
\end{equation*}
\end{lem}
\begin{proof}
The filtration $W^j$ is first constructed on $L$, then induced on $N_jL$. The proof is simultaneous for i) and ii).
We consider the diagram
\begin{equation*}
\begin{array}{ccccc}
0&\to& Gr^{W}_{k+1} L& = &Gr^{W}_{k+1}L\\
&& \downarrow N'_j \uparrow I'_j&& \downarrow N_j\uparrow I_j\\
 0&\to& Gr^{W^j}_k N_jL &\xrightarrow {s} & Gr^{W^j}_k  L \\
   \end{array}
\end{equation*}
where $s$  is injective,  $N'_j$ and $ I'_j$ are other symbols for the restriction of  $N_j$ and $I_j$. 
Set $P_k :=  \ker\,I_j  $, $ \im \, N_j$ for the image of $N_j$, $\im \, N'_j$ for the image of $ N'_j$  and $\widetilde \im \ N_j := \im \, (N_j: Gr^{W}_{k+1} L \to Gr^{W}_{k+1}L)$. 

The terms on the right column form a distinguished pair, that is we have a
  splitting   $Gr^{ W^j}_k L  \simeq  \im \, N_j \oplus
 \ker\, I_j $ and an isomorphism $\im\, N_j \simeq  \widetilde \im\, N_j$. 
 
 Set  $P'_k := (Gr^{ W^j}_k N_jL) \cap P_k $. As  
 $ \im \, N'_j \simeq s (\im \, N'_j) = \im \, N_j$, we deduce the splitting 
 \begin{equation*}
Gr^{ W^j}_k N_jL  = \im \, N'_j \oplus P'_k
\end{equation*}
 where $ P'_k = \ker \ (I'_j  : Gr^{ W^j}_k N_JL \to
 Gr^{\tilde W}_{k+1}N_K L )$. 
 We prove that  both components of this splitting are nilpotent orbits of weight $k$.
 
1)   We apply the descent lemma  (\cite{CKS1} descent lemma 1.16)  to  $\im \, N'_j $ since  it is isomorphic to the image $ \widetilde \im \ N_j$ (the image  of $N_j $ acting on the nilpotent orbit $ Gr^{ W}_{k+1} L$).  We deduce  that the data ($\im \, N'_j, F,  N_i , i \in M$) is a nilpotent orbit of weight $k$.  
  
2) The kernel  $P_k  = \ke \ I_j \subset Gr^{ W}_{k+1}N_j L \subset Gr^{ W}_{k+1} L$  is  compatible with the action of $N_i$ for all $i \in M$,  and since $N_j = N_j \circ I_j $ vanish on $P_k$ we deduce that $P_k \simeq Gr^{W(N_j)}_k (\ke \ N_j)$  has a pure polarized HS induced by the MHS of $W(N_j) \subset Gr^{ W^j}_{k+1} L$. In particular $W(N_j)_{k-1} \ke \ N_j $ is   in  $\im \ N_j \simeq  \widetilde \im \ N_j$ which is consistent with the descent lemma.  

Moreover the orthogonal subspace to $ P'_k = \ke \, I'_j \subset P_k $ in the primitive part $Gr^{W(N_j)}_k \ke \ N_j$ (resp. in $Gr^{W(N_j)}_k \ke \ N_j$) defines a splitting  compatible with the action of $N_i$ for all $i \in M$ as the $N_i$ are infinitesimal isometry of the bilinear product $S:  S(N_i a, b) + S(a, N_i b) = 0$, hence $ P'_k $ is a sub-nilpotent orbit 
of the nilpotent orbit  structure on $Gr_k^{W(N_j)}Gr^{W^j}_{k+1}L$.
\end{proof}
\begin{cor}\label{open} i) Let  ($ L,  W, F, N_i, i \in M$) be an IMHS, then
($N_K L,  W^K, F, N_i, i \in M$) is an IMHS for all $K \subset M$, and the embedding 
($N_K L,  W^K, F$) $\subset$ ($ L,  W^K, F$)  in the category of IMHS is strict for the filtrations.

ii)   The couple $(N_K L, W^K,  F,  N_i )  \xrightarrow{N_j}( N_J L,  W^j, F,  N_i)$ form a graded distinguished pair: $Gr^{W^j}_k N_JL$ splits into a direct sum 
\begin{equation*}
  \im\,(N_j : Gr^{ W^K}_{k+1}N_K L \rightarrow Gr^{ W^J}_k N_JL ) \,
 \oplus \,  \ke \, (I_j : Gr^{W^J}_k N_JL  \rightarrow Gr^{W^K}_{k+1}N_K L)
\end{equation*}
and the image $\im\, N_j$ is isomorphic to $ \im\,(N_j : Gr^{ W^K}_{k+1}N_K L \rightarrow Gr^{ W^K}_{k+1} N_KL )$.
\end{cor}
We prove the corollary by induction on the length of $K$.
By the lemma, the embedding $ (N_jL, W^j, F , N_i)  \subset (L, W^j, F , N_i) $ is an injective  morphism in the abelian category of IMHS. 
 
 We assume by induction that ($N_K L,  W^K, F, N_i, i \in M$) is an IMHS, and let $J := K \cup i$ for $i \not\in K$. First we consider  the induced morphism $N_j: W^K \to W^K$, the filtration  $N_j* W^K_{\vert N_K L}$ defined by the star operation on $N_KL$ and its induced filtration on $N_JL \subset N_KL$. 
 There is another filtration obtained from $W^K$ on $L$: $N_j*W^K$ on $L$ and its induced filtration on $N_KL$ such that we have a diagram 
 \begin{equation*}
(N_j*W^K)_{\vert N_J L} \to  N_j*W^K_{\vert N_K L}\to W^J:=  N_j*W^K
\end{equation*}
 By the lemma the first morphism is in the category of IMHS and by the inductive hypothesis the second morphism is compatible with the IMHS, hence the morphisms are strict and 
  the filtration $W^J$ in $L$ induces $N_j*W^K_{\vert N_J L} $ on $N_JL$.
 Hence ($N_J L,  W^J, F, N_i, i \in M$)  is an IMHS, and the embedding into ($ L,  W^J, F, N_i, i \in M$) is compatible with the IMHS. 
\subsubsection{{\bf Decomposition of} $IC^*L$}
 As in  the case of $\Omega^* L$ (subsection \ref{data}) we 
 
 \n introduce  the data $DR(IC^*(L):=$

\smallskip
\centerline{$  \{ N_JL, W^J,  F^J, I_{J,K}: N_J L \to  N_K L, N_{J-K}: N_K L \to N_J L \}_{K \subset J\subset M}$} 

\smallskip
 \n Set for each $J \subset M$
 \[ P ^J_k(N_J L):= \cap_{K \subset J, K \not= J}
 \hbox {Ker}\, (Gr I_{J,K}: Gr^{W^J}_k N_JL \to Gr^{W^K}_{k+ \vert J-K\vert} N_K L) \subset Gr^{W^J}_k L \]
 then $ P^J_k(N_J L) = P^J_k(L) \cap Gr^{W^J}_k  N_JL$ has pure weight $k$ with respect to
the weight filtration $M(
\sum_{j \in J} N_j, L)$ and it is polarized.
\begin{cor}\label{IC3} i) $Gr^{W^J}_k N_J L
\overset{\sim}{\longrightarrow}  \oplus_{K \subset J} N_{J-K } P^K_{k + \vert J-K\vert}(N_K L)$.

\smallskip
ii) \centerline{
 $Gr^W_k(IC^*L)\overset{\sim}{\longrightarrow} 
\oplus_{K \subset M} IC(P^K_{k-|K|}(N_KL )[-|K|]$.}
\end{cor}
\subsection { Global definition and  properties of the weight $W$} 
The local study
ended with  the local decomposition into Intersection complexes.
We develop now the corresponding global result.
The local cohomology of $\Omega^*\LL$ is reduced to $\LL$ on $X-Y$ and to  the various 
analytic restrictions $\LL_{Y_J}$ on $Y$ (proposition \ref{support}). 

In the case of a pure VHS, the weight  starts with the subcomplex  $IC^*\LL$. 
 
{\it Notations in the global case.} We  suppose $\LL$  locally unipotent by the remark \ref{uni} as other components of   the spectral decomposition have  acyclic de Rham complexes.
We recall that $\LL :=  \widetilde \LL[n]$ with $n:= dim. X$ is a shifted variation of MHS, hence all complexes and bundles below are already shifted by $n$ to the left. For all $J \subset I$, let  $ \LL_{Y_J}$ denote the restriction of the analytic module $\LL_X$ to  $Y_J$. 
The weight filtration $W$ on $\LL$ defines a filtration by sub-bundles  $\WW_X$ of $\LL_X$ with  restriction $\WW_{Y_J}$ to $Y_J$, such that
the relative monodromy  filtrations of the endomorphism
$\sum_{i\in J}\NN_i$ of $( \LL_{Y_J}, \WW_{Y_J})$   exist for all $J \subset I$.
We write $\MM (J, \WW_{Y_J}) := \MM (\sum_{i\in J}\NN_i, \WW_{Y_J}) $, it
 is   a filtration by analytic sub-bundles, so
that we can define
\[ (\NN_i*  \WW_{Y_i})_k:= \NN_i \WW_{Y_i, k+1}+\MM_k(\NN_i,
\WW_{Y_i})\cap \WW_k = \NN_i \WW_{Y_i, k+1}+\MM_k(\NN_i,
\WW_{Y_i})\cap \WW_{k+1}\]
and an increasing filtration
$\WW^J$ of $\LL_{Y_J}$  is recursively defined by the star operation
\begin{equation*}
 \WW^J:=   \NN_{i_1 }*( \ldots (\NN_{i_j } *
\WW)\ldots ) \,\,\hbox {for} \,\, J = \{i_1, \ldots, i_j \}.
\end{equation*} 

\begin{defn}[weight $W$ on $\Omega^*\LL$]  Let $(\LL, W, F) $ be an admissible variation of MHS on $U:= X-Y$.
The weight  filtration by sub-analytic complexes, denoted  also by  $ W $ on $\Omega^*\LL$, 
  is defined locally  in terms  of
  the various restrictions $\WW_{Y_J}$ to the strata $Y^*_J$.
  
  At a point  $y \in Y^*_M$ for  $M \subset I$,
 $W_r (\Omega^*_X (Log Y) \otimes {\LL}_X)_y$ is defined  locally in a neighborhood of $y \in Y^*_M$, 
 in terms of the IMHS $(L, W, F, N.)$ at $y$ and a set
  of coordinates $y_i$ for $ i\in M$
 (including local equations of  $Y$ at $y$), as follows:
 
The term $ W_r $ of the filtration is generated as an ${\Omega}^*_{X,y}-$ sub-module
  by the germs of the sections $  \widetilde {v}  \otimes \wedge_{j \in J} \frac
{dy_j}{y_j} $ for $ J \subset M$ and  $v \in W_{r-|J|}^JL$.

In particular, for $M = \emptyset$, at a point $y \in U$, $ W_r $ is generated as an ${\Omega}^*_{X,y}-$ sub-module
  by the germs of the sections  $v \in W_r \LL_y $.
\end{defn}

 The  definition of $W$ above is independent of the choice
of coordinates on a neighborhood $U(y)$, since if we change
the coordinate into $y'_i = f y_i$  with $f$
invertible holomorphic at $y$, we check first that the submodule
$W_{r-|J|}^J (\LL_X)_y$ of $\LL_{X,y}$ defined by the image of $
W_{r-|J|}^JL$ is independent of the coordinates as in the
construction of the canonical extension (see also lemma \ref{ind}). For a
fixed $ \alpha \in W_{r-|J|}^J \LL_{X,y} $, as the difference $
\frac {dy'_i}{y'_i} - \frac {dy_i}{y_i} = \frac {d f}{f}$ is
holomorphic at $y$, the difference of the sections $ \alpha \otimes \wedge_{j \in
J} \frac {dy'_j}{y'_j} -  \alpha \otimes \wedge_{j \in J} \frac
{dy_j}{y_j}  $  is still a section of the
${\Omega}^*_{X,y}-$sub-module generated
  by the germs of the sections $ W_{r-|J|}^J (\LL_X)_y \otimes\wedge_{j \in (J-i)} \frac
{dy_j}{y_j}$.\\
Finally, we remark that the sections defined at $y$
restrict to sections  defined  on $(U(y)- Y^*_M
\cap U(y))$.
\subsubsection{Global statements}\label{globe}
The Hodge filtration on the Intersection complex defined algebraically on $IC^*\LL$ induces a HS on the Intersection cohomology with coefficients in a polarized VHS.

  \begin{thm} Let $(\LL, F)$ be a shifted polarized VHS of weight $a$, then the sub-complex
  $(IC^*\LL, F)$ of the logarithmic
  complex with induced filtration $F$ is a Hodge complex which
   defines a pure HS of weight $a+i$ on its hypercohomology $\H(X, IC^*\LL)$, equal to  the Intersection cohomology $IH^i(X,\LL)$.
\end{thm}
We rely on the proof  by Kashiwara (\cite{K1}, Theorem 1 and Proposition 3).
\begin{rem}\label{Pb}
i)  The proof of the local purity theorem here relies on the polarization of the intersection cohomology in the case of the complement of a NCD in $X$. 

 The original proof of the purity theorem is local at $v$ \cite{DG}.  As noted by the referee of the notes, the MHS on the boundary of the exceptional divisor at a point $v \in V$ (definition \ref{dual0}, lemma \ref{dual1}) is independent of the global variety $X$ but depends only on the components of the NCD $X_v$ and their embedding in $X$. 
  
  It is interesting and possible to adapt such local proof,  then the algebraic-geometric constructions  of $W$ and $F$ leads to an alternative construction of  Hodge theory as follows.
  
  The case of curves is studied in \cite{Z}. We proceed by  induction on dim.$X$, and  assume the purity of $IC^*\LL$ in dimension $ n -1$ in the inductive step.
  We construct  a  morphism $f : X \to V$ to  a smooth curve $V$  for which the decomposition theorem apply with just a local proof of purity, from which we deduce that the intersection cohomology $\H^i(X, \LL)$ carries a pure HS induced by the  Hodge filtration defined in this paper.
  
  ii)  The subtle proof of the comparison in  \cite {K1} is   based on the
auto-duality of the Intersection cohomology.

We are not aware of  a direct comparison of  the filtration $F$ on $IC^*\LL$ defined  here in an algebro-geometric  way, with the filtration $F$    in  terms of $L^2$-cohomology in \cite{K2}
and \cite{CKS}. 
   \end{rem}

\subsubsection
{ The bundles $\PP^J_k(\LL_{Y_J})$} 
Given a subset $J
\subset I$, the filtration $\WW^J$ induces a filtration by
sub-analytic bundles of $\LL_{Y_J}$. We introduce the
following  bundles
 \[ \PP^J_k(\LL_{Y_J}):= \cap_{K \subset J, K \not= J}
 \hbox {Ker}\, (Gr I_{J,K}: Gr^{\WW^J}_k \LL_{Y_J} \to
 Gr^{\WW^K}_{k+ \vert J-K\vert} \LL_{Y_J}) \subset
Gr^{\WW^J}_k \LL_{Y_J} \] where $I_{J,K}$ is defined as in the local case.
 In particular $\PP^{\emptyset}_k(\LL_X)
 = Gr^{\WW}_k \LL_X $ and $\PP^J_k(\LL_{Y_J}) = 0$ if
 $  Y_J^* = \emptyset $.

\begin{thm}\label{PPP}
 i) The weight $W$  is
a filtration by  sub-complexes
of  $Rj_* \LL$   consisting of perverse sheaves defined over $\Q$. 

ii) The bundles $\PP^J_k(\LL_{Y_J})$ are  Deligne's extensions of
polarized VHS  $\PP^J_k(\LL)$ on $Y_J^*$ of weight $k$  with respect to the weight  induced by $\MM( \sum_{j \in J} \NN_j, \LL_{Y_J^*})$. 
 
 iii) The graded
perverse sheaves for the weight filtration, satisfy the
 decomposition property into intermediate extensions for all $k$ 
 \[
Gr^W_k (\Omega^*\LL):= Gr^W_k (\Omega^*_X(Log Y)\otimes \LL_X) \overset{\sim}{\longrightarrow}  \oplus_{J \subset
 I} (i_{Y_J})_* j_{!*}\PP^J_{k-|J|}(\LL)[-|J|].
 \]
 where $j$ denotes uniformly the inclusion of $Y_J^*$ into $Y_J$
 for each $J \subset I$.
 \end{thm}
The proof is essentially based on the local study above.
The various graded  complexes $Gr^W_k
(\Omega^*_X(Log Y)\otimes \LL_X)$ are direct sum of Intersections complexes defined over $\Q$.  In particular, we deduce that 
 the whole filtration $W_k$
on the de Rham complex is defined over $\Q$, and $W_k$ is 
perverse  with respect to the stratification defined by
$Y$.\\
 The local rational structure
 of the complexes $W_k$ glue into a global rational structure,
 since perverse sheaves may be glued as the usual sheaves, although
 they are not concentrated
 in a unique degree (\cite{BBD} corollary 2.1.23 and 2.2.19).
 
 Another proof of the existence of the rational structure is based
 on Verdier's specialization \cite{E2}. 
 
 The next result  on weights satisfy a property cited in (\cite{WII}, cor 3.3.5).
\begin{cor}  The de Rham logarithmic mixed Hodge complex $(\Omega^*\LL, W, F)$ of
an admissible variation of MHS of weight $\omega \geq a$ induces on the
hypercohomolgy $\H^i (X-Y, \LL)$ a MHS of weight $ \omega \geq a+i$.
\end{cor}
Indeed,  $W_k = 0$ on the logarithmic  complex  for
$k \leq a$.
\begin{prop}\label{Ico}  The Intersection complex $(IC(X, \LL), W, F)$
of an admissible variation of MHS,  as an embedded
sub-complex of $(\Omega^*\LL, W, F)$ with induced filtrations, is a
mixed Hodge complex satisfying for all $k$:
 \[(IC^*(Gr^W_k \LL), F) \simeq
 (Gr ^W_k IC^* \LL, F)
 \]
\end{prop}
 \n The proposition follows from corollary \ref{IC3}. 
 In general the
 intersection complex of an extension of two local systems, is not the
  extension of their intersection complex. 
  \subsection
 {The relative logarithmic complex  $\Omega_f^* \LL := \Omega^*_{X/V}(Log Y)\otimes \LL_X$}\label{rl} 

\

Let $f: X \to V$ be a smooth proper morphism of smooth complex algebraic varieties, and let $Y$ be an ``horizontal''  NCD in $X$, that is  a relative NCD  with smooth components over $V$. For each point $v \in V$, the fiber $Y_v$ is a NCD in the smooth fiber $X_v$. In this  case,  the various intersections  $Y_{i_1,\cdots,i_p}$ of $p$- components  are smooth over $V$ and $Y \to V$ is a topological fiber bundle. 

Let $U:= X-Y$, $ j: U \to X$. The  sheaf of modules $i_{X_v}^*\LL_X$  induced on each fiber  $X_v$, by the canonical extension  $\LL_X$,  is isomorphic to the canonical extension $ (i_{U_v}^*\LL)_{X_v}$ of the induced local system $ i_{U_v}^*\LL$. The family of cohomology spaces $H^i(U_v, \LL)$ (resp.$\H^i(X_v, j_{!*}\LL)$) for $v \in V$, form a variation of MHS. 
 The  logarithmic complex  $\Omega_f^* \LL := \Omega^*_{X/V}(Log Y)\otimes \LL_X$  (\cite{H}, section 2.22) satisfy, in the case of an horizontal NCD:
 $ i_{X_v}^* \Omega^*_{X/V}(Log Y)\otimes \LL_X \simeq \Omega^*_{X_v}(Log Y_v)\otimes (i_{U_v}^*\LL)_{X_v}$. 

 When $\LL$ underlies  an admissible graded polarized  variation of MHS $(\LL, W, F)$,  its restriction to 
the open subset $U_v $ in $X_v$  is also admissible. 
 \subsubsection
 {The relative logarithmic complex with coefficients $\Omega_f^* \LL$ for a smooth $f$} 
 
 The  image of the  filtrations  $ W$ and  $F$, by the map $\Omega^* \LL \to \Omega_f^* \LL := \Omega^*_{X/V}(Log Y)\otimes \LL_X$
 \begin{equation*}
\begin{split}
 F:= & \im (R^i f_*F_X \to R^i f_*(\Omega^*_{X/V}(Log Y)\otimes \LL_X)), \\
    W:= & \im (R^i f_*W_X \to R^i f_*(\Omega^*_{X/V}(Log Y)\otimes \LL_X))
\end{split}
\end{equation*}
  define a variation of MHS on  $R^i (f \circ j)_*\LL$ inducing at each point $v \in V$
the corresponding weight $W$  and Hodge $F$ filtrations of the MHS  on
$(\H^i(U_v, \LL), W, F)$. 
 \begin{prop}
i)  The direct image $R^i (f\circ j)_* \LL$ is a local system on $V$ and 
 $$R^i (f\circ j)_* \LL \otimes \OO_V  \simeq   R^i f_*(\Omega^*_{X/V}(Log Y)\otimes \LL_X)$$
ii)  Moreover, the connecting morphism in Katz-Oda's construction \cite{K-O} coincides with the connection on $V$
 defined by the local system $R^i (f\circ j)_* \LL$
 $$ R^i f_*(\Omega^*_{X/V}(Log Y)\otimes \LL_X) \xrightarrow{\nabla_V} \Omega^1_V\otimes  R^i f_*(\Omega^*_{X/V}(Log Y)\otimes \LL_X)$$
 iii)  The filtration $ F$,  is horizontal with respect to $\nabla_V$, while $W$ is locally constant,
 and   $ (R^i (f\circ j)_* \LL,  W,  F)$ is a graded  polarized variation of MHS on $V$.
 \end{prop}
Deligne's proof of (\cite{H}, proposition 2.28) extends in (i), as well the connecting morphism \cite{K-O} in (ii).
\paragraph{}
\begin{rem}
We refer to the remark \ref{Pb} for the case of a ''vertical'' NCD, that is when $f$ is smooth over the complement of a NCD $W \subset V$ and $Y:= f^{-1}W$.
\end{rem}

  \subsubsection
 {The relative Intersection  complex  $IC_f^* \LL$} 

\

We deduce from the bi-filtered complex $(IC^*\LL, W, F)$   the  $i$-th direct image 
$$(\LL^i, F, W):= (R^i f_* j_{!*}\LL, \,  Im \, (R^i f_* F \to R^i f_*  IC^*\LL), Im \,( R^i f_* W \to  R^i f_*IC^*\LL))$$
To prove that 
  {\it the filtrations $W$ and  $ F$  defined  on  $ \LL^i:= R^i f_*j_{!*}\LL$, form a structure  of  a graded polarized variation of MHS on $V$,}
  we consider  the image   complex of  $IC^*\LL$,
  
\n  $ IC_f^*\LL := \im \biggl(IC^*\LL \to \Omega^*_{X/V}(Log Y)\otimes \LL_X \biggr)$ with image filtrations $W  $ and $ F $,
 then:
 
  $ W_V:=  Im ( R^i f_* W \to R^i f_*(IC_f^*\LL))$ is locally constant on $V$, and

  $ F_V:=  Im ( R^i f_*F \to R^i f_*(IC_f^*\LL))$ on $R^i f_*j_{!*}\LL \otimes \OO_V \simeq  R^i f_*(IC_f^*\LL)$ is horizontal with respect to $\nabla_V$. 
  \begin{prop}\label{ri}
  i)  The direct image $R^i f_*j_{!*}\LL$ of the intersection complex is a local system on $V$ and 
 $$R^i f_*j_{!*} \LL \otimes \OO_V  \simeq   R^i f_*(IC^*_f \LL)$$

ii)  The connecting morphism in Katz-Oda's construction  coincides with the connection on $V$
 defined by the local system $R^i f_*\ilm$
 $$ R^i f_*(IC_f^*\ilm) \xrightarrow{\nabla_V} \Omega^1_V\otimes  R^i f_*(IC_f^*\LL)$$

 ii) The  de Rham complex defined by the connection on  $ (R^i f_*IC_f^*\LL,  W_V, F_V)$ underlies  a structure of mixed Hodge complex on $V$ defined by the  variation of MHS induced on  $R^i f_*\ilm$.  
 
 iii) The truncation  filtration $\tau$ induces a filtration on $\H^{i+j}(X, \ilm)$ compatible with the MHS and $Gr^\tau_i \H^{i+j}(X, \ilm) \simeq \H^j(V, R^i f_*IC_f^*\LL)$. 
 \end{prop}
 
  \section
{Logarithmic Intersection complex for an open algebraic variety}
 Let  $Z  = \cup_{i\in I_Z \subset I} Y_i \subset Y$ be a sub-divisor of $Y$, union of components of $Y$ with index in a subset $I_Z $ of $ I$, 
  $j_Z := (X-Z) \to X$, and $i_Z: Z \to X$.
 
 We construct a logarithmic  sub-complex   $IC^*\LL(Log Z) \subset \Omega^*\LL$  which is a realization of the direct image $ Rj_{Z *}(( j_{!*}\LL)_{\vert X-Z})$ (denoted  $ \Omega^*( \LL, Z)$ in \cite{
C} definition 8.3.31);  
from which we deduce by duality various logarithmic complexes with weight  and Hodge filtrations,  realizing  cohomological 
constructions associated to  $Z$ in $Y$ such as the functors:

\smallskip
 $ j_{Z !} ((j_{!*}\LL)_{\vert X-Z}), Rj_{Z *}( (j_{!*}\LL)_{\vert X-Z}), i_Z^*Rj_{Z *} ((j_{!*}\LL)_{\vert X-Z}), R i_Z^!j_{!*}\LL$ and $ i_Z^*j_{!*}\LL.$

 \smallskip
 
\n  We describe  a structure of mixed Hodge complexes (MHC) on  $IC^*\LL(LogZ)$
 that is $ Rj_{Z *} ((j_{!*}\LL)_{\vert X-Z})$,  and  $i_{Z*}i_Z^*Rj_{Z *} (j_{!*}\LL)_{\vert X-Z}) := i_{Z*}i_Z^*IC^*\LL(LogZ) $ (def . \ref{dual0} and lemma \ref{dual1} on  cohomology of the boundary of a tubular neighborhood  of $Z$).

\subsection
{ $IC^*\LL(Log Z) \simeq Rj_{Z *} ((j_{!*}\LL)_{\vert X-Z})$ }
 
 \

We suppose  $\LL$  locally unipotent as we refer to the  remark \ref{uni} in general.
The  subcomplex of analytic sheaves  $IC^*\LL(Log Z)$   of the logarithmic de Rham complex
 ${\Omega}^*_X (Log Y) \otimes {\LL}_X$ is defined below
 locally at
a point $x \in Y^*_M$  in terms of a set of
  coordinates $y_i$ for $ i\in M$  defining a set of equations of $Y_M$. 

\subsubsection{The Logarithmic Intersection complex: $IC^*L(Log Z)$}

\

Given an IMHS ($L, W, F, N_i, i \in M)$ and a subset $M_Z \subset M$,  we consider  
for each $J \subset M$ the subsets $ J_Z := J \cap M_Z $ and $J'_Z = J - J_Z$ such that $J = J_Z \cup J'_Z$, in particular  $ M'_Z := M - M_Z$.

 The correspondence which attach to each  index $J \subset M$ the subset  $N_{J'_Z}  L $ of $L$,  define a sub-complex $IC^*L(Log Z)$ of $\Omega^*L$  as a sum over $J \subset M$
\begin{equation}\label{lock}
IC^*L(Log Z) :=  s (N_{J'_Z}  L )_{J \subset M}\subset \Omega^*L
 \end{equation}
\begin{ex}
 On the $3$-dimensional  disc $D^3 \subset \C^3$, $Y = Y_1 \cup Y_2 \cup Y_3$  the NCD defined by the  coordinates $y_1y_2 y_3 = 0$; The local system $\LL$ is defined by a vector space $L$ with the action of $3$ nilpotent endomorphisms $N_i$.
  Let $ Z = Y_1\cup Y_2$ be   defined by $y_1y_2 = 0$, we consider 
  ($ L, N_i, i \in [1, 3]$) and $M_Z = \{1, 2 \}$, then $IC^*(Log Z)$  
 is defined by the diagram:
\[\begin{array}{ccccc}
  L \quad & \xrightarrow{N_1, N_2} & L \oplus L & \xrightarrow{N_1, N_2} & L\quad\\
   \downarrow N_3&  & \quad \downarrow N_3&  &  \downarrow N_3 \\
 N_3L & \xrightarrow{N_1, N_2} & N_3L \oplus N_3L
 & \xrightarrow{N_1, N_2} & N_3L\\
\end{array}
\]
with  differentials defined by  $N_i$ with  + or -
sign. Then, we have a quasi-isomorphism $ \R\Gamma (D^3-(D^3\cap Z), \ilm) \simeq  IC^*(Log Z)$.
\end{ex}
\subsubsection{Weight filtration}
\begin{defn}
The filtration $W^J$ on the term $ N_{J'_Z}  L $  of index $J$ in $IC^*L(Log Z) \subset \Omega^*L$ is induced by the embedding  $(N_{J'_Z}  L , W^J) \subset ( L, W^J)$
 
\centerline { $IC^*L(Log Z) :=  s (N_{J'_Z}  L )_{J \subset M}, \,\, W_k IC^*L(Log Z) := s (W^J_{k - \vert J\vert}N_{J'_Z}  L)_{J \subset M}$.}
\end{defn}

As in the case of $\Omega^*L$ and $IC^*L$, to compute $Gr^W_k IC^*L(Log Z) $
we introduce  the data $DR(IC^*L(Log Z):=$

\smallskip
\centerline{$  \{ N_{J'_Z}L, W^J,  F^J, I_{J, K}: N_{J'_Z} L \to  N_{K'_Z}L, N_{J- K}: N_{K'_Z}L \to N_{J'_Z} L\}_{K \subset J \subset M}$} 

\smallskip
 \n  satisfying  various properties as in the subsection \ref{data}, including the properties of distinguished pairs for consecutive terms $ N_i: N_{K'_Z}L \to N_{J'_Z} L$ for $J = K \cup i$
 (corollary \ref{open}).
 Set for each $J \subset M$
 \[ P ^J_k(N_{J'_Z} L):= \cap_{K \subset J, K \not= J}
 \hbox {Ker}\, (Gr I_{J,K}: Gr^{W^J}_k N_{J'_Z}L \to Gr^{W^K}_{k+ \vert J-K\vert} N_{K'_Z}L) \]
 then $ P^J_k(N_{J'_Z}L) \subset Gr^{W^J}_k N_{J'_Z}L$ has pure weight $k$ with respect to
the weight filtration $M(
\sum_{j \in J} N_j, L)$ and it is polarized.
\begin{cor}\label{IC3} 
We have:  $Gr^{W^J}_k N_{J'_Z} L
\overset{\sim}{\longrightarrow}  \oplus_{K \subset J} N_{J-K } P^K_{k + \vert J-K\vert}(N_{K'_Z} L)$,

\smallskip

and  $Gr^W_k(IC^*L(Log Z))\overset{\sim}{\longrightarrow} 
\oplus_{K \subset M} IC(P^K_{k-|K|}(N_{K'_Z}L )[-|K|]$.
\end{cor}
\subsubsection{Global statements} In terms of the diagram  $ (X-Y) \xrightarrow{j'} (X - Z) \xrightarrow{j_Z} X$,  $ j = j_Z \circ j' $ such that  $R j_{Z *}(( j_{! *}\LL)_{\vert X-Z}) \simeq R j_{Z *}  j'_{! *}\LL$. 

 \begin {defn}
[$IC^*\LL(Log Z)$] 
 For $J \subset M$, let $J_Z = J \cap I_Z, J'_Z  = J - J_Z$.
The fiber of $IC^*\LL(Log Z)$ is generated at  $x \subset Y^*_M$
   as an ${\Omega}^*_{X,x}$-submodule,
by the sections $\tilde {v} \wedge_{j \in J} \frac {dy_j}{y_j}$
for   $J
\subset M$ for all $v \in N_{J'_Z} L$.
\end {defn}
 Here $N_{\emptyset} = Id$. In other terms, in degree $k$, for  each subset $K \subset M$ of length $\vert K\vert = k$,  let $K' = (I-I_Z) \cap K$, and  $K_i \subset K'$ denote a subset of length $\vert K_i\vert = k-i$ and  $A := \OO_{X,x}$,  
the subcomplex  $ IC^*\LL(Log Z)_x \subset \Omega^*\LL_x$ is defined with the notations of (\ref{note}), as the $\Omega^*_{X,x}$-submodule:
\begin{equation*}
\oplus_ {|K|=k}
 \biggl(  y_{K'} A \widetilde L  \frac{dy_K}{y_{K}}+ \cdots   + \sum_{K_i \subset K'} y_{K_i} A (\widetilde {N_{K'-K_i}L}) \frac{dy_K}{y_{K}}+\cdots \biggr) 
\end{equation*}
 Next we check the isomorphism  $ (R j_{Z *}  j'_{! *}\LL)_x \simeq IC^*L(Log Z)$ of    
  the fiber at a point $x \subset Y^*_M$. 
 \begin {lem}
  We have:
  $(R j_{Z *} ( j_{! *}\LL)_{\vert X-Z})_x \overset{\sim}{\longrightarrow}  (IC^*\LL(Log Z)_x$
\end{lem}
 We check the proof on the terms of the
spectral sequence with respect to the truncation filtration on $j'_{!*}\LL$, that is  $E_2^{p,q} =  R^p j'_{Z *} H^q ((j'_{!*}\LL)_x)$ (after d\'ecalage in Deligne's notations), converging to $H^{p+q}(IC^*L(Log Z))$.

Indeed, the complex $IC^*L(Log Z)$ may be
   described   as the exterior algebra 
defined by the action  of $N_i, i \in M_Z$ on  the  Koszul complex $ s(N_{J'_Z}  L )_{(J'_Z) \subset (M'_Z)}$. 
That is the simple complex associated to the double complex 
 $ s ( s(N_{J'_Z}  L )_{(J'_Z) \subset (M'_Z)}, N_i)_ {i \in M_Z}$.

\centerline { $IC^*L(Log Z) :=  s ( \Omega (N_{J'_Z}  L, N_i,  i \in M_Z)_ {J'_Z \subset M'_Z} =   \Omega ( IC^* (L, N_i, i \in M'_Z), N_i, i \in M_Z)$.}

\n then the spectral sequence is associated to the double complex.

\smallskip

We define as well the modules $ P^J_k(N_{J'_Z}L) $,  the corresponding global sheaves $ \PP^J_k(\NN_{J'_Z}\LL) $ on $Y^*_M$ and deduce from the above corollary 
\begin{prop}
  The induced filtrations $W$ and $F$  on $IC^*\LL(Log Z)$ form a structure of MHC
 such
that the graded perverse sheaves for the weight filtration,
satisfy the
 decomposition property into intermediate extensions for all $k$ (corollary \ref{IC3})
 \[
 Gr^W_k IC^*\LL(Log Z) \overset{\sim}{\longrightarrow}  \oplus_{J \subset
 I} (i_{Y_{J_Z}})_* j_{!*} \PP^{J}_{k-|J|}(\NN_{J'_Z}\LL)[-|J|].
 \]
 where $j$ denotes uniformly the inclusion of $Y_{J}^*$ into $Y_{J}$
 for each $J  \subset I, Z:= \cup_{i\in I_Z} Y_i,  J_Z := J \cap I_Z$ and $Y_{J_Z}:= \cap_{i \in J_Z} Y_i$.
   In particular, for $J = \emptyset$
 we have $j_{!*}Gr^W_k \LL$, otherwise the summands are
 supported by $Z$.  
 
 A MHS is deduced on the hypercohomology $\H^*(X-Z, j_{!*}\LL)$.
\end{prop}
\begin{rem}
 \label{enlarge}
 For any NCD $Z$ such that $Z \cup Y $ is still a NCD, we may always suppose that $\LL$ is a  variation of MHS on  $X- (Y \cup Z)$ ( by enlarging $Y$)  and consider $ Z$  equal to a union of components of $Y \cup Z$.
 \end{rem}
\subsubsection{Thom-Gysin isomorphism}\label{TG}  
  Let $H $ be a smooth hypersurface intersecting transversally $Y $ such that $H \cup Y $ is a NCD, $i_H: H \to X, \, j_{H}: H- (Y\cap H) \to H$,  then  $ i_H^* j_{!*}\LL$ is isomorphic to  the (shifted) intermediate extension $ (j_{H})_{!*} (i_H^*\LL)$ of the restriction of $\LL$ to  $H$,
  with abuse of notations as we should write $ ((j_{H})_{!*} ((i_H^*\LL)[-1]))[1]$.
 The residue with respect to $H$, $R_H: i_H^* ( IC^*\LL(Log H)/ IC^*\LL) \overset{\sim}{\longrightarrow}  i_H^*IC^*\LL [-1]$ induce an isomorphism, inverse to Thom-Gysin isomorphism 

 \centerline {$ i_H^*j_{!*}\LL  [-1] \overset{\sim}{\longrightarrow}  i_H^!j_{!*}\LL [1] \overset{\sim}{\longrightarrow}   i_H^*(IC^*\LL(Log  H)/IC^*\LL)$.} 
 
  Moreover, if $H $  intersects transversally $Y \cup Z$ such that $H \cup Y \cup Z$ is a NCD,  we have    a triangle
\begin{equation*}  (i_H)_* R i_H^!IC^*\LL(Log Z) \to  IC^*\LL(Log Z) \to IC^*\LL(Log Z \cup H) \stackrel{[1]}{\rightarrow} \end{equation*}
 then we deduce an isomorphism of the quotient complex with the cohomology with support:
 $ (  IC^*\LL(Log Z \cup H)/  IC^*\LL(Log Z)) \overset{\sim}{\longrightarrow} R i_H^!IC^*\LL(Log Z)[1] $, induced by  the connecting isomorphism.
 
  We have also an   
  isomorphism of the restriction  $ i_H^* IC^*\LL(Log Z)$ with  the complex $ IC^* (i_H^* \LL) (Log Z \cap H)$ 
  constructed directly on $H$.  
  
   The residue with respect to $H$:  $IC^*\LL(Log Z \cup H) [1] \to 
 i_{H *}  IC^* i_H^* \LL (Log Z \cap H)$ vanish  on $IC^*\LL(Log Z)$ and induces 
  an inverse to  the Thom-Gysin isomorphism $ i_H^* IC^*\LL(Log Z) \overset{\sim}{\longrightarrow}  
 R i_H^!IC^*\LL(Log Z)[2] $.   The above constructions  are compatible with the filtrations up to  a shift in degrees.  
 
 \subsection{Duality and hypercohomlogy of the link}  The  link  at a point $v \in V$ refers to the boundary of a  ball with center $v$ intersecting   transversally the strata in a small neighborhood of $v$. It is a topological invariant, hence  its hypercohomology is well defined. In the case of a morphism $f:X \to V$, it refers to the hypercohomology of the boundary of a tubular neighboorhood  of $Z := f^{-1}(v)$, that is the inverse image of the link at $v$.
 
  The duality functor  $D$  in the derived category of sheaves of  vector spaces over $\Q$ (resp. $\C$)  (see \cite{BBD}, the references there and subsection \ref {6}),  
   and the cone construction,  are used here to deduce various  logarithmic complexes from the structure of mixed Hodge complex on $ IC^*\LL(Log Z) $. 
 \subsubsection{Duality}
  To develop a comprehensive theory of weights, we need to fix some conventions.
  The fields $\Q$ and $\C$ form a HS of weight $0$. If $H$ is a HS of weight $a$, we write $H(r)$ for the HS of weight $a-2r$.
 
 i) A Hodge complex  $K$ of $\Q-$vector spaces is  of  weight $a$ if its cohomology $H^i(K) $ is  a HS of weight $a+i$.  We set $\Q[r]$ for the complex with $\Q$  in degree $-r$ and zero otherwise, then $H^{-r}(\Q[r]) = \Q $ is a HS of weight $0$, hence $\Q[r]$ must be a HC of weight $r$.
 
 On a smooth proper variety $X$, the  dual of $H^i(X, \Q)$ has weight $-i$, while $H^{2n-i}(X, \Q) = \H^{-i}(X, \Q[2n]) $ has weight $2n-i$. We can write Poincar\'e duality with value in $\Q[2n](n)$ of weight $0$ , such that $H^{2n-i}(X, \Q(n)) = \H^{-i}(X, \Q[2n](n)) $ with weight 
 $-i$ corresponds to the  dual HS on $H^i(X, \Q)$.

  Hence, it is convenient to set $\omega_X := \Q_X[2n](n)$ for the dualizing complex in the category of  sheaves of $\Q-$ HS.  

ii) A  {\it pure variation} of HS $\widetilde \LL $ on a Zariski open subset $U \subset X$ of weight $w
(\widetilde \LL) = b$, has a dual 
 variation $\widetilde \LL^*$  of HS of weight $w (\widetilde \LL^*) = - b$.
   A polarization on  $ \widetilde \LL$ induces an isomorphism defined by the underlying non degenerate bilinear product
 $S: \widetilde \LL  \overset{\sim}{\longrightarrow}  \widetilde \LL^* (-b) $.
 
 As a complex of sheaves,  $\LL := \widetilde \LL[n]$ has weight $w = a = b+n $. 
Its dual in the derived category  $ D \LL := \HH om (\LL, \Q_X[2n] (n)) \overset {\sim}{\longrightarrow}  \widetilde {\LL}^*(n)[n]$,  has weight $w = -b+n-2n= -a$. Then, the  polarization  induces an isomorphism  $S:  \LL  \overset{\sim}{\longrightarrow}  D\LL (-a) $.
This isomorphism extends in the  category of topological constructible sheaves, into:

 $ j_{!*}\LL \simeq  j_{!*}(D \LL(-a) ) \overset{\sim}{\longrightarrow}    D(j_{!*}\LL)(- a):= R \HH om (\ilm, \Q_X[2n])(- a) $. 
\begin{lem}
For $X$ compact and  a shifted polarized  variation of HS  $\LL$ of weight $a$, we have the auto - duality isomorphism with  equal  weights
\begin{equation}
 \H^{i}(X, j_{!*}\LL)^* \overset{\sim}{\longrightarrow}  H^{-i}(X, j_{!*}\LL) (a)
 \end{equation}
  where $ w(\H^i(X,j_{!*}\LL)^*) = -(a+i)$ and $ w(H^{-i}(X, j_{!*}\LL)(a)) = a-i-2a$.
\end{lem}
\n The lemma follows from  the polarization $S: j_{!*}\LL \overset{\sim}{\longrightarrow}  D(j_{!*}\LL)(-a)$ and the duality 
\begin{equation*}
 \H^i(X, j_{!*}\LL)^* = H^{-i}(Hom(\R
\Gamma(X, j_{!*}\LL),
\Q) \overset{\sim}{\rightarrow}   H^{-i}(X, D(j_{!*}\LL))\overset{\sim}{\rightarrow}   H^{-i}(X, j_{!*}\LL)(a)
 \end{equation*}
 iii) \smallskip
  {\it Dual filtrations}.   The   dual of a  variation of MHS $(\widetilde \LL, W, F) $ on  a smooth Zariski open subset $U$, is   a variation of MHS $({\widetilde \LL}^*, W, F)$,  
  with the  filtrations dual to  $W$ and $F$ on $\LL$, such that $Gr^W_{-i}({\widetilde \LL}^*)\simeq 
 (Gr^W_i{\widetilde \LL})^*$.
 
In the bifiltered derived category of an abelian category with a dualizing functor  $D$,
  the dual of a bifiltered complex $(K, W, F)$, is denoted 
 by $(D K, W, F)$      with dual filtrations  defined   by:
 \[   W_{-i} D K:=
  D (K/W_{i-1}),  F^{-i}
D K:= D (K/F^{i+1}) \]
 such that:
$D Gr^W_i K \overset{\sim}{\longrightarrow}  Gr^W_{-i} D K$ and $D Gr_F^i K \overset{\sim}{\longrightarrow} 
Gr_F^{-i} D K$.

 As a complex of sheaves,  $(\LL := \widetilde \LL[n], W, F)$ has  its   weight increased by $n$, such that $ W^i(\LL) :=  W^{i-n}(\widetilde \LL)[n]$, and $ F_i(\LL) :=  F_i( \widetilde \LL)[n]$. Thus the filtration on its cohomology $H^{-n}(\LL)$, decreased by $n$, satisfy $(H^{-n}(\LL), W) = (\widetilde \LL, W)$.
 
The 
 dual in the bifiltered derived category
 
 \smallskip
 \centerline{ $ D (\LL, W, F) := \R
Hom ((\LL, W, F), \Q_U[2n](n)) $,}
  
 \smallskip
\n is a shifted variation of MHS ($ D{\LL}, W, F $) such that:
 $Gr^W_{-i}(D \LL, F)\simeq D (Gr^W_i \LL, F)$.
 
   Thus,   the
 dual of a mixed Hodge complex on a variety $X$  is a MHC.
 
  A graded polarization  on a  variation of MHS is defined by a family of polarizations  
   $S_i$ on $ (Gr^W_i\LL, F) $. It induces isomorphisms:
    
  $$ (Gr^W_i\LL, F)  \overset{\sim}{\longrightarrow}  D (Gr^W_{ i} \LL, F)(-i) \overset{\sim}{\longrightarrow}   (Gr^W_{-i} D \LL, F) (-i)  $$

 \subsubsection
  {Structure of mixed Hodge complex on  $R i_Z^!j_{!*} \LL$, $i_Z^* \ilm$ and $ j_{Z !}( \ilm_{\vert X - Z})$}\label{mc}
  We use the duality $D R i_Z^!j_{!*} \LL \simeq i_Z^*D j_{!*} \LL \simeq i_Z^* j_{!*} \LL(a)$ to deduce the weight filtrations on $i_Z^* j_{!*} \LL$.
   \begin{defn}
  [Dual filtrations] \label{df} 
  Let $\LL $ be a  shifted  polarized  VHS of weight $a$ on a Zariski open subset $U:= X-Y$ where $Y$ is a NCD in $X$ smooth projective, and $Z$ a closed sub-NCD of  $Y$. 

 i) 
 The filtrations  on $R i_Z^!j_{!*} \LL$ are defined by the isomorphism
 
  \smallskip
 \centerline{ $(R i_Z^!j_{!*} \LL, W, F)\simeq i_Z^*(IC^*\LL(Log Z) / IC^*\LL, W, F) [-1]$.}
  such that  $W_i(i_{Z *}R i_{Z}^{!}\ilm) =
  W_{i+1} (IC^*\LL (Log Z)/IC^*\LL)[-1] $. 
 
 \smallskip
 
 ii)   The   filtrations on the  complex $i_Z^*  j_{!*} \LL$ 
  are  deduced by duality  
  
   \smallskip
\centerline{$       (i_Z^*  j_{!*}  \LL, W, F) 
 \simeq D (i_Z^!  j_{!*}  \LL(a), W, F)$}
  
  \smallskip
 iii)   The  filtrations on  $ j_{Z !}  (\ilm_{\vert X - Z}) $  are deduced  by duality:
 
 \smallskip
 
\centerline{$( j_{Z !} ((\ilm)_{\vert X - Z}), W, F)\simeq D ( IC^*\LL(a)(Log Z) , W, F) $} 
\end{defn}
\begin{cor}\label{wi}
The weights $w$ satisfy the following inequalities:
\begin{equation*}
\begin{split}
&w \H^i (X-Z, \ilm) \geq a +i, w \H^i_Z(X, \ilm) \geq a+i\\
 &w \H^i ( Z, \ilm) \leq a+i, w \H_c^i (X-Z, \ilm) \leq a+i
\end{split}
\end{equation*}
\end{cor}
The weights of  $IC^*\LL(Log Z)$ as a MHC,  are $\geq a$ by construction. The weights of  $R i_Z^!\LL$ as a MHC, are $\geq a$ since the weights of $IC^*\LL(Log Z)/IC^*\LL$ are $\geq a+1$.

\n We use duality computations for the weights of $\H^i ( Z, \ilm)$:

\n $\H^i ( Z, Gr^W_j i_Z^*\ilm) \simeq \H^i ( Z, Gr^W_j D i_Z^!\ilm (a)) \simeq \H^i ( Z, D (Gr^W_{-j} i_Z^!\ilm (a))(-j)) \simeq \H^{-i} ( Z,  (Gr^W_{-j} i_Z^!\ilm (a))(-j)) ^*$ 

\n  where 
$w(Gr^W_{-j} i_Z^!\ilm (a))(-j)) = -a-j +2j$,
  hence $w \H^{-i} ( Z,  (Gr^W_{-j} i_Z^!\ilm (a))(-j)) ^* = a+i -j \leq a+i$ since it vanishes for 
  $j < 0$.
\begin{lem}
 Let $\LL$ be a polarized VHS of weight $a$ on a Zariski open subset $U$ and $Z$ a closed subvariety of  $X$ projective. 

 i) There exists a long exact sequence  of MHS
\begin{equation}
 \cdots
\to \H^i_Z(X, j_{!*} \LL) \to \H^i (X, j_{!*}\LL)
\to \H^i( X - Z,  j_{!*} \LL) \to \H^{i+1}_Z(X, j_{!*} \LL) \to \cdots
\end{equation}
with weights 

\n  $w (\H^i_Z(X, j_{!*} \LL) ) \geq a+i, w (\H^i (X, j_{!*}\LL))  = a +i, w (\H^i( X - Z,  j_{!*} \LL) ) \geq a+i$.

ii) We have a dual exact sequence  of MHS
\begin{equation}
\to \H^i_c ( X-Z, j_{!*} \LL) \to \H^i (X, j_{!*}\LL)
\to \H^i( Z,  j_{!*} \LL) \to \H^{i+1}_c (X-Z, j_{!*} \LL) \to \cdots
\end{equation}
with weights 

\n  $w (\H^i_c (X-Z, j_{!*} \LL))  \leq a+i, w (\H^i (X, j_{!*}\LL) ) = a +i, w (\H^i(Z,  j_{!*} \LL))  \leq a+i$.
\end{lem}
The lemma is proved first for  $X$ smooth projective and $U:= X-Y$ is the complement of  a NCD with $Z$ a sub-NCD of $Y$.

For example, the exact sequence is associated to the triangle $ Ri_Z^!j_{!*} \LL\to j_{!*}\LL
\to Rj_{Z *}( \LL_{\vert X - Z})$ and the inequalities on the weights follow from the corollary above.

 In the general case, the Hodge theory is deduced from the case of  NCD by application of the decomposition theorem as in the subsection \ref{gen}.
\begin{ex}
i)  For a pure $\LL$ of weight $a$,    $i_{Z *}R i_{Z}^{!}\ilm$,  is supported by $Z$, of weights $w_i \geq a $,  and 
 $ i_{Z *} i_{Z}^* \ilm$  of  weights  $w_i \leq a $. 
  The morphism $I$ induces the   morphism:
  $Gr^W_i i_{Z *}R i_{Z}^{!}\ilm \to Gr^W_i \ilm \to Gr^W_i i_{Z *} i_{Z}^* \ilm$ for $i = a$ and $0$ for $i \not= a$. 
  
 ii) A polarized VHS  on $\C^*$ corresponds to the action of a nilpotent endomorphism $ N$ on $L$. The  duality in the  assertion ii) at $0$  between $Gr^W_{-r+1} i_Z^* ( j_{!*}  \LL)) [-1] $ and $Gr^W_r i_Z^*(IC^*\LL(Log Z) /IC^*\LL)$
   corresponds to the duality between $ Gr^W_{-r+1} \ke N$ and $Gr^W_{r-1}L/NL $.
   For higher dimension, one needs to introduce strict simplicial coverings of a NCD to have such interpretation.
\end{ex}
\subsubsection{Structure on the
 Mixed cone.} We set $  W_{j+1}(K[1]) := (W_{j}K) [1]$. The mixed cone over a morphism $f: (K,W, F) \to (K', W', F')$ is the cone $ (C(f), W, F) $ with the corresponding filtrations on $K[1] \oplus K'$. 

 We  construct  a MHC as the mixed cone over a morphism of MHC. Due to the definition of morphisms of complexes up to homotopy, a special attention is needed for the compatibility of the constructions over $\Q$  and $\C$. This  can be achieved by using simplicial coverings
 as in \cite{HIII} (see section \ref{6} in our case). 
 
 \subsubsection
  { Structure of MHC on    $i_{Z * }i_Z^* Rj_{Z*} (( j_{!*}\LL)_{\vert X-Z})$ and  the cohomology of the boundary of a tubular neighborhood of $Z$ }\label{tube}
  
  \
  
Let $\KK := (j_{! *}\LL)_{\vert X-Z}$, we have  the following  triangles 
\begin{equation}
  j_{Z !} \KK \overset{can}{\longrightarrow}  R j_{Z *}\KK {\longrightarrow} 
 i_{Z *}i_Z^* R j_{Z *}\KK, \quad  R i_Z^! j_{! *}\LL \overset{I}{\longrightarrow}  i_Z^*j_{! *} \LL  {\longrightarrow}  i_Z^*  Rj_{Z *} \KK
 \end{equation}
 where $can$ is the composition of $ j_{Z !} \KK \xrightarrow{p }\ilm \xrightarrow{i} R j_{Z *}\KK $ and $I$ is the restriction of the composition of $ i_{Z *} R i_Z^! j_{! *}\LL \overset{q} \to \ilm \overset{r} \to   i_{Z *}  i_Z^*  Rj_{Z *} \KK $. 
 From which we deduce  two descriptions of the complex $i_Z^*  Rj_{Z *} \KK$, as  the restriction of the cone complex over the morphism $can$ or as  the cone over $I$. 
 
 We remark that the morphism $p$ is a morphism of perverse sheaves so that we have a triangle 
 $ \ke \, p  \to j_{Z !} \KK \xrightarrow{p }\ilm \xrightarrow{[1] }$, and  $ \ke \, p \simeq i_{Z *}  i_Z^*  Rj_{Z *} \KK[-1]$. A description of $ \ke \, p$ is given by lifting to a strict simplicial covering of $Z$ in the spirit of Deligne's simplicial construction, which is  dual to 
  the logarithmic construction of $IC^*\LL (Log Z)/IC^*\LL $. In particular we put on $ \ke \, p$ the dual filtrations $W$ and $F$ (subsection \ref{auto} and proposition \ref{App}).
  
  \begin{defn}
[$i_{Z * }i_Z^* Rj_{Z*} (j_{!*}\LL_{\vert X-Z})$]
\label{dual0}
 Let $can$ be defined as a bifiltered morphism of complexes $can:  j_{Z!} (j_{! *}\LL_{\vert X-Z})  {\rightarrow} Rj_{Z*} ( j_{!*}\LL_{\vert X-Z})$ with their filtrations $W$ and $F$.
 The mixed cone $C(can)$  over $can$ is isomorphic to $ i_{Z * }i_Z^* Rj_{Z*}(j_{! *}\LL)_{\vert X-Z}$.
 
 The cone depends only on the neighborhood of $Z$ and not on $X$, and we have a duality isomorphism
$$ i_Z^*Rj_{Z *}(\ilm_{\vert X - Z})[-1] \overset{\sim}{\longrightarrow}  D(i_Z^*Rj_{Z *}(\ilm_{\vert X - Z})).$$
\end{defn}
To realize the morphism $I$, we introduce  the shifted cone $C(i)[-1]$ over the embedding $i: IC^*\LL \to IC^*\LL (Log Z)$. We have an   isomorphism 
  $ C(i)[-1] \simeq  i_{X_v *}R i_{Z}^{!}\ilm $ such that   the projection $C(i)[-1] \to IC^*\LL$ is well defined and   compatible with the filtrations; then $I$ is the composition  with the restriction map  $IC^*\LL  \to i_{Z *} i_{Z}^* \ilm$ (see also definition \ref {dual}).

  Let $B_Z$ denote a small neighborhood of $Z$ in $X$,  then 
 $$\H^i( Z, i_{Z * }i_Z^* Rj_{Z*} (( j_{!*}\LL)_{\vert X-Z})) \simeq  \H^i(B_Z-Z, \ilm) $$
 and we have an exact sequence of MHS
\begin{equation*}
\cdots \to \H^i(X-Z, j_{! *}\LL) \to \H^i(B_Z-Z, \ilm) 
\to \H^{i+1}(X, j_! (j_{! *}\LL_{\vert X-Z})) \to \cdots
\end{equation*}
The MHS on the hypercohomology of $B_{Z}^*:= B_{Z}-Z$ is defined equivalently by the mixed cone over the morphism $can$ or the Intersection  morphism $I$
  \begin{lem}\label{dual1}
There exist an isomorphism of mixed cones $C (I) \overset{\sim}{\longrightarrow}  i_{Z}^* C(can)$  compatible  with both filtrations $W$ and $F$.

\end{lem}
The isomorphism is constructed by comparison with the cone  over the morphism  
$j_{Z!}  ( j_{! *}\LL_{\vert X-Z}) \oplus   i_{Z *}R i_{Z}^{!}\ilm  \to \ilm$, such that we have a diagram with coefficients in $\ilm$ where
 $X^* = X -  Z$
\begin{equation}\label{dia}
\begin{array}{ccccccccc}
 \H^{r }_{Z}(X)&&\to&&\H^{r }(Z)&&&&\H^{r +1}_{Z}(X)\\
 &\searrow&& \nearrow &   & \searrow&&\nearrow &\\
  && \H^{r}(X)&&& &
  \H^{r }(B_{Z}^*)&&\\
  &\nearrow&& \searrow &   & \nearrow&&\searrow &\\
  \H_c^{r }(X^*)&&\to&& \H^{r }(X^*) &&&&\H_c^{r+1 }(X^*)
   \end{array}
\end{equation}
\subsection
{Compatibility of the perverse filtration with MHS}\label{C6}
Let $V$  be a quasi-projective variety and     $K \in D^b_c (V, \Q)$  a complex with constructible cohomology  sheaves. The perverse filtration $\pt$ on $K$ induces a perverse filtration $\pt$
  on the hypercohomology  groups ${\H}^k(V,  K)$.  This topological filtration has an interesting 
 construction described by algebraic-geometric techniques in \cite{DM}. 
 
 We consider 
two  families  of hyperplanes  $\Lambda_*:= \Lambda_i$ and $\Lambda'_*:= \Lambda'_i$  for $ 1 \leq i \leq n$  in  $\P^N$, defining by intersection with $V$ two  families $H_*$  and $W_*$  of increasing closed sub-varieties  on $V$, 
where $H_{-r} := \cap_{1 \leq i \leq r} \Lambda_i \cap V$ and  $W_{-r} := \cap_{1 \leq i \leq r} \Lambda'_i \cap V$:
   $$
   H_*:  V = H_0 \supset H_{-1}  \supset \ldots  \supset H_{-n},\quad W_*:  V= W_0 \supset W_{-1}  \supset \ldots  \supset W_{-n} 
   $$
and $H_{-n-1}= \emptyset =  W_{-n-1} $. Let $h_i: (V- H_{i-1} ) \to V$ with indices in [-n, 0] denote the open embeddings. The following filtration 
\begin{equation}\label{delta}
\delta_p \H^*(V,K) := \hbox{Im} \biggl(\oplus_{i+j = p} \H^*_{W_{-j}}(V,(h_i)_! h_i^* K) \to  \H^*(V,K) \biggr) 
\end{equation}
 is defined in (\cite{DM}, remark 3.6.6), where  the main  result in  (\cite{DM}, Thm 4.2.1) states that for an affine embedding of $V$ into a projective space and for a  generic choice of both families depending on $K$ and the embedding of  $V$, the filtration $\delta$ is equal to the
perverse filtration $\pt$ up to a shift in  indices. 

\subsubsection{Relative case}  
   Let $f: X \to V$ be a projective morphism of complex  varieties where  $V$ is projective. Let $Y\subset X$ be a closed subvariety, $\LL$ a local system on $X-Y$, $j: (X-Y) \to X$,  $W$  a closed subspace of $V$,  $Z := f^{-1} W$  the inverse image of $W$ and $j_Z: (X-Z) \to X$. We  apply the above result in the following cases: $K_1:= Rf_* R(j_Z)_* (j_Z)* \ilm $ (including the case  $Z = Y$  or $ Z = \emptyset) $, and $K_2:= Rf_* R i_Z^! \ilm $ (the cases $Rf_* (j_Z)_! (j_Z)* \ilm $ and $Rf_* (i_Z)^* \ilm $ are dual). 
     
      The   perverse  filtration $\pt$ is defined on  $\H^k(X-Z,  j_{!*}\LL)$   via the isomorphism  $\H^k(X-Z,  j_{!*}\LL) \simeq {\H}^k(V, K_1)$ (resp. on $\H_Z^k(X,  j_{!*}\LL) $ via $\H_Z^k(X,  j_{!*}\LL) \simeq \H_W^k(V, K_2)$). 
   
\subsubsection
{Proof of the proposition \ref{global}}\label{pglobal}
 We suppose $f$ a fibration by NCD over the strata, $Y$  a NCD in $X$ and $\LL$  an admissible variation of MHS on $X-Y$,  such that   
  $Z := f^{-1} W$ and $Z \cup Y$ are NCD and $W$ union of strata.  We consider the  bifiltered logarithmic complex   $K_1 = Rf_* IC^*\LL (Log Z)$ (resp.  $K_2 $ as in definition \ref{df} i).

We check  that the induced filtration ${\pt}$  on $\H^k(X-Z,  j_{!*}\LL)$ (resp.   on $\H_Z^k(X, \ilm)$) is compatible with the MHS, that is a  filtration  by sub-MHS  as stated in the propositions \ref{global}. 
    
Starting with the complex $(IC^*\LL(Log Z), W, F)$ (\ref{note}),  the system of truncations maps $\cdots \to\pt_{\leq -i} Rf_* IC^*\LL(Log Z)\to \pt_{\leq -i+1} Rf_* IC^*\LL(Log Z) \cdots$ is isomorphic in $D^b_c(V, \Q)$ to  a system of inclusions maps $\cdots \to P^i K \to P^{i-1} K \cdots$ where we  suppose $K$, all $P^iK$ and $Gr^i_P K$ are injective complexes (\cite{BBD}, 3.1.2.7), then the filtration $P$ on $K$  is defined up to unique isomorphism in the category of filtered complexes $(D^b_c(V, \Q), F)$. 
 
We need to consider not only a fine Dolbeault resolution 
of  $(IC^*\LL(Log Z), W, F)$ to deal with Hodge theory, but an acyclic resolution for the filtrations $P, W$ and $F$ on the same complex  on $V$. 

Let $(K_1, P, W, F)$ be  a representative complex of $ Rf_*IC^*\LL(Log Z), W, F) $ with acyclic filtrations  on $V$ and where $P$ represents the perverse filtration.
To prove that $R\Gamma( V, Gr^P_iK_1, W, F )$ is a MHC, or $\H^r(V, Gr^P_iK_1, W, F )$ is a MHS, we need to prove that the filtration $\delta$ defined by the formula \ref{delta} above is a filtration by sub-MHS.

\begin{lem}
 i) Let $f: X \to V$ be a fibration by NCD over the strata, and $W$ a closed  subvariety  of $V$  such that $Z := f^{-1}(W)$ is  a NCD in $X$. Then the MHS on $ \H^j( X, IC^*\LL(Log Z))$ (resp. $  \H^j( X, R i_Z^!\ilm)$ is well defined, as well on the terms of the  filtration $\delta$ (\ref{delta}) for $Rf_* IC^*\LL(Log Z)$ (resp. $Rf_* Ri_Z^! \ilm $).  Consequently the perverse  filtration
  $\pt$  on $\H^j(X-Z), \ilm)$ (resp. $\H^j_Z(X, \ilm) $) is compatible with the MHS.
 
 ii) The dual statement  and the corresponding {\em filtration $\delta$} show that the perverse  filtration
  $\pt$  on $\H_c^j(X-Z, \ilm)$ )(resp.  $\H^j(Z, \ilm)$) is compatible with the MHS.
  
  \end{lem}
\begin{proof}
We  reduce the proof to the case of an affine embedding of $V-W$.
Let $\pi: \widetilde V \to V$ be a blowing up of $W$ such that the embedding of $U:= V-W$ in $\widetilde V$ is affine.  Since $f^{-1}(U)$ is a NCD in $X$, the morphism $f: X \to V$ factors as $\pi \circ g$ for $g: X \to \widetilde V$. We apply the lemma to $U$  in $\widetilde V$, to construct the families $H_*$ and $W_*$ as intersection of two general families of hypersurfaces $\Lambda_*$ and $\Lambda'_*$ in $\widetilde V$ such that the union of their inverse image $g^{-1}\Lambda_*$ and $g^{-1}\Lambda'_*$, the  existing NCD $Y$ and $f^{-1}(U)$, form together a large  NCD in $X$. In particular,   the  intersections of the  various inverse  image  $ f^{-1}(H_i \cap W_j ) $  are {\it transversal and intersect transversally}   the various NCD  in $X$.

Then, we consider this large NCD to construct  $IC^*\LL(Log Z) \subset \Omega^*\LL$ with the weight and Hodge filtrations on it.

  The inverse image of $H'_* := g^{-1}(H_*)$ and $W'_* := g^{-1}(W_*)$ 
 are transversal in $X$ and the restrictions
 $i_{H'_i}^*\LL$  are shifted VHS on $H'_i \cap Y$ for various indices $i$,  whenever $\LL$ is a shifted VHS.  Let $h'_i: (X-H'_i) \to X$,  we  construct by the techniques of sections \ref{St1}, \ref{4},  and  Thom-Gysin isomorphisms (subsection  \ref{TG}), bifiltered logarithmic mixed Hodge complexes   $R i_{W_j}^! (h'_i)_!{h'_i}^*\ilm$, such that  the embedding in the formula \ref{delta} is compatible with MHS.  The proof for all other hypercohomology functors is similar.
\end{proof} 
\subsubsection
{The proof of the proposition  \ref{mainp} is similar but does not follow from \cite{DM}}\label{Link1}
The hypercohomology  of  a   tubular neighborhood with the  central  fibre deleted
  $B_{X_v}^* = B_{X_v} - X_v$  with coefficients  in $\ilm$,
  is defined  (definition \ref{dual0}, lemma \ref{dual1}, see also  \ref{dual})   by the cone  over  the morphism 
  $ can : j_{X_v!}  ( j_{! *}\LL_{\vert X-X_v})  \to Rj_{X_v*}  (j_{! *}\LL_{\vert X-X_v})$ which is a realization of   $i_{X_v * }i_{X_v}^* Rj_{X_v*} (( j_{!*}\LL)_{\vert X-X_v})$. 
   
   The mixed cone $ (C(can), W, F)$ is endowed with the filtrations $W$ and $F$ deduced from the filtrations  on its components.

We remark that the preceding characterization  of the perverse filtration, as a filtration $\delta$ (\ref{delta}),
 apply to a punctured ball $B_v - \{v\}$ as $B_v$ is Stein. In fact the preceding lemma 
apply to two hyperplane sections  in general position containing $v$.  As a consequence the filtration $\delta$ on $\H^*(B_{X_v}^*, \ilm)$ coincide up to indices with the  perverse filtration and moreover it consists of sub-MHS.

\subsection
{Complements on Duality }\label{6}

We deduced from $\Omega^*\LL$, by duality, various logarithmic complexes. In  this section,   we develop what we may call duality calculus to  describe the filtered realizations of the complexes $ i_Y^*j_{!*} \LL$,  $Ri_Y^!j_{!*} \LL$,  and $j_{!}{\LL}$ obtained from the filtered structure  of $  \Omega^*\LL$.

 An explicit { \it filtered version of the auto-duality isomorphism of the Intersection  complex }:
$j_{!*}{\LL} \overset{\sim}{\longrightarrow} R Hom(j_{!*}{\LL},\Q_X[2n]) $ is also described (proposition \ref{App}). 

For a  local system $\LL$ with singularities along  a normal crossing divisor $Y$, the dual of  $\Omega^*\LL$ with value in  Grothendieck's residual complex is described, in terms of residues and  Grothendieck's symbols, using the simplicial covering space of $Y$ (formula \ref{formula}).

In  the case of a NCD sub-divisor  $Z\subset Y$, we deduce   from  $i_Z^* \Omega^*\LL(Log Z)$ various logarithmic  complexes needed  in the proof of the decomposition, 
in particular a cone construction in the case of the tubular neighborhood.
The description  presents similarity with Deligne's simplicial techniques in \cite{HII}. 

In order to deduce a MHC by the cone construction over a morphism of MHC, the compatibility between $\Q$ and $\C$ coefficients must be well defined, not only up to homotopy, which is achieved by the simplicial constructions.  

 \subsubsection
{Serre  duality} On a smooth
compact  complex algebraic variety $X$ of dimension $n$,
Serre duality  for  locally free modules, has value in the  
 canonical sheaf of modules $\Omega_X^n$ \cite{Se}. 
 Since  the logarithmic complex $ \Omega^*\LL := \Omega_X^* (Log Y)\otimes_{\OO_X}\LL_X $ consists of locally free sheaves, 
  its dual can be  defined in the following way.
   
Let  $y_1, \ldots,y_p$ denote a subset of a coordinate set of parameters defining an equation   of $Y$ locally  at a point $x$. The ideal $\II_Y $ of $Y$ is generated at $x$ by the product $ (y_1 \ldots y_p) := \II_{Y,x}$. 
The wedge product defines a map
 
 $\II_Y  \Omega_X^j(Log Y)\otimes \Omega_X^{n-j}(LogY) \to  \Omega_X^n$ for all $j$.
  
\n  In each degree the sub-module $ \II_Y  \Omega^i \LL := \II_Y  (\Omega_X^i (Log Y)\otimes_{\OO_X}\LL_X)$ is contained in $ \Omega_X^i \otimes_{\OO_X}\LL_X$;  it is acyclic at points of $Y$  and it  is  a sub-complex of the Intersection sub-complex $IC^*\LL$; 
We use  the product $S: \LL \otimes \LL \to \C$, to define a morphism $\phi:
\II_Y \Omega^j\LL
 \xrightarrow{\phi} \HH om_{\OO_X}  ( \Omega^{n-j}\LL,  \Omega_X^n)$
  as follows: for $ \omega  \in \II_Y  \Omega^j\LL$ and  $  \alpha \in \Omega^{n-j}\LL$
\begin{equation}\label{formule}
  \omega  \otimes \widetilde v  \mapsto 
 \{ \alpha  \otimes \widetilde {v'}  \mapsto  S( v, v')
  \,  \omega \wedge \alpha. \}
\end{equation}
as  $\omega \wedge \alpha \in \Omega_X^n$ for $ \omega  \in \II_Y  \Omega^j\LL$. We check locally term by term, that $\phi$  is an isomorphism (with $\C$-linear differentials and not $\OO_X$-linear). As the modules of the 
logarithmic complex are locally free, the duality isomorphism  apply 
\begin{equation}
 \II_Y  \Omega^*\LL \overset{\sim}{\longrightarrow}   \HH om_{\OO_X}  ( \Omega^*\LL,  \Omega_X^n )) \overset{\sim}{\longrightarrow}   D R j_* \LL \overset{\sim}{\longrightarrow}  j_! \LL.
\end{equation}
Let $i_Y: Y \to X$,  we remark  the following  quasi-isomorphisms:

\centerline {$i_{Y *} i_Y^* Rj_* \LL \overset{\sim}{\longrightarrow}   \Omega^*\LL / \II_Y \Omega^*\LL, \quad
i_{Y *} i_Y^* j_{! *} \LL \overset{\sim}{\longrightarrow}   IC^*\LL / \II_Y\Omega^*\LL$.}
\subsubsection
{Grothendieck-Verdier duality  and the Dualizing complex}
  
  \

 Grothendieck introduced the concept of  residual complex  $K^\b_X $ on complex  algebraic varieties \cite{Gr2, Ha0} to extend  duality theory to the category of coherent sheaves; the complex $K^\b_X $ consists  of injective sheaves. In the case of a smooth $X$, it is an injective resolution of $\Omega_X^n[n]$.  
  
  The complex  $\C[2n]$ is  dualizing in the category of sheaves of $\C$-vector spaces on  $X^{an}$ (the terminology is related to Poincar\'e duality).  Both duality theory,  in the category of sheaves of  modules and in the category of $\C$-vector spaces, are related.
  To explain the relation between  Serre duality and Poincar\'e duality, the category of complexes of $\OO_X$-modules with differential operators  of order $\leq 1$  is used in (\cite{He}, \S 2)  
  where a graded module over the algebra  of differential forms $\Omega_X^*$ (resp. $\Omega_{X^{an}}^*$ in the analytic case) is associated to such  complexes. A basic formula (\cite{He}, proposition (2.9), 3) establishes for each $\Omega_X^*$-graded module $\FF^\b$ an isomorphism
  \begin{equation}\label{gra}
\HH om_{\Omega_X^*} ^k(\FF^\b, \Omega_X^*) \simeq \HH om_{\OO_X}(\FF^{n-k}, \Omega_X^n)
\end{equation}
Then Poincar\'e duality on $X^{an}$ between the  hypercohomology $\H^*(X^{an}, \Omega_{X^{an}}^*)$
  and $\hbox{Ext}^*_{\Omega_{X^{an}}^*}( \Omega_{X^{an}}^*, \Omega_{X^{an}}^*[2n]) \simeq \hbox{Ext}^*_{\OO_{X^{an}}}( \Omega_{X^{an}}^*, \Omega_{X^{an}}^n[n])$ follows from Serre duality at the level of associated spectral sequences  (\cite{He}, Introduction).

As a consequence of Grothendieck's algebraic de Rham theory \cite{Gr}, we may introduce the  algebraic  complex  $\HH om_{\OO_X} (\Omega_X^*, \Omega_X^n[n])$ in duality theory \cite{Ha, E0}, although it is not a resolution  of $\C[2n]$ in the algebraic case. 

The dual of  a complex of $\OO_X$-modules $\FF^\b$ with differential operators of order $\leq 1$, is the complex  $\HH om_{\OO_X}^\b (\FF^b,  K^\b_X)$ with   differential operators of order $\leq 1$, with value in the injective resolution $K^\b_X$ of $\Omega_X^n[n])$. 

 \begin{defn}[Dual complex] \label{dualc} 
 The dual complex of $\Omega^*\LL $ is  the complex with   differential operators of order $\leq 1$ 
 
 $\HH om_{\OO_X}^\b (\Omega^*\LL,  K^\b_X) \simeq R \HH om_{\OO_X}^\b (\Omega^*\LL,  \Omega_X^n[n]) $. 
 
\n  The dual filtrations of $W$ and $F$  are defined with value in $ K^\b_X$.
 \end{defn}
 
 \begin{rem} \label{dualan} 
 i) On the analytic variety $X^{an}$, the complex $\OO_{X^{an}}\otimes K^\b_X$ is dualizing for  complexes of analytic sheaves associated to algebraic sheaves. However, we continue to use the notation $K^\b_X$ in both algebraic and analytic cases.
  A dualizing complex exists for  analytic coherent sheaves \cite{R}, but it is not needed here.
 
 ii)
The relation between the dualizing complex and currents  is discussed in  (\cite{R-G},\S 5). In fact $K_X^\b$ embeds into  
 currents on $X^{an}$(see also  \cite{E3}).
   \end{rem}
  \subsubsection{The simplicial covering of a NCD and the duality map}
  To construct a  duality map  for MHS, we use  a simplicial covering of $Y$.
  
  We attach to $Y$ with normally crossing components $Y_i$ indexed by $i \in I$,  the  strict simplicial covering $Y_\b$
    (resp. $Y^+_\b$)   defined as follows:   $ \pi: Y_\b \to X$ is indexed by  $\N^*$ such that $Y_n := \coprod Y_A,  A \subset I, \vert A \vert = n, n \not =  0$ (hence $A \not= \emptyset $) where $Y_A := \cap_{i\in A}Y_i$, with  connecting morphisms   defined by  the natural embeddings for each inclusion $A \subset B \subset I$; 
resp. we add $ Y_0 := Y_\emptyset := X$ with index $0$ and write $ \pi_+: Y^+_\b \to X$.
 For example, in the case of  two components:
   $Y_i \cup Y_j$,  $Y^+_\b$ = $( Y_2 :=  Y_{i,j} \to Y_1:= Y_i \amalg Y_j \to  Y_0 := X )$.  

\n A complex $K \in D ( X, \Z) $ on $X$, lifts to a simplicial complex $ \pi^*  K $ on $Y_\b$
(resp. $ \pi_+^*  K $ on $Y_\b^+$). We write $\pi_* \pi^*  K$ or  $s(\pi^* K)$ (resp. $\pi_{+ *} \pi_+^*  K$ or  $s(\pi_+^* K$)  for the simple associated complex, then we have a quasi-isomorphism   $ i_Y^* K  \overset{\sim}{\longrightarrow}   s(\pi^* K) $  (resp. $  j_! (K_{\vert X-Y})  \overset{\sim}{\longrightarrow}  s(\pi_+^* K)$ where $ j: (X-Y) \to X$).
In some sense, simplicial coverings correspond in topological terms, 
to simplicial resolutions of complexes.
\begin{defn}\label{point}
 The  simplicial topological inverse image by  $\pi_+^*$ (resp. $\pi^*$) of  $  \Omega^*\LL$ (resp. $  IC^*\LL$) is denoted by  $\Omega^*\LL_{\vert Y^+_\b} $   on $Y^+_\b$
(resp.   $IC^*\LL_{\vert Y^+_\b} $ on $Y^+_\b$, and  $\Omega^*\LL _{\vert Y_\b}$,   $IC^*\LL _{\vert Y_\b}$ on $Y_\b$).
 The  direct images by $\pi_{+ *}$  on $X$ (resp. $\pi_*$ on $Y$) are: 
 \begin{equation*}
  \begin{split}
\Omega_+\LL(\b) := s(\Omega^*\LL_{\vert Y^+_\b}) \simeq  j_!\LL, &\quad  IC_+\LL(\b) := s(IC^*\LL_{\vert Y^+_\b})\simeq   j_!\LL\\
 \Omega\LL(\b) := s(\Omega^*\LL_{\vert Y_\b})\simeq   \Omega^*\LL_{\vert Y}, &\quad
   IC\LL(\b) := s(IC^*\LL_{\vert Y_\b})\simeq   IC^*\LL_{\vert Y}.
\end{split}
\end{equation*}
\end{defn}
For example, in dimension $2$, $ \Omega\LL(\b)$ is defined by the double complex
 
\n $\Omega^*\LL_{\vert  Y_1} \oplus \Omega^*\LL_{\vert Y_2} \to \Omega^*\LL_{\vert Y_{1,2}}$, while $\Omega_+\LL(\b)$   is defined by $ (\Omega^*\LL \to \Omega\LL(\b)[-1])$.

We have an exact sequence: $ 0 \to  IC\LL(\b)_{\vert Y}[-1] \to  IC_+\LL(\b) \to   IC^*\LL \to 0$.

 \subsubsection
{Duality map and  Grothendieck's symbols}\label{symbol}
  
 We  construct a dualizing map with value in the residual complex by local computations with Grothendieck's symbols
 \begin{equation}\label{duall} 
 \Omega_+\LL (\b) := s(\Omega^*\LL_{\vert Y^+_\b}) \xrightarrow{\phi}  \HH om^\b_{\OO_X}  ( \Omega^*\LL, K_X^\b).
  \end{equation}
    Let $z$ denotes a 
 generic point of an irreducible sub-variety $Z$ of codimension $i$ in $X$. The module of local cohomology  $H_z^i (\Omega_X^n)$ defines a constant sheaf  on $Z$ denoted by  $i_zH_z^i (\Omega_X^n)$,
then $K_X^i =  \sum_{{\rm codim}. Z = i}  i_zH_z^i (\Omega_X^n)$ in degree $i$. 

A set of local equations  $ y_j$ for $ j \in [1,i]$ of $Z$ at $z$ generates the ideal of $Z$ at $z$, so that the  local cohomology group  is computed by an inductive limit of Koszul resolutions defined by various $m-$th powers of $ y_j$, hence an element $\overline w \in H_z^i (\Omega_X^n)$ is written  as a Grothendieck symbol  $ \left[
\begin{array}{c}
\omega   \\
  y_1^m,\ldots,y_i^m   
\end{array}
\right] $. 

Since $IC^*\LL$ consists of modules which are not  locally free, the dualizing map has value in 
the injective resolution $K_X^\b$ of $ \Omega_X^n[n]$.  In the following computation, we need only the terms with value  in the subcomplex ${\underline \Gamma}_Y K^\b_X$ with support in the algebraic divisor  $Y$.
 A duality calculus based on Grothendieck's  symbols \cite{Gr2, Ha0, E0}  leads to the construction of  the dual weight filtrations and the duality map $\phi$ (formula \ref{formula} below). 

 Let $x_0 \in X$ be the generic point of $X$,  the wedge product defines a map :
 $ \Omega_X^i(Log Y)\otimes \Omega_X^j(LogY) \to  i_{x_0}\Omega_{X,x_0}^{i+j}$ for $i+j = n$ with value in the constant sheaf of rational $n$-forms. 
 We define the duality map as:
\begin{equation}\label{for}
\phi:  \Omega^*\LL \otimes_{\OO_X} \Omega^*\LL \to  i_{x_0}\Omega_{X,x_0}^*[n]:  (\omega  \otimes \widetilde v ) \otimes (\alpha  \otimes \widetilde {v'})  \mapsto  S( v, v')
 \omega \wedge \alpha  
\end{equation}
where the product $S: \LL \otimes \LL \to \C$ is defined by the polarization, and $\omega $ and  $ \alpha$ have logarithmic singularities.  This map is compatible with the differentials of forms, as :
$d ( S( v, v')
 \omega \wedge \alpha ) =  S( v, v')
( (d \omega) \wedge \alpha  + (-1)^i 
 \omega \wedge (d \alpha) $ since $ S( v, v')$ is constant as $v$ and $ v'$ are flat.
In the analytic setting,  we consider  the complex of analytic sheaves $K^\b_X \otimes \OO_{X^{an}}$ as a dualizing complex on $X^{an}$. As the NCD $Y$ is algebraic, the denominator terms are algebraic  and $\OO_{X^{an}}$ is flat over $\OO_X$.
  For example on $X^{an}$, we use the sheaf  $\OO_{X^{an}}\otimes i_{x_0}\Omega_{X,x_0}^*$, instead of  $i_{x_0}\Omega_{X,x_0}^*$. But we omit the corresponding notation for $X^{an}$.

However, to construct a map:
$  \Omega^*\LL \otimes_{\OO_X} \Omega^*\LL   \to  K^\b_X $, we view $ \Omega_{X,x_0}^n $
as the term $ H^0_{x_0}(\Omega^n_X)$ of the complex $K_X^\b $. This map is no more compatible with the total differential since  we introduce  the differential $d_1:  H^0_{x_0}(\Omega^n_X) \to \oplus_{y_i} H^1_{y_i}(\Omega^n_X)$ where the codimension of the closure $Y_i$ of $y_i$ is  $1$.

 For example, in the case of 
$
\Omega^i\LL
 \xrightarrow{\phi} \HH om_{\OO_X}  ( \Omega^{n-i}\LL,  \Omega_{X,x}^n)$

$d_1 (\omega \wedge \frac{d y_{i_1}}{ y_{i_1}}\wedge \cdots \wedge \frac{d y_{i_k}}{y_{i_k}}) = \oplus_{s}  \left[
\begin{array}{c}
(-1)^s  \omega \wedge \frac{d y_{i_1}}{ y_{i_1}}\wedge \cdots d y_{i_s} \cdots \wedge \frac{d y_{i_k}}{y_{i_k}}   \\
  y_{i_s}   
\end{array}
\right] $

(In comparison with  the case of currents instead of  $K_X^\b$  (\cite{Dol}),  the defect of compatibility is corrected  there by the residue in the formula 
  $\phi (D \omega) - D \phi (\omega) = \phi (Res \omega)$ (\cite{Dol} formula 9)). 

 We  extend the map  $\phi$ to the terms $\Omega^*\LL_{\vert Y^+_\b}:= \pi_+^*(\Omega^*(LogY)\otimes_{\OO_X}\LL_X)$ on  the semi-simplicial covering $ Y^+_\b$. For  sections $\alpha  \in \Omega^i\LL_{\vert Y_{i_1, \ldots, i_k}}$ and $\beta \in \Omega^{n-i}\LL$, 
we use in the next formula, the symbols in $ Ext_{\OO_X}^k (\OO_{Y_{i_1, \ldots, i_k}}, \Omega^*\LL)_{y_{i_1, \ldots, i_k}} $
 considered  as a subset of  $H^k_{y_{i_1, \ldots, i_k}}(\Omega^*_X)$ where $y_{i_1, \ldots, i_k}$ is a generic point of a component of $\cap_j Y_{i_j}$ and $y_{i_1} \cdots y_{i_k} \alpha \wedge \beta$ is regular at $y_{i_1, \ldots, i_k}$:
 \begin{equation}\label{formula}
\begin{split}
 ( \Omega^i\LL)_{i_1, \ldots,i_k}& \otimes_{\OO_X}  \Omega^{n-i}\LL  \xrightarrow{\phi_{i_1, \ldots, i_k}}  H^k_{y_{i_1, \ldots, i_k}}(\Omega^n_X) \\
  \alpha \otimes \beta & \mapsto \epsilon(i_1, \ldots, i_k) \left[
\begin{array}{c}
y_{i_1} \cdots y_{i_k} \alpha \wedge \beta  \\
  y_{i_1}, \cdots, y_{i_k}   
\end{array}
\right]  \in H^k_{y_{i_1, \ldots, i_k}}(\Omega^n_X).
\end{split}
\end{equation}
 that is the class of  $y_{i_1} \cdots y_{i_k} \alpha \wedge \beta \in \Omega^n_{X, y_{i_1, \ldots, i_k}}$  with value in the module of $n$-forms at the generic point $y_{i_1, \ldots, i_k}$ modulo the sub-module generated by $y_{i_1}, \cdots , y_{i_k}$. This class is   also called $ Res_{Y_{i_1, \ldots, i_k}}(\alpha \wedge \beta )$.

  We deduce from the above, an extension of  $\phi$ (formula \ref{formula}) to $  \Omega_+\LL (\b)$ on $X$ (definition \ref{point}),  
 defining the duality isomorphism  of complexes
  \begin{equation}\label{dual11}
  \begin{array}{ccc} 
   \Omega_+\LL(\b) &\xrightarrow{\phi}  & \HH om^\b_{\OO_X}  (\Omega^*\LL, K_X^\b)\\
   \downarrow \simeq&  & \downarrow \simeq \\
 j_!\LL& \overset{\sim}{\longrightarrow}& D_X R j_* \LL\\
\end{array}
\end{equation}
 \subsubsection
 {The complex $\ke \,\phi (\b)$ kernel of $\phi$ and the auto-duality of $j_{!*}\LL$}\label{auto}

We   construct a filtered version of  the triangles of perverse sheaves 

\smallskip 
\centerline {$i_{Y *} i_Y^*\ilm[-1] \to j_!\LL \to D \ilm$ dual to $IC^*\LL \to \Omega^*\LL \to \Omega^*\LL / IC^*\LL$}

\smallskip 
\n We deduce from  the  duality map $\phi$  an induced map
\begin{equation}\label{phi2}
 IC_+\LL (\b) := s(IC^*\LL_{\vert Y^+_\b}) \xrightarrow{\phi}  \HH om^\b_{\OO_X}  ( IC^*\LL, K_X^\b).
 \end{equation}
 The map $\phi$ does not vanish on the subcomplex $s(IC^*\LL_{\vert Y_\b})[-1] \simeq i_Y^* IC^*\LL [-1]$. 
 The kernel of $\phi $, denoted  $\ke \,\phi (\b)[-1]$   is a sub-complex  of  $s(IC^*\LL_{\vert Y_\b})[-1]$. 
 \begin{prop}\label{App}
[$i_Y^* IC^*\LL\simeq \ke \,\phi (\b)$ and  auto-duality]\label{autodual}
\

\n i) The kernel of $\phi $ (formula \ref{dual11}, \ref{phi2},  \ref{formula}) satisfy

\smallskip 
\centerline { $\ke \,\phi (\b)\overset{\sim}{\longrightarrow}s(IC^*\LL_{\vert Y_\b})
\overset{\sim}{\longrightarrow}  i_Y^* IC^*\LL$.}

and  $\ke \,\phi (\b) [-1]$ is perverse.

\n ii)  We have a filtered auto-duality isomorphism induced by $\phi$: 
 \begin{equation}\label{phi1}
 \begin{split}
  IC^*\LL \overset {\sim}\rightarrow  IC_+\LL(\b)/&\ke \,\phi (\b)[-1]  \xrightarrow{\phi } \HH om^\b_{\OO_X}  ( IC^*\LL, K_X^\b)  \\
\ke \,\phi (\b) [-1] \xrightarrow{\phi }&  \HH om^\b_{\OO_Y}  ( \Omega^*\LL/IC^*\LL, K_Y^\b)
  \end{split}
\end{equation}
where $K_Y^\b :=  \HH om^\b (\OO_Y, K_X^\b)$ is the dualizing complex on $Y$.
\end{prop}
The complex $\ke \,\phi (\b) [-1]$ is denoted by $\ke \, p$ in subsection \ref{tube}.  
 
With the notations of  subsection \ref{note}, let  $Y_J $ denote a component space of  the simplicial space $ Y^+_\b$, $x \in Y^*_M$ and $L$ the fiber of the local system. To each subset $K_i \subset K \subset M$ and to each subspace 
$ V := (\ke \, N_{J  - J \cap K}) \cap  (N_{K-K_i}L)$,  we associate the submodule 
 $A \widetilde V := A (\widetilde {(\ke \, N_{J  - J \cap K}) \cap ( N_{K-K_i}L})) \subset \LL_x$.
 
\begin{defn}[The Kernel complex  $\ke \,\phi (\b)$]\label{ker}

\

 The simplicial sub-complex  of modules  $(\ke \,\phi)^*_J \subset IC^*\LL_{J,x} $
consist of fibers of analytic sheaves. It is defined  as  the subcomplex  generated as an $ (\Omega^*_{Y_J})_x$-algebra by the sub-modules
\begin{equation*}
 \oplus_ {|K|=k} \biggl(  \sum_{K_i \subset K} y_{K_i} A ((\widetilde {\ke \, N_{J  - (J \cap (K-K_i))} \cap  N_{K-K_i}L}) \frac{dy_K}{y_K}  \biggr) 
\end{equation*}
 $\ke \,\phi(\b) : = s((\ke \,\phi)^*_J )_{\emptyset \not= J \subset M}$
  is the sum. 
 \end{defn}
 We explain, how $\ke \, N_{J  - (J \cap (K-K_i))} \cap  N_{K-K_i}L$ appears in the description of the kernel of $\phi$.
Let $ K := \{i_1, \cdots, i_k \}$,  $Y_K := \cap_{i \in K}Y_i $  with generic point $y_K$,   and $dy_K/y_K :=  d y_{i_1}/y_{i_1} \wedge  \cdots \wedge d y_{i_k}/y_{i_k}  $. 
 For  $\alpha \in i_{Y_K}^*\Omega^p\LL$ and $ \beta \in \Omega^q\LL$, the pairing at a point  $x \in  Y_K$:
 
  $(\alpha, \beta):= \phi_{i_1, \ldots, i_k} (\alpha \wedge \beta)  \in H^k_{y_K}(\Omega^{p+q}_X) \subset  K^\b_X (\Omega_X^*)_x$ vanish unless $dy_K/y_K$ divides $\alpha \wedge \beta$ (as its value is considered modulo the ideal generated by $ \{y_{i_1}, \cdots, y_{i_k} \})$. 
  
  For $\beta =  \widetilde v dy_Q/y_Q$ for $Q \subset K$, let $P := K - Q$ and $\alpha  =  \widetilde v' dy_P/y_P$, the pairing  $(\alpha, \beta) = 0$ vanish
 if and only if  $S(v',v) = 0$. 
 
{\it  In particular  $(\alpha, \beta) = 0$ for all $v \in N_Q L$ if and only $v' \in \ke N_Q$.}

 Respectively, to describe the dual filtrations we remark that for $v \in W_{-r-1} L$,  $S(v',v) = 0$ vanish   if and only if  $v' \in (W_r)^* L$.

\subsubsection{The filtration $N!W$  dual to $N*W$}
  The dual filtrations may be constructed directly. Let $N$ be a nilpotent endomorphism respecting $(L, W)$ a filtered space with relative monodromy filtration $M$. We introduce the filtration  (\cite[3.4.2]{K})
  \begin{equation*}
 (N!W)_k := W_{k-1} + M_k(N,W) \cap N^{-1}W_{k-1}, 
 \end{equation*}
with  induced morphisms $I:  W_{k-1}  \to  (N!W)_k$ and  $N: (N!W)_k  \to  W_{k-1}$
 satisfying the relations  $N \circ I = N $ and $I \circ N = N$.  

Then $N!W$ is dual to the filtration $N*W$ on $L$   (formula \ref{*}),
with    induced morphisms $ W_{k+1} \xrightarrow{N} (N*W)_k \xrightarrow{I} W_{k+1}$ satisfying  $I \circ N = N$ and $I \circ N = N$. 

 Let  $(L^*, N^*)$ denote the dual of $(L,N)$ and  $W^*$  the dual filtration to $W$  defined  by: $W^*_k := Hom ( L/W_{-k-1}, \Q)$
 such that $Gr^{W^*}_k L^* \overset{\sim}{\longrightarrow}   (Gr^W_{-k} L)^*$. 
 There is  a duality relation between $(N^*!W^*) $ and  $(N*W)^*$, in the following sense:
\begin{lem}\label{d}  Let $W^*$ denote the filtration on the vector space
$L^*:= Hom(L, \Q)$ dual to a filtration $W$ on $L$; then for all
$a$,
\begin{equation*}
 (N^*!W^*)_a = (N*W)^*_a \subset L^*.
 \end{equation*}
\end{lem}
\begin{rem} We may develop this construction, in parallel to the constructions corresponding to $N*W$, and define the dual  weight  filtration  on  $\ke \,\phi (\b)$.
\end{rem}
\subsubsection{Simplicial coverings of  $Z \subset Y$.}

We realize the cone construction of the boundary of the tubular neighborhood of $Z$.

   To dualize the filtration $W$ on $IC^*\LL(Log Z) $,
we use  the strict  simplicial covering $ \pi_+: Z^+_\b \to X$   (resp. $ \pi: Z_\b \to Z$) of $X$,  defined by the NCD $Z$, with index subsets $A \subset I_Z$ including $Z_\emptyset:= X$ (resp. excluding  $Z_\emptyset$). It is naturally embedded in $Y^+_\b$ (resp. $Y_\b$). For any complex $\KK$ on $X$, the sum    $\pi_{+ *} \pi_+^* \KK := s(\pi_+^* \KK) \overset{\sim}{\longrightarrow}  j_{Z !} (\KK_{\vert X-Z})$. If $Z_\emptyset:= X$ is excluded, then    $ i_Z^*\KK \overset{\sim}{\longrightarrow}  \pi_* \pi^* \KK$.
\subsubsection{Simplicial complex}

  The Intersection complex $IC^*\LL$  lifts  to a simplicial complex $ IC^*\LL_{\vert Z^+_\b} := (IC^*\LL_{\vert Z_J})_{J\subset I_Z} $ on $Z^+_\b$. The simple complex 
$ IC_+\LL (Z,\b)$ is defined by summing $IC^* \LL_{\vert Z_J}$ over $Z_J$ for $J \subset I_Z$  including $J = \emptyset$.
We write $IC \LL (Z, \b)$ when $Z_\emptyset:= X$ is  excluded from the sum:
  \begin{equation*}
   \begin{split}
IC_+\LL (Z,\b):= & s(IC^*\LL_{\vert Z_J})_{J\subset I_Z} \overset{\sim}{\longrightarrow}   j_{Z !}(( j_{!*}\LL)_{\vert X-Z}), \\
 IC \LL (Z,\b):= & s(IC^* \LL_{\vert Z_J})_{J \not= \emptyset \subset I_Z}  \overset{\sim}{\longrightarrow}  i_Z^*( j_{!*}\LL).
 \end{split}
 \end{equation*}
 With the notations of section \ref{auto}
the restriction of $\phi$ (formula \ref{formula}) on components of  $Z_\b$ 
 \begin{equation*}
IC_+\LL (Z,\b)   \xrightarrow{\phi}   \HH om^\b_{\OO_X}  (IC^*\LL(Log Z), K_X^\b)  
\end{equation*}
induce the duality isomorphism
$ j_{Z !}(( j_{!*}\LL)_{\vert X-Z}) \overset{\sim}{\longrightarrow}   D Rj_{Z*} (( j_{!*}\LL)_{\vert X-Z})$.

The kernel of $\phi$ (prop. \ref{App}), is a simplicial sub-complex $(\ke \,\phi (\b)_J^*)_{J \subset I_Z, J \not= \emptyset} $. Its sum $\ke \,\phi (\b) (Z, \b):= s(\ke \,\phi (\b)_J^*)_{J \subset I_Z, J \not= \emptyset} $ is a sub-complex of 
$ IC \LL(Z,\b) $.

It is embedded with a shift in the degree,  as a subcomplex
$\ke \,\phi (\b) (Z, \b) [-1]\subset IC_+\LL (Z,\b)$ equal to $\ke \, \phi $ on $Z^+_\b$ .
There exist quasi-isomorphisms
  $$ \ke \,\phi (\b) (Z, \b)  \overset{\sim}{\longrightarrow}  IC \LL(Z,\b) \overset{\sim}{\longrightarrow}  i_Z^*  j_{! *}\LL .$$
 The morphism $\phi$ induces a duality isomorphism 
    \begin{equation*}
   IC_+\LL (Z,\b) [-1] \overset{\sim}{\longrightarrow}  D_X (IC^*\LL(Log Z) /IC^*\LL)
  \end{equation*} 

\subsubsection{ $i_Y^* IC^*\LL$ and $j_! \LL$}  
 
The  filtration $W$ on $ \ke \,\phi (\b) (Z, \b)$  is defined by duality.

   The complex $IC_+ \LL(Z,\b) \simeq j_{Z !} j^*_Z  \ilm$ has a structure of mixed Hodge complex with the following weight filtration:
\begin{equation}\label{0}
W_r IC_+ \LL(Z, \b) := W_r ( IC^* \LL)\oplus (W_{r+1} \ke \,\phi (\b) (Z, \b)) [-1]
\end{equation}
 It defines a MHS on  $\H^*(X, j_{Z !} j^*_Z  \ilm)$
dual to the MHS on $\H^*(X, Rj_{Z *} j^*_Z  \ilm)$.

 There exist filtered  duality isomorphisms 
 \begin{equation*}
 (i_Z^* j_{!*}\LL, W)  \overset{\sim}{\longrightarrow}    D_X(Ri_Z^! j_{!*}\LL, W) \overset{\sim}{\longrightarrow}  D_X(IC^*\LL(Log Z)/IC^*\LL)[-1], W).
  \end{equation*}
  induced by the morphism $\phi:
 \ke \,\phi (\b) (Z, \b) [-1] \overset{\sim}{\longrightarrow}  D_X ((IC^*\LL(Log Z) /IC^*\LL)$,
  
\n and a filtered  duality isomorphisms  $(j_{Z !}j_Z^*\ilm, W) \overset{\sim}{\longrightarrow}   D_X (IC^* \LL(Log Z),W) $.

\subsubsection{Cohomology with support $R i_Z^! \ilm[1]$}

  The complex $IC^*\LL(Log Z)$  lifts  to a simplicial complex on $Z_\b$ 
  
 $IC^*\LL(Log Z)_{\vert Z_J}:= i^*_{Z_J}IC^*\LL(Log Z)$. The simple  associated complex is the sum 
    \begin{equation*}
 IC^*\LL (Log Z)( \b):=  s(IC^*\LL_{\vert Z_J})_{J\subset I_Z, J \not= \emptyset} \overset{\sim}{\longrightarrow}   i_Z^*R j_{Z *}( j_Z^*  j_{!*}\LL).
 \end{equation*}
We have an embedding  $IC^*\LL (Z,\b) \subset IC^*\LL (Log Z)( \b)$ with quotient complex: 

$IC^*\LL (Log Z)( \b)/ IC^*\LL (Z, \b) \simeq i_Z^*(IC^*\LL (Log Z)/ IC^*\LL) \simeq R i_Z^! \ilm[1]$.

\subsubsection
{The complex $i_{Z * }i_Z^* Rj_{Z*} (( j_{!*}\LL)_{\vert X-Z})$ and the cohomology of the boundary of a tubular neighborhood of $Z$ }
 Let $B_Z$ denotes a small neighborhood of $Z$ in $X$
 $$\H^i( Z, i_Z^* Rj_{Z*} (( j_{!*}\LL)_{\vert X-Z})) \simeq  \H^i(B_Z-Z, \ilm) $$
 and we have an exact sequence 
\begin{equation*}
\cdots \to \H^i(X-Z, j_{! *}\LL) \to \H^i(B_Z-Z, \ilm) 
\to \H^{i+1}(X, j_! (j_{! *}\LL_{\vert X-Z})) \to \cdots
\end{equation*}
\begin{defn}[$i_{Z * }i_Z^* Rj_{Z*} (( j_{!*}\LL)_{\vert X-Z})$]\label{dual}
 Let  $p: j_{Z!} (j_{! *}\LL)_{\vert X-Z} =  IC_+\LL(Z, \b)  \to IC_\emptyset^*\LL = IC^*\LL$ denote the   projection  on the term with index $\emptyset $,  and $  i : IC^*\LL \to  IC^*\LL(Log Z) $   the natural embedding, the composition morphism $\widetilde I:= i\circ p $ defines a morphism 
\begin{equation*}
\widetilde I :  IC_+\LL(Z, \b) \to IC^*\LL(Log Z), \quad (\widetilde I:  j_{Z!} j_Z^*j_{! *}\LL \to Rj_{Z*}j_Z^* j_{! *}\LL)
\end{equation*}
The cone $C(\widetilde I)$  over $\widetilde I$ is isomorphic to $ i_{Z * }i_Z^* Rj_{Z*}(j_{! *}\LL)_{\vert X-Z}$.

The weight filtration $W$ is deduced  from the weights on $IC^*\LL( Log Z) $ 

\n and $ IC_+\LL(Z, \b)$: $W_r C(\widetilde I):= W_{r-1}(IC_+\LL(Z, \b)) [1] \oplus W_r(IC^*\LL(Log Z))$. 
\end{defn}
\begin{rem}
i) The cone depends only on the neighborhood of $Z$ and not on $X$ as we have a triangle   $  i_{Z * }   \ke \,\phi (\b) (Z, \b) \to C(\widetilde I) \to (IC^*\LL(Log Z)/IC^*\LL)$.

 ii) We have a duality isomorphism :
 
  $D(Gr^W_r C(\widetilde I,\LL)\simeq Gr^W_{-r} D C(\widetilde I,\LL)\simeq (Gr^W_{-r+1}  C(\widetilde I,D\LL))[-1]$,
since 

\n $D i_{Z * }i_Z^* Rj_{Z*} (( j_{!*}\LL)_{\vert X-Z})\simeq i_{Z * } Ri_Z^! Rj_{Z !} (( j_{!*} D\LL)_{\vert X-Z}) \simeq $
$i_{Z * }i_Z^* Rj_{Z*} (( j_{!*}D \LL)_{\vert X-Z})[1]$,
hence  $\H^j(Z,  C(\widetilde I))^* \simeq \H^{-j-1}(Z,  C(\widetilde I))$.

iii) If $\LL $ is a polarized VHS of weight $a$, we have 

$Gr^W_{a+i} C(\widetilde I) \simeq  Gr^W_{a+i} IC^*\LL(Log Z) $ for $i >0$,
and 
$Gr^W_{a+i} C(\widetilde I)  \simeq  Gr^W_{a+i} \ke \,\phi (\b) (Z, \b)$ for $i \leq 0$.
\end{rem}

\end{document}